\documentclass[11pt,a4paper,reqno]{amsart}
\usepackage[english]{babel}
\usepackage[T1]{fontenc}
\usepackage{verbatim}
\usepackage{palatino}
\usepackage{amsmath}
\usepackage{mathabx}
\usepackage{amssymb}
\usepackage{amsthm}
\usepackage{amsfonts}
\usepackage{graphicx}
\usepackage{esint}
\usepackage{color}
\usepackage{mathtools}
\usepackage{overpic}
\DeclarePairedDelimiter\ceil{\lceil}{\rceil}
\DeclarePairedDelimiter\floor{\lfloor}{\rfloor}
\usepackage[colorlinks = true, citecolor = black]{hyperref}
\author{Tuomas Orponen}
\title{On the discretised $ABC$ sum-product problem}
\address{Department of Mathematics and Statistics\\ University of Jyv\"askyl\"a,
P.O. Box 35 (MaD)\\
FI-40014 University of Jyv\"askyl\"a\\
Finland}
\email{tuomas.t.orponen@jyu.fi}
\date{\today}
\thanks{T.O. is supported by the Academy of Finland via the projects \emph{Quantitative rectifiability in Euclidean and non-Euclidean spaces} and \emph{Incidences on Fractals}, grant Nos. 309365, 314172, 321896.}

\newcommand{\R}{\mathbb{R}}

\newcommand{\N}{\mathbb{N}}

\newcommand{\Z}{\mathbb{Z}}

\newcommand{\spt}{\operatorname{spt}}
\newcommand{\Hd}{\dim_{\mathrm{H}}}

\newcommand{\spa}{\operatorname{span}}

\newcommand{\dist}{\operatorname{dist}}

\def\Barint_#1{\mathchoice
          {\mathop{\vrule width 6pt height 3 pt depth -2.5pt
                  \kern -8pt \intop}\nolimits_{#1}}%
          {\mathop{\vrule width 5pt height 3 pt depth -2.6pt
                  \kern -6pt \intop}\nolimits_{#1}}%
          {\mathop{\vrule width 5pt height 3 pt depth -2.6pt
                  \kern -6pt \intop}\nolimits_{#1}}%
          {\mathop{\vrule width 5pt height 3 pt depth -2.6pt
                  \kern -6pt \intop}\nolimits_{#1}}}

\numberwithin{equation}{section}

\theoremstyle{plain}
\newtheorem{thm}[equation]{Theorem}
\newtheorem*{"thm"}{"Theorem"}
\newtheorem{conjecture}[equation]{Conjecture}
\newtheorem{lemma}[equation]{Lemma}

\newtheorem{cor}[equation]{Corollary}
\newtheorem{proposition}[equation]{Proposition}

\newtheorem*{counter}{Counter assumption}

\theoremstyle{definition}

\newtheorem{definition}[equation]{Definition}

\theoremstyle{remark}
\newtheorem{remark}[equation]{Remark}

\newcommand{\nref}[1]{(\hyperref[#1]{#1})}

\DeclareMathSymbol{\intop}  {\mathop}{mathx}{"B3}

\begin{document} 

\begin{abstract} Let $0 < \beta \leq \alpha < 1$ and $\kappa > 0$. I prove that there exists $\eta > 0$ such that the following holds for every pair of Borel sets $A,B \subset \R$ with $\Hd A = \alpha$ and $\Hd B = \beta$:
\begin{displaymath} \Hd \{c \in \R : \Hd (A + cB) \leq \alpha + \eta\} \leq \tfrac{\alpha - \beta}{1 - \beta} + \kappa. \end{displaymath}
This extends a result of Bourgain from 2010, which contained the case $\alpha = \beta$. The paper also contains a $\delta$-discretised, and somewhat stronger, version of the estimate above, and new information on the size of long sums of the form $a_{1}B + \ldots + a_{n}B$. \end{abstract}

\maketitle

\tableofcontents

\section{Introduction}

Let $A,B,C \subset \R$ be large but finite sets. Is it true that there exists some $c \in C$ such that $|A + cB| \gg |A|$? Here $|\cdot|$ refers to cardinality. Not necessarily: consider for example
\begin{equation}\label{form87} A_{n} = \left\{\tfrac{1}{n^{1/2}},\tfrac{2}{n^{1/2}},\ldots,1 \right\} \quad \text{and} \quad B_{n} = \left\{\tfrac{1}{n^{1/4}},\tfrac{2}{n^{1/4}},\ldots,1\right\} = C_{n}. \end{equation}
It is not hard to check that for every $\epsilon > 0$, there exists $n \in \N$ such that $|A_{n} + B_{n}C_{n}| \leq n^{\epsilon}|A|$, so in particular $|A_{n} + cB_{n}| \leq n^{\epsilon}|A|$ for all $c \in C_{n}$. The problem can be fixed by adding one assumption: $|B||C| \gg |A|$. Then, a positive answer to the question follows easily from the Szemer\'edi-Trotter theorem \cite{MR729791} applied to the planar set $A \times B$. The requirement $|B||C| \gg |A|$ is also necessary, as one can see by variants of \eqref{form87}.

The $ABC$ sum-product problem, stated above, also makes sense in contexts where the Szemer\'edi-Trotter bound is not available, for example if $A,B,C \subset \mathbb{Z}_{p}$, and $p \in \N$ is prime. Again, it turns out that the lower bound $|B||C| \gg |A|$ yields the existence of $c \in C$ with $|A + cB| \gg |A|$. One way to show this is to adapt elementary techniques of Garaev \cite{MR2344270}, Glibichuk and Konyagin \cite{MR2359478}, and Bourgain \cite{MR2481734}. The details can be found in \cite{OV18}. Another way is to apply directly an incidence bound in finite fields due to Stevens and de Zeeuw \cite{MR3742451}. The theorem of Stevens and de Zeeuw gives a stronger lower bound for $|A + cB|$ than the elementary approach (see \cite[Proposition 1.3]{OV18} for the details), but ultimately relies on the polynomial method.  

The purpose of this paper is to consider the \emph{$\delta$-discretised $ABC$ sum-product problem} in $\R$, and lower bounds for $\Hd (A + cB)$, the Hausdorff dimension of $A + cB$. The $\delta$-discretised problem is otherwise the same as the question we started with, but instead of counting the cardinality $|A + cB|$, we seek lower bounds for the \emph{$\delta$-covering number} $|A + cB|_{\delta}$ for some small scale $\delta > 0$. We will also assume that the sets $A,B,C$ are $\delta$-separated, and have cardinalities $|A| = \delta^{-\alpha}$, $|B| = \delta^{-\beta}$, and $|C| = \delta^{-\gamma}$. In this variant of the problem, hypotheses on $|B||C|$ need to be coupled with additional non-concentration conditions to hope for positive results. The following theorem of Bourgain \cite{Bourgain10} from 2010 (extending his own work \cite{Bo1} from 2003) treats the case $A = B$:
\begin{thm}[Bourgain]\label{bourgain} Given $\alpha \in (0,1)$ and $\gamma,\kappa > 0$, there exist $\epsilon_{0},\epsilon > 0$ such that that the following holds for $\delta > 0$ sufficiently small.

Let $\nu$ be a probability measure on $[0,1]$ satisfying $\nu(B(x,r)) \leq r^{\gamma}$ for all $x \in \R$ and $0 < r \leq \delta^{\epsilon_{0}}$. Let additionally $A \subset [0,1]$ be a $\delta$-separated set with $|A| \geq \delta^{-\alpha}$, which also satisfies the non-concentration condition $|A \cap B(x,r)| \leq r^{\kappa}|A|$ for $x \in \R$ and $\delta \leq r \leq \delta^{\epsilon_{0}}$. 

Then, there exists a point $c \in \spt (\nu)$ such that
\begin{equation}\label{form85} |A + cA|_{\delta} \geq \delta^{-\alpha - \epsilon}. \end{equation}
\end{thm}

\begin{remark}\label{subsetRemark} Bourgain's theorem admits the following stronger version, which, to the best of my knowledge, was first stated and proved by He \cite[Theorem 1]{MR4148151} (see also \cite[(7.43), p. 221]{Bourgain10} for a slightly weaker result): under the assumptions of Theorem \ref{bourgain}, there exists a point $c \in \spt(\nu)$ such that $|\pi_{c}(G)|_{\delta} \geq \delta^{-\alpha - \epsilon}$ for all subsets $G \subset A \times A$ of cardinality $|G| \geq \delta^{\epsilon}|A|^{2}$. Here $\pi_{c}(x,y) = x + cy$. This version is useful for proving lower bounds for $\Hd (A + cA)$. Bourgain also proved such lower bounds in \cite[Theorem 4]{Bourgain10} without explicitly mentioning the stronger version of Theorem \ref{bourgain}: while his proof is correct, it requires some care from the reader to extract all the details. Applying the stronger version directly is simpler, see \cite[Theorem 2]{MR4148151}. \end{remark}

To see the connection between Theorem \ref{bourgain} and the $ABC$ problem, let $C \subset [0,1]$ be a $\delta$-separated set satisfying $|C \cap B(x,r)| \leq r^{\gamma}|C|$ for all $x \in \R$ and $\delta \leq r \leq \delta^{\epsilon_{0}}$. Then the uniformly distributed probability measure $\nu$ on the $\delta$-neighbourhood of $C$ satisfies $\nu(B(x,r)) \lesssim r^{\gamma}$, and it follows from \eqref{form85} that there exists $c \in C$ with $|A + cA|_{\delta} \geq \delta^{-\alpha - \epsilon}$.

Theorem \ref{bourgain} formally only treats the case $A = B$, but an inspection of its proof (or, more directly, an application of \cite[Theorem 3]{Bourgain10}), reveals that the result remains valid for two different $\delta$-separated sets $A,B \subset [0,1]$, provided that $|A| = |B|$, or at least $|B| \approx |A|$. The precise meaning of "$\approx$" is defined via the various constants appearing in \cite[Theorem 3]{Bourgain10}. To the best of my knowledge, Theorem \ref{bourgain} does not cover the case where $|A| = \delta^{-\alpha}$ and $|B| = \delta^{-\beta}$ with $\beta < \alpha$ (the case $\beta > \alpha$ is not relevant here: then $|A + cB|_{\delta} \gtrsim |B| \gg |A|$ for any $c \in \R$ with $|c| \sim 1$). 

The following conjecture would correspond to the assumption $|B||C| \gg |A|$ which suffices in the discrete variants (on $\R$ and $\Z_{p}$) of the $ABC$ sum-product problem:
\begin{conjecture}\label{mainConjecture} Let $\alpha,\beta,\gamma \in (0,1)$ with $\beta \leq \alpha$ and $\gamma > \alpha - \beta$. Assume that $A,B,C \subset [0,1]$ are $\delta$-separated sets with cardinalities $|A| \leq \delta^{-\alpha}$, $|B| = \delta^{-\beta}$, and $|C| = \delta^{-\gamma}$. Assume moreover that $|B \cap B(x,r)| \lesssim r^{\beta}|B|$ and $|C \cap B(x,r)| \lesssim r^{\gamma}|C|$ for all $x \in \R$ and $r > 0$. Then, there exists $\epsilon = \epsilon(\alpha,\beta,\gamma) > 0$ and a point $c \in C$ such that $|A + cB|_{\delta} \gtrsim_{\alpha,\beta,\gamma} \delta^{- \epsilon}|A|$. \end{conjecture}

The lower bound for $\gamma$ in Conjecture \ref{mainConjecture} is necessary, but the non-concentration assumptions on $B$ and $C$ are quite likely not sharp. The main result of this paper is the following partial result, where the lower bound $\gamma > \alpha - \beta$ is upgraded to $\gamma > (\alpha - \beta)/(1 - \beta)$:
\begin{thm}\label{main} Let $0 < \beta \leq \alpha < 1$ and $\kappa > 0$. Then, for every $\gamma \in ((\alpha - \beta)/(1 - \beta),1]$, there exist $\epsilon_{0},\epsilon,\delta_{0} \in (0,\tfrac{1}{2}]$, depending only on $\alpha,\beta,\gamma,\kappa$, such that the following holds. Let $\delta \in 2^{-\N}$ with $\delta \in (0,\delta_{0}]$, and let $A,B \subset (\delta \cdot \Z) \cap [0,1]$ satisfy the following hypotheses:
\begin{enumerate}
\item[(A)] \label{A} $|A| \leq \delta^{-\alpha}$.
\item[(B)] \label{B} $|B| \geq \delta^{-\beta}$, and $B$ satisfies the following Frostman condition: 
\begin{displaymath} |B \cap B(x,r)| \leq r^{\kappa}|B|, \qquad \delta \leq r \leq \delta^{\epsilon_{0}}. \end{displaymath} 
\end{enumerate}
Further, let $\nu$ be a Borel probability measure with $\spt (\nu) \subset [\tfrac{1}{2},1]$, and satisfying the Frostman condition $\nu(B(x,r)) \leq r^{\gamma}$ for $x \in \R$ and $0 < r \leq \delta^{\epsilon_{0}}$. Then, there exists a point $c \in \spt(\nu)$ such that the following holds: if $G \subset A \times B$ is any subset with $|G| \geq \delta^{\epsilon}|A||B|$, then 
\begin{displaymath} |\pi_{c}(G)|_{\delta} \geq \delta^{-\epsilon}|A|, \quad \text{where} \quad \pi_{c}(x,y) = x + cy. \end{displaymath}
\end{thm}
Theorem \ref{main} with $\alpha = \beta$ recovers Theorem \ref{bourgain}, and the stronger version in Remark \ref{subsetRemark}. In fact, Theorem \ref{main} is formally stronger than Theorem \ref{bourgain}, since Theorem \ref{main} does not impose any non-concentration conditions on $A$. This is useful in proving Corollary \ref{longSumCor} below. Theorem \ref{main} easily yields the following corollary for Hausdorff dimension:
\begin{cor}\label{hausdorffCor} Let $0 < \beta \leq \alpha < 1$ and $\kappa > 0$. Then, there exists $\eta = \eta(\alpha,\beta,\kappa) > 0$ such that if $A,B \subset \R$ are Borel sets with $\Hd A = \alpha$, $\Hd B = \beta$, then 
\begin{displaymath} \Hd \{c \in \R : \Hd (A + cB) \leq \alpha + \eta\} \leq \tfrac{\alpha - \beta}{1 - \beta} + \kappa. \end{displaymath}
\end{cor}
The case $\alpha = \beta$ is already contained in Bourgain's paper \cite{Bourgain10}. The reduction from Theorem \ref{main} to Theorem \ref{hausdorffCor} is a standard pigeonholing argument, and goes the same way as the proof of \cite[Theorem 2]{MR4148151}. For completeness, I give the details in Section \ref{appA}. A "continuous" version of Conjecture \ref{mainConjecture} would imply that the number $(\alpha - \beta)/(1 - \beta)$ in Corollary \ref{hausdorffCor} can be replaced by $\alpha - \beta$.

The lower bound on $|\pi_{c}(G)|_{\delta}$ in Theorem \ref{main} is indispensable for deducing Corollary \ref{hausdorffCor}, but makes Theorem \ref{main} difficult to prove with a direct assault. Instead, Theorem \ref{main} will be formally reduced to the following simpler version, which only treats $G = A \times B$:

\begin{thm}\label{mainTechnical} Let $0 < \beta \leq \alpha < 1$ and $\kappa > 0$. Then, for every $\gamma \in ((\alpha - \beta)/(1 - \beta),1]$, there exist $\epsilon,\epsilon_{0},\delta_{0} \in (0,\tfrac{1}{2}]$, depending only on $\alpha,\beta,\gamma,\kappa$, such that the following holds. Let $\delta \in 2^{-\N}$ with $\delta \in (0,\delta_{0}]$, and let $A,B \subset (\delta \cdot \Z) \cap [0,1]$ satisfy the following hypotheses:
\begin{enumerate}
\item[(A)] \label{A} $|A| \leq \delta^{-\alpha}$.
\item[(B)] \label{B} $|B| \geq \delta^{-\beta}$, and $B$ satisfies the following Frostman condition: 
\begin{displaymath} |B \cap B(x,r)| \leq r^{\kappa}|B|, \qquad \delta \leq r \leq \delta^{\epsilon_{0}}. \end{displaymath} 
\end{enumerate}
Further, let $\nu$ be a Borel probability measure with $\spt (\nu) \subset [0,1]$, satisfying the Frostman condition $\nu(B(x,r)) \leq r^{\gamma}$ for $x \in \R$ and $\delta \leq r \leq \delta^{\epsilon_{0}}$. Then, there exists $c \in \spt (\nu)$ such that 
\begin{displaymath} |A + cB|_{\delta} \geq \delta^{-\epsilon}|A|. \end{displaymath}
\end{thm}
Theorem \ref{mainTechnical} is the heart of the paper, but as far as I know, it is also news that Theorem \ref{main} can be literally reduced to Theorem \ref{mainTechnical}. This takes some work, but is mostly a matter of "standard techniques" in additive combinatorics. Since these details can be carried out without reference to the rest of the paper, they are postponed to Section \ref{appB}.

\begin{remark}\label{rem7} As written above, Theorem \ref{main} is deduced from Theorem \ref{mainTechnical} in Section \ref{appB}. A variant of this problem is the following. Assume that we want to prove Theorem \ref{main} with a fixed non-concentration exponent "$\kappa$". Can we deduce it from the version of Theorem \ref{mainTechnical} with the same $\kappa$? The answer is "almost": it turns out that in order to deduce Theorem \ref{main} for a fixed non-concentration exponent $\kappa > 0$, we only need to invoke Theorem \ref{main} with non-concentration exponent $\bar{\kappa} \in (0,\kappa)$ arbitrarily close to $\kappa$: however, the values of the constants $\epsilon,\delta_{0}$ produced by the argument will tend to $0$ as $\bar{\kappa} \nearrow \kappa$. The reductions in Section \ref{appB} will be written in such a way that this claim becomes apparent -- and the matter will be further refreshed in Remarks \ref{rem4}, \ref{rem5}, and \ref{rem6}. \end{remark}

\subsection{Related work} A relevant piece of recent literature is the paper of Guth, Katz, and Zahl \cite{MR4283564}, where the authors extend an argument (due to Garaev \cite{MR2344270}) from finite fields to give a new, relatively simple, proof of Bourgain's Theorem \ref{bourgain}. Given that the $\Z_{p}$ analogue of Conjecture \ref{mainConjecture} is known \cite{OV18}, it may be plausible that Conjecture \ref{mainConjecture} can be solved by extending the $\Z_{p}$ argument in the fashion of Guth, Katz, and Zahl. I was not able to carry this out, and here is why. The proof in \cite{OV18} is chiefly based on the following lemma: if $A,B \subset \mathbb{Z}_{p}$ are sets with $|A| = p^{\alpha}$ and $|B| = p^{\beta}$, then for every $\eta > 0$ there exists an integer $n = n(\alpha,\beta,\eta) \in \N$, and choices $a_{1},\ldots,a_{n} \in \pm A$ such that 
\begin{equation}\label{form93} |a_{1}B + \ldots + a_{n}B| \gtrsim p^{-\eta}\min\{|A||B|,p\}. \end{equation}
I was not able to extend the finite field techniques in \cite{OV18} to (directly) prove a $\delta$-discretised analogue of \eqref{form93}. However, once Theorem \ref{mainTechnical} is known, it can be applied to make partial progress towards a $\delta$-discretised analogue of \eqref{form93} (a sharper result would follow from Conjecture \ref{mainConjecture} in the same way):
\begin{cor}\label{longSumCor} Let $\beta,\gamma \in (0,1)$ and $0 < \eta < \gamma(1 - \beta)$. Then, there exists $\epsilon_{0},\delta_{0} > 0$ and $n \in \N$, depending on $\beta,\gamma,\eta$,  such that the following holds for all $\delta \in (0,\delta_{0}]$. Let $B,C \subset (\delta \cdot \Z) \cap [0,1]$ be non-empty sets satisfying 
\begin{equation}\label{form115} |B \cap B(x,r)| \leq r^{\beta}|B| \quad \text{and} \quad |C \cap B(x,r)| \leq r^{\gamma}|C| \end{equation}
for $x \in \R$ and $\delta \leq r \leq \delta^{\epsilon_{0}}$. Then, there exist points $c_{1},\ldots,c_{n} \in C$ such that
\begin{displaymath} |c_{1}B + \ldots + c_{n}B|_{\delta} \geq \delta^{-\gamma - \beta(1 - \gamma) + \eta} = \delta^{-\beta - \gamma(1 - \beta) + \eta}. \end{displaymath}
\end{cor}

\begin{remark} A classical projection theorem of Kaufman \cite{Ka} implies the existence of $c \in C$ such that $|B + cB|_{\delta} \gtrapprox \max\{\delta^{-\beta},\delta^{-\gamma}\}$. For $\gamma > \beta$, a recent sharpening of Kaufman's theorem by the author and Shmerkin \cite{2021arXiv210603338O} even yields $|B + cB|_{\delta} \geq \delta^{-\gamma - \eta}$ for some $\eta = \eta(\beta,\gamma) > 0$, and for $\delta > 0$ small enough (to be clear, this statement is only a corollary of the main result in \cite{2021arXiv210603338O}). In comparison, Corollary \ref{longSumCor} gives a far more substantial improvement, but at the cost of adding the number of summands. I give the simple proof straight away. \end{remark}

\begin{proof}[Proof of Corollary \ref{longSumCor}] Start by applying Theorem \ref{mainTechnical} with 
\begin{displaymath} \alpha := \beta + \gamma(1 - \beta) - \eta \in (\beta,1), \quad \beta, \quad \kappa := \beta, \quad \text{and} \quad \gamma. \end{displaymath}
Note that $\gamma > (\alpha - \beta)/(1 - \beta)$, so the parameters are admissible. Let $\epsilon,\epsilon_{0},\delta_{0} \in (0,\tfrac{1}{2}]$ be the constants given by Theorem \ref{mainTechnical} with $\alpha,\beta,\gamma,\kappa$. Let $\nu := |C|^{-1} \cdot \mathcal{H}^{0}|_{C}$ be the normalised counting measure on $C$, which satisfies the Frostman condition $\nu(B(x,r)) \leq r^{\gamma}$ for all $x \in \R$ and $\delta \leq r \leq \delta^{\epsilon_{0}}$ by \eqref{form115}. We also note that $|B| \geq \delta^{-\beta}$ by \eqref{form115} applied with $r = \delta$, and $B$ satisfies the $\kappa = \beta$-dimensional Frostman condition required in Theorem \ref{mainTechnical}.

We construct a sequence of sets $H_{n} \subset \delta \cdot \Z$, $n \in \N$, with the following greedy algorithm. We first define $H_{1} := (c_{1}B)_{\delta}$ arbitrarily, where $A_{\delta} := (\delta \cdot \Z) \cap A(\delta)$. Then, we assume that $H_{n}$ has already been defined for some $n \geq 1$, and we let
\begin{displaymath} H_{n + 1} := H_{n} + (c_{n + 1}B)_{\delta} \subset \delta \cdot \Z, \end{displaymath}
where $c_{n + 1} \in C$ maximises $|H_{n} + cB|_{\delta}$ among all choices $c \in C$. We observe (by induction) that $H_{n} \subset (\delta \cdot \Z) \cap [0,n]$, so $|H_{n}| \leq n\delta^{-1}$. For arbitrary $N \in \N$ with $N \geq 2$, it follows from the pigeonhole principle that there exists $n \in \{1,\ldots,N - 1\}$ such that
\begin{equation}\label{form122} |H_{n + 1}| \leq 2(N\delta^{-1})^{1/(N - 1)} |H_{n}| \leq 4\delta^{-1/(N - 1)}|H_{n}|. \end{equation}
Indeed, if the first inequality failed for every $n \in \{1,\ldots,N - 1\}$, then
\begin{displaymath} |H_{N}| > 2(N\delta^{-1})^{1/(N - 1)}|H_{N - 1}| > \ldots > 2^{N - 1}(N\delta^{-1})^{(N - 1)/(N - 1)}|H_{1}| \geq 2^{N - 1}N\delta^{-1}, \end{displaymath}
contradicting that $H_{N} \subset (\delta \cdot \Z) \cap [0,N]$. By definition of $H_{n + 1}$, \eqref{form122} implies
\begin{equation}\label{form116} |H_{n} + cB|_{\delta} \lesssim |H_{n + 1}| \leq 4\delta^{-1/(N - 1)}|H_{n}|, \qquad c \in C. \end{equation}
We now choose $N \in \N$ so large that $4\delta^{-1/(N - 1)} \leq \delta^{-\epsilon/2}$, where $\epsilon = \epsilon(\alpha,\beta,\gamma) > 0$ was one of the constants produced by Theorem \ref{mainTechnical}. Since $B$ and $\nu$ satisfy the hypotheses of Theorem \ref{mainTechnical}, we see from \eqref{form116} that $A := H_{n}$ must fail the hypotheses. However, the only hypotheses on $A$ in Theorem \ref{mainTechnical} are 
\begin{displaymath} A \subset (\delta \cdot \Z) \cap [0,1] \quad \text{and} \quad |A| \leq \delta^{-\alpha}. \end{displaymath}
Of course $H_{n} \not\subset [0,1]$, but this is not really relevant: we may find $k \in \{0,\ldots,n - 1\}$ such that $|H_{n} \cap [k,k + 1]| \geq \tfrac{1}{N}|H_{n}|$. Now, defining instead $A := (H_{n} \cap [k,k + 1]) - \{k\}$, we have $A \subset (\delta \cdot \Z) \cap [0,1]$, and $|A + cB|_{\delta} \leq \delta^{-\epsilon/2}|H_{n}| \leq \delta^{-\epsilon}|A|$ by \eqref{form116}, for $\delta > 0$ so small that $\delta^{-\epsilon/2} \geq N$. This violates Theorem \ref{mainTechnical}, unless
\begin{displaymath} |H_{n}| \geq |A| > \delta^{-\alpha} = \delta^{-\beta - \gamma(1 - \beta) + \eta}, \end{displaymath}
and this is what the corollary claimed.  \end{proof}

The $ABC$ sum-product problem is, of course, related to the highly active area of sum-product theory. The main open question is the Erd\H{o}s-Szemer\'edi sum-product conjecture \cite{MR820223}: if $A \subset \R$ or $A \subset \mathbb{Z}_{p}$ is a finite set ($p \in \N$ is prime), the E-S conjecture asks to prove that 
\begin{displaymath} \max\{|A + A|,|A \cdot A|\} \gtrsim_{\epsilon} |A|^{2 - \epsilon}, \qquad \epsilon > 0. \end{displaymath}
The research around this problem is too active to survey here: I only mention the papers \cite{MR4469270} of Rudnev-Stevens and \cite{MR4565644} of Mohammadi-Stevens for some current world records, and further references. For results on the the $\delta$-discretised variant of the Erd\H{o}s-Szemer\'edi problem, see \cite{MR4283564} by Guth-Katz-Zahl, and \cite{MR4452675} by D\k{a}browski, the author, and Villa.

Bourgain's $\delta$-discretised sum-product estimate, Theorem \ref{bourgain}, has been extended in various ways beyond the real line. For example, He \cite{MR4148151} found a version of the theorem in $\R^{n}$. Closely related are also the works \cite{MR2966656,MR2358056} by Bourgain-Gamburd, \cite{MR4041116} by He, \cite{MR4244524} by He-de Saxc\'e, \cite{MR3529116} by Benoist-de Saxc\'e, and \cite{MR4299175} by Li. These papers contain $\delta$-discretised sum-product or product theorems in various Lie groups. Viewing the $\delta$-discretised $ABC$ sum-product problem as a special case of a $\delta$-discretised incidence problem between points and $\delta$-tubes in $\R^{2}$, the papers \cite{MR4447307,GSW19} are also relevant.

Theorems \ref{bourgain} and \ref{mainTechnical} can be viewed as statements concerning linear projections of planar sets, as discussed more in the next subsection. Starting with this interpretation, one may ask if analogous statements hold for non-linear projections. Examples of particular interest are the \emph{pinned distance projections} $\bigtriangleup_{x}(y) = |x - y|$ and the \emph{radial projections} $\pi_{x}(y) = (x - y)/|x - y|$. Again, the literature is too broad for a survey, but see the recent papers \cite{Shmerkin20} by Shmerkin, \cite{2021arXiv211209044S} by Shmerkin-Wang, and \cite{2021arXiv210807311R} by Raz-Zahl for recent exciting developments and more references.

Finally, Conjecture \ref{mainConjecture} was recently solved by the author \cite{MR4388762} for Ahlfors-regular sets $A,B \subset [0,1]$. In fact, a much stronger result can be obtained for such sets. Let $\alpha,\beta \in (0,1)$. Assume that $A,B \subset \R$ are closed sets, $A$ is $\alpha$-Ahlfors-regular and $B$ is $\beta$-Ahlfors-regular. Then
\begin{displaymath} \Hd \{c \in \R : \dim (A + cB) < \alpha + \eta\} = 0 \end{displaymath} 
for $\eta := \beta(1 - \alpha)/(2 - \alpha) > 0$. (The paper \cite{MR4388762} also contains a $\delta$-discretised version.)

\subsection{Comparison to classical projection theorems} A popular topic in fractal geometry is to study the orthogonal projections of subsets of $\R^{d}$. In this section we will see what "classical" projection theorems in fractal geometry have to say about the size of $A + cB$.

For $e \in S^{1}$, let $\pi_{e} \colon \R^{2} \to \spa(e)$ be the orthogonal projection. A theorem of Kaufman \cite{Ka} from 1968, sharpening a seminal result of Marstrand \cite{Mar}, states the following: if $K \subset \R^{2}$ is a compact set with dimension $\Hd K = t$, then
\begin{equation}\label{form89} \Sigma(K,s) := \Hd \{e \in S^{1} : \Hd \pi_{e}(K) \leq s\} \leq s, \qquad 0 \leq s < t. \end{equation}
Another classical estimate, due Peres-Schlag \cite{MR1749437} but building on a Fourier-analytic technique introduced by Falconer \cite{MR673510}, shows that 
\begin{equation}\label{falconer} \Sigma(K,s) \leq \max\{1 + s - t,0\}, \qquad 0 \leq s \leq t. \end{equation} 
A folklore conjecture (made explicit in \cite{MR3162243}) proposes to improve \eqref{form89}-\eqref{falconer} to $\Sigma(K,s) \leq \max\{2s - t,0\}$ for $0 \leq s < t$. Bourgain \cite{Bourgain10} showed that $\Sigma(K,s) \to 0$ as $s \to t/2$, which supports the conjecture. A recent preprint \cite{2021arXiv210603338O} of the author and Shmerkin additionally shows that $\Sigma(K,s) \leq s - \epsilon$ for some $\epsilon = \epsilon(s,t) > 0$, for all $0 \leq s < t$.

The connection between orthogonal projections and the $A + cB$ problem is the following. Take $K = A \times B$, where $A,B \subset \R$. Then, for $e \in S^{1} \, \setminus \{(0,1),(0,-1)\}$, the projection $\pi_{e}(K)$ can, up to rescaling, be rewritten as $A + cB$, for a suitable $c = c(e) \in \R$. With this in mind, the bounds \eqref{form89}-\eqref{falconer} can be used to deduce the following. 

Let $0 < \beta \leq \alpha < 1$. Assume that $A,B \subset \R$ are Borel sets with $\Hd A = \alpha$ and $\Hd B = \beta$. Then, \eqref{form89}-\eqref{falconer} applied with $t := \Hd (A \times B) \geq \alpha + \beta$ yield
\begin{displaymath} \Hd \{c \in \R : \Hd (A + cB) \leq \alpha\} \leq \min\{\alpha,1 - \beta\}. \end{displaymath}
In contrast, letting $\eta \to 0$ in Corollary \ref{hausdorffCor} gives the upper bound $(\alpha - \beta)/(1 - \beta)$. This bound is $< \alpha$ for all $\alpha < 1$, and also $< 1 - \beta$ whenever $0 < \beta \leq \alpha < \tfrac{3}{4}$. If $\alpha > \tfrac{3}{4}$, then the "$1 - \beta$" estimate coming from \eqref{falconer} is better for some values of $\beta$, e.g. $\beta = \tfrac{1}{2}$.

The conjectured bound $\Sigma(K,s) \leq \max\{2s - t,0\}$ would imply the (Hausdorff dimension version of) Conjecture \ref{mainConjecture}:
\begin{displaymath} \Hd \{c \in \R : \Hd (A + cB) \leq \alpha\} = \Sigma(A \times B,\alpha) \leq \max\{2\alpha - t,0\} \leq \alpha - \beta. \end{displaymath}
To summarise, Corollary \ref{hausdorffCor} is stronger than all previous results in the case $K = A \times B$ and $s = \alpha = \Hd A < \tfrac{3}{4}$, whereas the conjecture $\Sigma(K,s) \leq \max\{2s - t,0\}$ is even stronger than (the Hausdorff dimension version of) Conjecture \ref{mainConjecture}. 

\subsection{Paper outline and proof sketch}\label{s:outline} The proof of Theorem \ref{main} has two distinct components: the first one is a reduction to Theorem \ref{mainTechnical4}, which differs from Theorem \ref{main} in the following aspects: (a) $\nu$ satisfies a Frostman condition on all scales $\delta \leq r \leq 1$, (b) the set $B$ has small doubling, that is $|B + B| \leq \delta^{-\epsilon}|B|$, and (c) the conclusion $|\pi_{c}(G)|_{\delta} \geq \delta^{-\epsilon}|A|$ is only required for $G = A \times B$. These reductions are performed in several steps:
\begin{align*} \text{Theorem \ref{mainTechnical4}} \quad & \stackrel{\S \ref{s:subsetReduction}}{\Longrightarrow} \quad \text{Theorem \ref{mainTechnical3}} \quad \stackrel{\S \ref{s:FrostmanReduction}}{\Longrightarrow} \quad \text{Theorem \ref{mainTechnical2}}\\
& \stackrel{\S \ref{s:doublingReduction}}{\Longrightarrow} \quad \text{Theorem \ref{mainTechnical}} \quad \,\,\, \stackrel{\S \ref{s:subsetReductionB}}{\Longrightarrow} \quad \text{Theorem \ref{mainSubset2}}\\
& \stackrel{\S \ref{s:toytoyReduction}}{\Longrightarrow} \quad \text{Theorem \ref{mainSubset1}} \quad \,\, \stackrel{\S \ref{s:mainProof}}{\Longrightarrow} \quad \text{Theorem \ref{main}}.  \end{align*} 
The outline of the paper is that the reduction from Theorem \ref{mainTechnical} to Theorem \ref{mainTechnical4} is performed first, then Theorem \ref{mainTechnical4} is proved with a direct argument, and finally Theorem \ref{main} is reduced to Theorem \ref{mainTechnical} in Section \ref{appB}. 

The additional assumptions (a)-(c) in Theorem \ref{mainTechnical4} are technically important. However, at the current level of discussion, all the theorems above are indistinguishable. So, for example, the reader may think that the following outline concerns the proof of Theorem \ref{mainTechnical}, which has the simplest statement. 

For the sake of exposition, I make the following additional assumptions on $A$ and $B$. Both sets have a "tree" (or "Cantor set") structure: for a suitable parameter $m \in \N$, each dyadic interval $I \in \mathcal{D}_{ms}$ intersecting $A$ contains exactly $R_{A}(s)$ sub-intervals in $\mathcal{D}_{m(s + 1)}$ which intersect $A$. The same is assumed of $B$. The numbers $R_{A}(s)$ and $R_{B}(s)$ are known as the \emph{branching numbers} of $A$ and $B$, respectively. Assume that the scale parameter $\delta > 0$ has the special form $\delta = 2^{-mN}$ for some $N \in \N$ (thus $R_{A}(s) = 1 = R_{B}(s)$ for $s \geq N$, since $A,B$ were assumed to be $\delta$-separated). We make even more assumptions:
\begin{itemize}
\item[(P1)] For every $s \in \N$, either $R_{B}(s) = 1$ or $R_{A}(s) = 2^{m}$.
\item[(P2)] $|B| = \delta^{-\beta}$, and for every $s \in \N$, either $R_{B}(s) = 1$ or $R_{B}(s) = 2^{m}$.
\end{itemize} 
Property (P2) needs the small doubling assumption $|B + B| \leq \delta^{-\epsilon}|B|$. Now, as we will see in a moment, the key question turns out to be: given a scale $s \in \N$ with $R_{B}(s) = 1$, what upper bound can we guarantee for $R_{A}(s)$? It turns out that we can easily use (P1)-(P2) to deduce an answer.

Assume that $R_{A}(s) \geq 2^{\Gamma m}$ for \textbf{all} $s \in \{0,\ldots,N - 1\} =: [N]$ with $R_{B}(s) = 1$. Write $\mathcal{N} := \{s \in [N] : R_{B}(s) = 1\}$, and note that $R_{A}(s) = 2^{m}$ for all $s \in [N] \, \setminus \, \mathcal{N}$ by assumption (P1). Now, we may calculate a lower bound on the cardinality of $A$ as follows:
\begin{equation}\label{form90} 2^{\alpha mN} \geq |A| = \prod_{s \in [N]} R_{A}(s) = \prod_{s \in [N] \, \setminus \, \mathcal{N}} R_{A}(s) \cdot \prod_{s \in \mathcal{N}} R_{A}(s) \geq 2^{m(N - |\mathcal{N}|)} \cdot 2^{\Gamma m|\mathcal{N}|}. \end{equation} 
On the other hand, by assumption (P2), we have
\begin{equation}\label{form92} 2^{\beta m N} = |B| = \prod_{s \in [N] \, \setminus \, \mathcal{N}} 2^{m} = 2^{m(N - |\mathcal{N}|)}, \end{equation}
so may solve $N - |\mathcal{N}| = \beta N$ and $|\mathcal{N}| = (1 - \beta)N$. Plugging this information into \eqref{form90} yields $\Gamma \leq (\alpha - \beta)/(1 - \beta)$. This is where the numerology in Theorem \ref{mainTechnical} comes from. Namely, the argument above shows that if $\Gamma > (\alpha - \beta)/(1 - \beta)$, then there exists at least one scale $s \in [N]$ such that $R_{A}(s) \leq 2^{\Gamma m}$. In fact, the same must be true for a positive fraction of the scales, say $\mathcal{G} \subset [N]$, where $|\mathcal{G}|/N$ only depends on $\Gamma - (\alpha - \beta)/(1 - \beta)$.

After this observation, we focus attention separately on pieces of $A \times B$ of the form $(A \cap I) \times (B \cap J)$, where $I,J \in \mathcal{D}_{ms}$ are intervals intersecting $A,B$, respectively, and $s \in \mathcal{G}$. By definition, $|A \cap I|_{m(s + 1)} = R_{A}(s) \leq 2^{\Gamma m}$ for some $\Gamma$ slightly larger than $(\alpha - \beta)/(1 - \beta)$. To be precise, we choose $(\alpha - \beta)/(1 - \beta) < \Gamma < \gamma$, where $\gamma$ is the Frostman exponent of the measure $\nu$ in Theorem \ref{mainTechnical}. If we \textbf{additionally} knew that $|B \cap J|_{2^{-m(s + 1)}} = R_{B}(s) \geq 2^{\epsilon m}$ for some $\epsilon > 0$, and the points in $B \cap J$ are well enough separated, we could at this point use an elementary argument (essentially the "potential theoretic method" due to Kaufman \cite{Ka}) to deduce that
\begin{equation}\label{form91} |(A \cap I) + c(B \cap J)|_{2^{-m(s + 1)}} \geq 2^{\epsilon m}|A \cap I|_{2^{-m(s + 1)}} \end{equation}
for a generic choice $c \in C = \spt (\nu)$. This argument is crucially based on $\gamma > \Gamma$, see Lemma \ref{lemma3} for the details. After this, summing up the increments \eqref{form91} for all $s \in \mathcal{G}$ would complete the proof of Theorem \ref{mainTechnical}. 

A major problem is that, as a matter of fact, $|B \cap J|_{2^{-m(s + 1)}} = R_{B}(s) = 1$ for all $s \in \mathcal{G}$. This follows from our assumption (P1), since $R_{A}(s) \leq 2^{\Gamma m} < 2^{m}$ for all $s \in \mathcal{G}$. To solve the problem, we follow Bourgain's proof of Theorem \ref{bourgain} rather faithfully: instead of considering individual scales $s \in [N]$, we recombine consecutive elements of $[N]$ into longer intervals $\mathcal{I} \subset [N]$ where the branching of $B$ is small but non-trivial, say $R_{B}(\mathcal{I}) = 2^{\epsilon m|\mathcal{I}|}$. Then, we carry out calculations similar to the ones we saw at \eqref{form90}-\eqref{form92} to make sure that also $R_{A}(\mathcal{I}) \leq 2^{\Gamma m |\mathcal{I}|}$ for a subset of these intervals $\mathcal{I}$ with substantial total length. At the end of the day, the intervals $\mathcal{I}$ with $R_{B}(\mathcal{I)} = 2^{\epsilon m|\mathcal{I}|}$ and $R_{A}(\mathcal{I}) \leq 2^{\Gamma m |\mathcal{I}|}$, will actually play the role we had written for the scales $\mathcal{G}$ in the discussion above.

There are numerous places in the arguments below where I either follow Bourgain's argument for the case $A = B$, or at least draw heavy inspiration from such an argument. Bourgain's influence on this paper will be treated as an absolute constant, and not spelled out every time separately.

\subsection{Acknowledgements} I would like to thank Pablo Shmerkin for clarifying a point about applying his inverse theorem \cite[Theorem 2.1]{Sh}, see Remark \ref{rem2}. I'm also grateful to the reviewer for reading the manuscript carefully and making many helpful suggestions.

\section{Notation and preliminaries}

\subsection{Dyadic cubes and covering numbers}\label{s:notation} Let $\mathcal{D}_{n}$ be the family of dyadic cubes in $\R^{d}$ with side-length $2^{-n}$. We will only use this notation for $n \geq 0$. For $Q = x + [0,2^{-n})^{d} \in \mathcal{D}_{n}$, we associate the affine map $T_{Q}(y) := 2^{n}(y - x)$, which rescales $Q$ to $[0,1)^{d}$. 

If $\mu$ is a Borel measure on $\R^{d}$, and $E \subset \R^{d}$ is a Borel set with $\mu(E) > 0$, we write $\mu_{E} := \mu(E)^{-1}\mu|_{E}$ for the \emph{renormalised restriction of $\mu$ to $E$}. This notation is most commonly used in the case $E = Q \in \mathcal{D}_{n}$. In this special case, we additionally define the notation
\begin{displaymath} \mu^{Q} := T_{Q}\mu_{Q}. \end{displaymath}
Here $f\nu(H) := \nu(f^{-1}H)$ refers, in general, to the push-forward of a measure $\nu$ under a map $f$. For a dyadic rational $r = 2^{-n}$, and a bounded set $A \subset \R^{d}$, we write $|A|_{r}$ for the least number of cubes in $\mathcal{D}_{n}$ required to cover $A$ (in the introduction, we used the same notation for the $r$-covering number, which is comparable up to a multiplicative constant). We will also write $A(r)$ for the open $r$-neighbourhood of $A$, and
\begin{displaymath} A_{r} = (r \cdot \Z) \cap A(r). \end{displaymath}
Finally, for $n \in \N$, $n \geq 1$, we abbreviate $[n] := \{0,\ldots,n - 1\}$. 

\subsection{Entropy} If $(\Omega,\mu)$ is a probability space, and $\mathcal{F}$ is a countable $\mu$-measurable partition of $\Omega$, we denote the \emph{$\mathcal{F}$-entropy of $\mu$} by
\begin{displaymath} H(\mu,\mathcal{F}) := \sum_{F \in \mathcal{F}} \mu(F) \log \tfrac{1}{\mu(F)}, \end{displaymath} 
with the convention $0 \cdot \log 0 = 0$. If $\mathcal{E},\mathcal{F}$ are two countable partitions, we denote the \emph{conditional $\mathcal{F}$-entropy of $\mu$ relative to $\mathcal{E}$} by
\begin{equation}\label{conditionalEntropy} H(\mu,\mathcal{F} \mid \mathcal{E}) := \sum_{E \in \mathcal{E}} \mu(E) H(\mu_{E},\mathcal{F}). \end{equation}
If $\mathcal{F}$ \emph{refines} $\mathcal{E}$ (each element of $\mathcal{E}$ can be written as a disjoint union of elements of $\mathcal{F}$), the conditional entropy can be alternatively written as
\begin{equation}\label{form83} H(\mu,\mathcal{F} \mid \mathcal{E}) = H(\mu,\mathcal{F}) - H(\mu,\mathcal{E}). \end{equation}
For a proof, see \cite[Proposition 3.3]{MR3590535}. In practice, we will only be concerned with $\mathcal{D}_{n}$-entropies of compactly supported Borel probability measures on $\R^{d}$, where $\mathcal{D}_{n}$ is the partition of $\R^{d}$ into dyadic cubes of side-length $2^{-n}$. In this special case $\mathcal{D}_{n + 1}$ always refines $\mathcal{D}_{n}$, so the formula \eqref{form83} is available. We record the following simple lemma, whose proof is a combination of \cite[Lemma 3.5]{MR3590535} and \cite[Remark 3.6]{MR3590535}:
\begin{lemma}\label{entropyLemma} Let $\mu$ be a Borel probability measure on $\R^{d}$, and let $\pi \colon \R^{d} \to \R^{D}$ be linear. Let $n \in \N$, and let $0 = n_{0} < n_{1} < \ldots < n_{h} = n$ be an arbitrary partition of $\{0,\ldots,n\}$. Then,
\begin{displaymath} H(\pi\mu,\mathcal{D}_{n}) \geq \sum_{j = 0}^{h - 1} \sum_{Q \in \mathcal{D}_{n_{j}}} \mu(Q) \cdot H(\pi \mu^{Q},\mathcal{D}_{n_{j + 1} - n_{j}} \mid \mathcal{D}_{0}). \end{displaymath} 
The inner summation only runs over those $Q \in \mathcal{D}_{n_{j}}$ with $\mu(Q) > 0$.
 \end{lemma}

A basic fact about entropy (which follows from Jensen's inequality) is that
\begin{displaymath} |\{F \in \mathcal{F} : \mu(F) > 0\}| \leq N \quad \Longrightarrow \quad H(\mu,\mathcal{F}) \leq \log N. \end{displaymath}
In particular, if $\pi \colon \R^{d} \to \R^{D}$ is $L$-Lipschitz in Lemma \ref{entropyLemma},  then
\begin{displaymath} H(\pi\mu^{Q},\mathcal{D}_{n_{j + 1} - n_{j}} \mid \mathcal{D}_{0}) =  H(\pi\mu^{Q},\mathcal{D}_{n_{j + 1} - n_{j}}) -  H(\pi\mu^{Q},\mathcal{D}_{0}) \geq H(\pi\mu^{Q},\mathcal{D}_{n_{j + 1} - n_{j}}) - CL, \end{displaymath} 
where $C \geq 1$ only depends on $d,D$. Therefore, the lower bound of Lemma \ref{entropyLemma} can be upgraded to
\begin{equation}\label{entropyIneq} H(\pi\mu,\mathcal{D}_{n}) \geq \left( \sum_{j = 0}^{h - 1} \sum_{Q \in \mathcal{D}_{n_{j}}} \mu(Q) \cdot H(\pi \mu^{Q},\mathcal{D}_{n_{j + 1} - n_{j}}) \right) - h \cdot CL. \end{equation}
We mention two further useful fact about entropy: first, if $\mathcal{E},\mathcal{F}$ are two countable $\mu$-measurable partitions such that 
\begin{displaymath} |\{F \in \mathcal{F} : F \cap E_{0}\}| \leq N \quad \text{and} \quad |\{E \in \mathcal{E} : E \cap F_{0}\}| \leq N \end{displaymath}
for all $E_{0} \in \mathcal{E}$ and $F_{0} \in \mathcal{F}$, then $|H(\mu,\mathcal{E}) - H(\mu,\mathcal{F})| \leq \log N$. Second, entropy (and also conditional entropy) is concave. We will use the convexity of entropy in the following form: if $\mu,\nu$ are two Borel probability measures on $\R^{d}$, then 
\begin{equation}\label{form84} H(\mu \ast \nu,\mathcal{D}_{n}) \geq \int H(\mu_{x},\mathcal{D}_{n}) \, d\nu(x), \end{equation} 
where $\mu_{x}$ is the probability measure defined by $\mu_{x}(H) = \mu(H - x)$. Since $(\mu \ast \nu)(H) = \int \mu_{x}(H) \, d\nu(x)$ for all Borel sets $H \subset \R^{d}$, one may view $\mu \ast \nu$ as a convex combination of the measures $\mu_{x}$. Formally, \eqref{form84} is deduced by applying Jensen's inequality to the concave function $f(r) = r\log(1/r)$ on $[0,1]$, and the random variable $X \colon x \mapsto \mu_{x}(Q)$ in the probability space $(\R^{d},\nu)$ (for fixed $Q \in \mathcal{D}_{n}$).

\section{Three initial reductions} 

This section contains a reduction of Theorem \ref{mainTechnical} to a special case, where we additionally assume that $|B + B| \leq \delta^{-\epsilon_{B}}|B|$, and $\nu(B(x,r)) \leq 40 \cdot r^{\gamma}$ for all $x \in \R$ and $r \geq \delta$ (see Theorem \ref{mainTechnical4}). It seems difficult to do achieve this reduction in a "single pass": instead, we add the extra assumptions in two separate steps (Sections \ref{s:doublingReduction} and \ref{s:FrostmanReduction}). After these steps, we arrive at Theorem \ref{mainTechnical3}, where the assumptions are present, but unfortunately the conclusion is also a little stronger. Then, the final reduction to Theorem \ref{mainTechnical4} "restores" the weaker conclusion, but maintains the additional assumptions. This is the version of Theorem \ref{mainTechnical} we will eventually be able to prove directly.

\subsection{Reduction to the case where $B$ has small doubling}\label{s:doublingReduction}

The purpose of this section is to reduce the proof of Theorem \ref{mainTechnical} to the following version, where the hypothesis $|B + B| \leq \delta^{-\epsilon_{B}}|B|$ has been added. This does not come for free: the price to pay is that the conclusion of Theorem \ref{mainTechnical2} is also a little stronger (that is, more difficult to prove).

\begin{thm}\label{mainTechnical2} Let $0 < \beta \leq \alpha < 1$ and $\kappa > 0$. Then, for every $\gamma \in ((\alpha - \beta)/(1 - \beta),1]$, there exist $\epsilon_{0},\epsilon,\epsilon_{B},\delta_{0},\rho \in (0,\tfrac{1}{2}]$, depending only on $\alpha,\beta,\gamma,\kappa$, such that the following holds. Let $\delta \in 2^{-\N}$ with $\delta \in (0,\delta_{0}]$, and let $A,B \subset (\delta \cdot \Z) \cap [0,1]$ satisfy the following hypotheses:
\begin{enumerate}
\item[(A)] \label{A} $|A| \leq \delta^{-\alpha}$.
\item[(B)] \label{B} $|B| \geq \delta^{-\beta}$, and $B$ satisfies the following Frostman condition: 
\begin{displaymath} |B \cap B(x,r)| \leq r^{\kappa}|B|, \qquad \delta \leq r \leq \delta^{\epsilon_{0}}. \end{displaymath} 
Assume moreover that $|B + B| \leq \delta^{-\epsilon_{B}}|B|$.
\end{enumerate}
Further, let $\nu$ be a Borel probability measure with $\spt (\nu) \subset [0,1]$, and satisfying the Frostman condition $\nu(B(x,r)) \leq r^{\gamma}$ for $x \in \R$ and $\delta \leq r \leq \delta^{\epsilon_{0}}$. Then, there exists a point $c \in \spt (\nu)$ such that $|A' + cB|_{\delta} \geq \delta^{-\epsilon}|A|$ for all $A' \subset A$ with $|A'| \geq (1 - \rho)|A|$. \end{thm}

\begin{remark} The following concerns Theorem \ref{mainTechnical2}, and also all other versions of Theorem \ref{main} or Theorem \ref{mainTechnical} in this paper: while the results claim the existence of $c \in \spt(\nu)$ (with certain properties), they can be formally upgraded to the existence of $c \in C$, where $C \subset \R$ is an arbitrary $\delta$-dense subset of $\spt(\nu)$. In particular, any subset $C \subset \R$ of full $\nu$ measure will work. The reason is that bounds for $|A + cB|_{\delta}$ (or $|A' + cB|_{\delta}$) are invariant, up to a change in constant factors, if the point $c \in \spt(\nu)$ is replaced another point $c' \in \R$ with $|c - c'| \leq \delta$. I leave further details to the reader. \end{remark}

The proof will use the following version of the Pl\"unnecke-Ruzsa inequality:
\begin{lemma}[Pl\"unnecke-Ruzsa inequality]\label{PRIneq} Let $\delta \in 2^{-\N}$, let $A,B_{1},\ldots,B_{n} \subset \R$ be arbitrary sets, and assume that $|A + B_{i}|_{\delta} \leq K_{i}|A|_{\delta}$ for all $1 \leq i \leq n$, and for some constants $K_{i} \geq 1$. Then, for every $\epsilon > 0$, there exists a subset $A' \subset A$ with $|A'|_{\delta} \geq (1 - \epsilon)|A|_{\delta}$ such that
\begin{displaymath} |A' + B_{1} + \ldots + B_{n}|_{\delta} \lesssim_{\epsilon,n} K_{1}\cdots K_{n} |A'|_{\delta}. \end{displaymath}
\end{lemma}
This form of the inequality is due to Ruzsa \cite{MR2314377}. For a more general result, see \cite[Theorem 1.5]{MR2484645}, by Gyarmati-Matolcsi-Ruzsa. To be accurate, these statements are not formulated in terms of $\delta$-covering numbers, but one may consult \cite[Corollary 3.4]{MR4283564} by Guth-Katz-Zahl to see how to handle the reduction.

\begin{remark} In a typical result in this paper, such as Theorem \ref{mainTechnical}, we are given a list of parameters $p_{1},\ldots,p_{m}$, and we are asked to find positive constants $\epsilon_{1},\ldots,\epsilon_{n}$ which only depend on $p_{1},\ldots,p_{m}$. Additionally, we are given a "known" theorem, such as Theorem \ref{mainTechnical2}, which outputs positive constants $\bar{\epsilon}_{1},\ldots,\bar{\epsilon}_{l}$ given a list of parameters $\bar{p}_{1},\ldots,\bar{p}_{k}$. To deduce the "unknown" theorem from the "known" one, the algorithm is always the same. First, fix the parameters $p_{1},\ldots,p_{m}$. Second, modify them suitably to produce new parameters $\bar{p}_{1},\ldots,\bar{p}_{k}$. Third, apply the "known" theorem with parameters $\bar{p}_{1},\ldots,\bar{p}_{k}$ to gain access to the constants $\bar{\epsilon}_{1},\ldots,\bar{\epsilon}_{l}$. Since the parameters $\bar{p}_{1},\ldots,\bar{p}_{k}$ were functions of $p_{1},\ldots,p_{m}$, so are the constants $\bar{\epsilon}_{1},\ldots,\bar{\epsilon}_{l}$. Therefore, it is legitimate to define the constants $\epsilon_{1},\ldots,\epsilon_{n}$, depending on all of the data $p_{1},\ldots,p_{m}$, $\bar{p}_{1},\ldots,\bar{p}_{k}$, and $\bar{\epsilon}_{1},\ldots,\bar{\epsilon}_{l}$. \end{remark}

\begin{proof}[Proof of Theorem \ref{mainTechnical} assuming Theorem \ref{mainTechnical2}] Let $\alpha,\beta,\gamma,\kappa$ be the constants given in Theorem \ref{mainTechnical}. Our task is to find $\epsilon,\epsilon_{0},\delta_{0} \in (0,\tfrac{1}{2}]$ depending on $\alpha,\beta,\gamma,\kappa$, such that the claims of Theorem \ref{mainTechnical} are satisfied. To do this, we fix some $\bar{\beta} < \beta$ slightly smaller than $\beta$ so that still
\begin{displaymath} \gamma > (\alpha - \bar{\beta})/(1 - \bar{\beta}). \end{displaymath}
Then, we apply Theorem \ref{mainTechnical2} with the parameters $\alpha,\bar{\beta},\gamma,\kappa/4$, and first extract the constants $\bar{\epsilon},\bar{\epsilon}_{0},\bar{\epsilon}_{B},\bar{\delta}_{0},\bar{\rho} > 0$, depending only on $\alpha,\bar{\beta},\gamma,\kappa/4$. Now, we claim that Theorem \ref{mainTechnical} holds with constants
\begin{equation}\label{form63} \epsilon_{0} := \tfrac{2}{\kappa n} \quad \text{and} \quad \epsilon := 2^{-n - 1}\bar{\epsilon}, \end{equation} 
where
\begin{equation}\label{defn} n := \max\{\ceil{2/(\kappa \bar{\epsilon}_{0})},\ceil{2/\bar{\epsilon}_{B}}\} \end{equation}
and any $\delta_{0} \in (0,\bar{\delta}_{0}]$ with the additional requirements
\begin{equation}\label{form62} 2^{n + 1} \leq \delta_{0}^{-\bar{\epsilon}_{B}/2} \quad \text{and} \quad \delta_{0}^{\bar{\beta} - \beta} \geq 2^{n}. \end{equation} 
The choice of the constants $\epsilon,\epsilon_{0}$ does not depend on the parameter $\bar{\rho} > 0$, but at the very end of the proof (see below \eqref{form123}), there will be an additional requirement for $\delta_{0} > 0$, which depends on $n,\bar{\rho}$; this is not spelled out explicitly, since the bounds depend on the implicit -- nonetheless effective -- constants in Lemma \ref{PRIneq}.

To prove Theorem \ref{mainTechnical}, fix $\delta \in (0,\delta_{0}]$, and assume that $A,B \subset (\delta \cdot \Z) \cap [0,1]$ and $\nu$ satisfy the assumptions of Theorem \ref{mainTechnical} with parameters $\alpha,\beta,\gamma,\kappa$ and $\epsilon_{0}$, as specified in \eqref{form63}. Thus $|A| \leq \delta^{-\alpha}$, and $|B| \geq \delta^{-\beta}$, and $|B \cap B(x,r)| \leq r^{\kappa}|B|$ for all $x \in \R$ and $\delta \leq r \leq \delta^{\epsilon_{0}}$. We claim that there exists $c \in \spt(\nu)$ such that $|A + cB|_{\delta} \geq \delta^{-\epsilon}|A|$.

Write $kB$ for the $k$-fold sum $B + \ldots + B$. Clearly $kB \subset (\delta \cdot \Z) \cap [0,k]$, so 
\begin{equation}\label{form54} |kB| \leq k \cdot \delta^{-1}, \qquad k \geq 1. \end{equation}
We claim that for any $n \geq 1$, there exists $1 \leq k \leq n$ such that
\begin{equation}\label{form55} |2^{k}B + 2^{k}B| = |2^{k + 1}B| \leq 2 \delta^{-1/n}|2^{k}B|. \end{equation}
Indeed, if this were not the case, then 
\begin{displaymath} 2^{n} \delta^{-1} \stackrel{\eqref{form54}}{\geq} |2^{n}B| > 2 \delta^{-1/n}|2^{n - 1}B| > 4 \delta^{-2/n}|2^{n - 2}B| > \ldots > 2^{n}\delta^{-n/n} |B| \geq 2^{n}\delta^{-1}, \end{displaymath}
a contradiction. We then apply this observation with $n \in \N$ as in \eqref{defn}, and we pick $1 \leq k \leq n$ such that \eqref{form55} holds. Write $\bar{B} := 2^{k}B$ for this choice of "$k$", so 
\begin{equation}\label{smallDoubling} |\bar{B} + \bar{B}| \leq 2\delta^{-1/n}|\bar{B}|. \end{equation}
Recalling that $1/n \leq \bar{\epsilon}_{B}/2$, this looks promising for the purpose of applying Theorem \ref{mainTechnical2} to the pair of sets $A$ and $\bar{B}$. But does $\bar{B}$ satisfy a Frostman condition? It turns out that it does, with constants "$\kappa/4$" and "$\bar{\epsilon}_{0}$". The following argument is copied from \cite[Section 8.2]{Bourgain10}. Assume, to reach a contradiction, that there exists some dyadic scale 
\begin{equation}\label{form64} r \in [\delta,\delta^{\bar{\epsilon}_{0}}] \stackrel{\eqref{defn}}{\subset} [\delta,\delta^{2/(\kappa n)}] \stackrel{\eqref{form63}}{=} [\delta,\delta^{\epsilon_{0}}], \end{equation}
and a point $a \in \R$, such that $|\bar{B} \cap B(a,r)|_{\delta} \geq C_{0}r^{\kappa/2}|\bar{B}|$ for a suitable absolute constant $C_{0} \geq 1$. Now, since $r \leq \delta^{\epsilon_{0}}$, we know by hypothesis that $B$ satisfies $|B \cap B(x,r)| \leq r^{\kappa}$ for all $x \in \R$. Consequently $|\bar{B}|_{r} \geq |B|_{r} \geq r^{-\kappa}$, and
\begin{align*} |\bar{B} + \bar{B}|_{\delta} & \gtrsim |\bar{B}|_{r} \cdot |\bar{B} \cap B(a,r)|_{\delta}\\
& \geq C_{0} |B|_{r} \cdot r^{\kappa/2}|\bar{B}|\\
& \geq C_{0}r^{-\kappa} \cdot r^{\kappa/2} |\bar{B}|\\
& \geq C_{0}\delta^{-1/n}|\bar{B}|, \end{align*}  
using $r \leq \delta^{\bar{\epsilon}_{0}} \leq \delta^{2/(\kappa n)}$ in the final inequality. If $C_{0} \geq 1$ was chosen large enough, this lower bound contradicts the small doubling property \eqref{smallDoubling}. We conclude that
\begin{equation}\label{form59} |\bar{B} \cap B(x,r)| \leq C_{0}r^{\kappa/2}|\bar{B}|, \qquad \delta \leq r \leq \delta^{\bar{\epsilon}_{0}}. \end{equation}
There are also a few smaller issues before we can apply Theorem \ref{mainTechnical2} to $\bar{B}$: evidently $|\bar{B}| \geq |B| \geq \delta^{-\beta}$, but unfortunately  $\bar{B} \subset [0,2^{n}]$ instead of $\bar{B} \subset [0,1]$. Regardless, there exists some an interval $I_{0} = [m,m + 1] \subset [0,2^{n}]$ such that $|\bar{B} \cap I_{0}| \geq |\bar{B}|/2^{n}$. We define 
\begin{displaymath} \bar{B}_{m} := (\bar{B} \cap I_{0}) - \{m\} \subset (\delta \cdot \Z) \cap [0,1], \end{displaymath}
so $|\bar{B}_{m}| \geq |\bar{B}|/2^{n}$. Then,
\begin{displaymath} |\bar{B}_{m} + \bar{B}_{m}| \stackrel{\eqref{smallDoubling}}{\leq} 2\delta^{1/n}|\bar{B}| \stackrel{\eqref{defn}-\eqref{form62}}{\leq} \delta^{-\bar{\epsilon}_{B}} |\bar{B}_{m}| \quad \text{and} \quad |\bar{B}_{m}| \geq 2^{-n}\delta^{-\beta} \stackrel{\eqref{form62}}{\geq} \delta^{-\bar{\beta}}.  \end{displaymath}
Finally,
\begin{displaymath} |\bar{B}_{m} \cap B(x,r)| \stackrel{\eqref{form59}}{\leq} C_{0}2^{n}r^{\kappa/2}|\bar{B}_{m}| \stackrel{\eqref{form62}}{\leq} r^{\kappa/4}|\bar{B}_{m}|, \quad r \in [\delta,\delta^{\bar{\epsilon}_{0}}]. \end{displaymath}
Now we have shown that the triple $A,\bar{B}_{m},\nu$ satisfies all the hypotheses of Theorem \ref{mainTechnical2} with parameters $\alpha,\bar{\beta},\gamma,\kappa/4$, and $\bar{\epsilon}_{0},\bar{\epsilon}_{B},\bar{\delta}_{0}$.  It follows that there exists $c \in \spt (\nu)$ with the property that if $A' \subset A$ is any subset with $|A'| \geq (1 - \bar{\rho})|A|$, then 
\begin{equation}\label{form72} |A' + c\bar{B}_{m}|_{\delta} \geq \delta^{-\bar{\epsilon}}|A|. \end{equation}
We now claim that $|A + cB|_{\delta} \geq \delta^{-\epsilon}|A|$ for this specific $c \in \spt (\nu)$, which will complete the proof of Theorem \ref{mainTechnical}. If this fails, then by the Pl\"unnecke-Ruzsa inequality, Lemma \ref{PRIneq}, we find a subset $A' \subset A$ of cardinality $|A'| \geq (1 - \bar{\rho})|A|$ such that
\begin{equation}\label{form123} |A' + c\bar{B}_{m}|_{\delta} \lesssim |A' + c(2^{k}B)|_{\delta} \lesssim_{n,\bar{\rho}} \delta^{-2^{n}\epsilon}|A| \stackrel{\eqref{form63}}{=} \delta^{-\bar{\epsilon}/2}|A|. \end{equation}
This contradicts \eqref{form72} for $\delta > 0$ small enough, depending on $n,\bar{\rho}$. This contradiction completes the proof of Theorem \ref{mainTechnical}. \end{proof}

\subsection{Reducing the Frostman constant of $\nu$}\label{s:FrostmanReduction} Let $\nu$ be the measure appearing in the statement of Theorem \ref{mainTechnical} or \ref{mainTechnical2}. We assumed that $\nu(B(x,r)) \leq r^{\gamma}$ for all scales $\delta \leq r \leq \delta^{\epsilon_{0}}$. We will need, in fact, is that $\nu$ satisfies the Frostman condition $\nu(B(x,r)) \leq Cr^{\gamma}$ for all $r \geq \delta$, and with an absolute constant $C \geq 1$. It turns out that this can be achieved, eventually with $C = 40$. In this section, we reduce the proof of Theorem \ref{mainTechnical2} to the following:
\begin{thm}\label{mainTechnical3} Let $0 < \beta \leq \alpha < 1$ and $\kappa > 0$. Then, for every $\gamma \in ((\alpha - \beta)/(1 - \beta),1]$, there exist $\epsilon,\epsilon_{0},\epsilon_{B},\delta_{0},\rho \in (0,\tfrac{1}{2}]$, depending only on $\alpha,\beta,\gamma,\kappa$, such that the following holds. Let $\delta \in 2^{-\N}$ with $\delta \in (0,\delta_{0}]$, and let $A,B \subset (\delta \cdot \Z) \cap [0,1]$ satisfy the following hypotheses:
\begin{enumerate}
\item[(A)] \label{A} $|A| \leq \delta^{-\alpha}$.
\item[(B)] \label{B} $|B| \geq \delta^{-\beta}$, and $B$ satisfies the following Frostman condition: 
\begin{displaymath} |B \cap B(x,r)| \leq r^{\kappa}|B|, \qquad \delta \leq r \leq \delta^{\epsilon_{0}}. \end{displaymath}
Assume moreover that $|B + B| \leq \delta^{-\epsilon_{B}}|B|$. 
\end{enumerate}
Further, let $\nu$ be a Borel probability measure with $\spt (\nu) \subset [-1,1]$ which satisfies the Frostman condition $\nu(B(x,r)) \leq 20r^{\gamma}$ for $x \in \R$ and $r \geq \delta$. Then, there exists a point $c \in \spt (\nu)$ such that
\begin{displaymath} |A' + cB|_{\delta} \geq \delta^{-\epsilon}|A'| \end{displaymath} 
for all subsets $A' \subset A$ with $|A'| \geq (1 - \rho)|A|$.
\end{thm}

The proof will require the Pl\"unnecke-Ruzsa inequality, that is Lemma \ref{PRIneq} in the previous section, and also the following \cite[Exercise 6.5.12]{MR2289012} in the book of Tao and Vu:

\begin{lemma}\label{TVLemma} Let $A,B \subset \delta \cdot \Z$, and assume that $|A + B| \leq K|A|$ for some $K \geq 1$. Then, for every $N \geq 1$ and $\rho > 0$, there exists a subset $A' \subset A$ with $|A'| \geq (1 - \rho)|A|$ with the property $|A' - B| \lesssim_{\rho,n} K^{2^{N}/N}|A|^{1 + 1/N}$.
\end{lemma}

The exercise is only stated with constant $\rho = \tfrac{1}{2}$, but if one reads the subsequent hint about how to solve the exercise, it is clear (based on \cite[Exercise 6.5.1]{MR2289012}) that any $\rho > 0$ will work, at the cost of making the implicit constant larger. We are then prepared to reduce Theorem \ref{mainTechnical2} to Theorem \ref{mainTechnical3}.

\begin{proof}[Proof of Theorem \ref{mainTechnical2} assuming Theorem \ref{mainTechnical3}] The argument roughly follows \cite[Section 5]{Bourgain10} in Bourgain's paper. Fix the constants $\alpha,\beta,\gamma,\kappa$ from the statement of Theorem \ref{mainTechnical2}. Our task is to find the constants $\epsilon_{0},\epsilon,\epsilon_{B},\delta_{0},\rho \in (0,\tfrac{1}{2}]$, depending only on $\alpha,\beta,\gamma,\kappa$. This will be accomplished by applying Theorem \ref{mainTechnical3} with constants $\alpha,\beta,\gamma,\kappa/2$. Recall from the statement of Theorem \ref{mainTechnical3} that there exist constants $\bar{\epsilon}_{0},\bar{\epsilon},\bar{\epsilon}_{B},\bar{\delta}_{0},\bar{\rho} > 0$, which only depend on the constants $\alpha,\beta,\gamma,\kappa/2$. We begin by choosing $\epsilon > 0$ so small that
\begin{equation}\label{form96} \frac{C}{\log_{2}(1/\epsilon)} \leq \bar{\epsilon} \quad \Longleftrightarrow \quad \epsilon \leq 2^{-C/\bar{\epsilon}} \end{equation}
for a suitable absolute constant $C > 0$ to be determined later. We now pick the other constants $\epsilon_{0},\epsilon_{B} \in (0,\tfrac{1}{2}]$ so that
\begin{equation}\label{form80} \epsilon_{0}  + \epsilon_{B} := \min\{ \kappa\bar{\epsilon}_{0}/8,\bar{\epsilon}_{B}/2\}. \end{equation}
We choose $\delta_{0} \leq \bar{\delta}_{0}$, and additionally $\delta_{0}$ needs to satisfy a few other restrictions, which we explain on the spot. We finally choose $\rho > 0$ so small that
\begin{equation}\label{form117} (1 - \rho)(1 - \rho - \sqrt{\rho}) \geq (1 - \bar{\rho}). \end{equation}
With these choices of constants, fix $\delta \in (0,\delta_{0}]$, and let $A,B \subset (\delta \cdot \Z) \cap [0,1]$ be sets, and let $\nu$ be a Borel probability measure on $[0,1]$, satisfying the hypotheses in Theorem \ref{mainTechnical2}. To land in a situation where Theorem \ref{mainTechnical3} becomes applicable, we consider initially the measure $\bar{\nu} := \nu \ast (-\nu)$, where $-\nu(A) := \nu(-A)$. Evidently $\spt (\bar{\nu}) \subset [-1,1]$. As Bourgain shows in \cite[(5.5)]{Bourgain10}, the measure $\bar{\nu}$ has the property
\begin{equation}\label{form2} \bar{\nu}(B(x,r)) \leq 4 \cdot \bar{\nu}(B(0,r)) \leq 4 \cdot \sup_{y \in \R} \nu(B(y,r)), \qquad x \in \R, \, r > 0. \end{equation}
Now, let $c_{0} \geq 0$ be the infimum of the numbers such that
\begin{equation}\label{form3} \bar{\nu}(B(0,c_{0})) > 5 \cdot c_{0}^{\gamma}, \end{equation} 
if any such numbers exist. Evidently $c_{0} \in [0,1]$, since $\bar{\nu}$ is a probability measure on $B(0,1)$. If no $c_{0}$ as in \eqref{form3} exists, then let $c_{0} := \max \{|c| : c \in \spt (\bar{\nu})\}$, and note that 
\begin{equation}\label{form4} 5 \cdot c_{0}^{\gamma} \geq \bar{\nu}(B(0,c_{0})) \geq \bar{\nu}(B(0,1)) = 1 \quad \Longrightarrow \quad c_{0} \geq 5^{-1/\gamma} \geq \delta^{\epsilon_{0}},  \end{equation}
assuming here that $\delta_{0} \geq \delta$ is sufficiently small in terms of $\gamma,\epsilon_{0}$. Assume then that $c_{0}$, as in \eqref{form3}, exists. Since $\sup_{y \in \R} \nu(B(y,r)) \leq r^{\gamma}$ for all $0 < r \leq \delta^{\epsilon_{0}}$ by assumption, \eqref{form2} implies that $c_{0} \geq \delta^{\epsilon_{0}}$. In both cases, $c_{0} \geq \delta^{\epsilon_{0}}$. Moreover, we note that $\spt (\nu) \cap \{c_{0},-c_{0}\} \neq \emptyset$ in both cases (in the non-trivial case, otherwise some smaller value of $c_{0}$ would also satisfy \eqref{form3}). 

We then consider the re-normalised measure $\bar{\nu}^{c_{0}}$ defined by
\begin{displaymath} \bar{\nu}^{c_{0}}(H) := \tfrac{1}{\bar{\nu}(B(0,c_{0}))} \cdot \bar{\nu}|_{B(0,c_{0})}(c_{0} \cdot H), \qquad H \subset \R, \end{displaymath}
which satisfies $\spt (\bar{\nu}^{c_{0}}) = c_{0}^{-1} \cdot (\spt \bar{\nu} \cap \bar{B}(0,c_{0})) \subset [-1,1]$. Clearly $\bar{\nu}^{c_{0}}$ is a Borel probability measure. Moreover, if $x \in \R$ and $r \in [\delta,1]$, then, assuming that $c_{0} \in [0,1]$ was defined via \eqref{form3}, we have
\begin{displaymath} \bar{\nu}^{c_{0}}(B(x,r)) \leq \frac{\bar{\nu}(B(c_{0}x,c_{0}r))}{\bar{\nu}(B(0,c_{0}))} \stackrel{\eqref{form2}}{\leq} 4 \cdot \frac{\bar{\nu}(B(0,c_{0}r))}{5 \cdot c_{0}^{\gamma}} \leq 4 \cdot \frac{5 \cdot (c_{0}r)^{\gamma}}{5 \cdot c_{0}^{\gamma}} = 4 \cdot r^{\gamma}. \end{displaymath} 
If $c_{0}$ was, instead, defined as $c_{0} = \max\{|c| : c \in \spt (\bar{\nu})\}$, then $\bar{\nu}(B(0,c_{0})) = 1$, so 
\begin{displaymath} \bar{\nu}^{c_{0}}(B(x,r)) \leq \frac{\bar{\nu}(B(c_{0}x,c_{0}r))}{\bar{\nu}(B(0,c_{0}))} \stackrel{\eqref{form2}}{\leq} 4 \cdot \bar{\nu}(B(0,c_{0}r)) \leq 20 \cdot (c_{0}r)^{\gamma} \leq 20 \cdot r^{\gamma}. \end{displaymath} 
The same estimates are also true for $r > 1$, since $\|\bar{\nu}^{c_{0}}\| = 1$. Therefore, in any case $\bar{\nu}^{c_{0}}$ satisfies the hypotheses of Theorem \ref{mainTechnical3} with Frostman constant $20$. 

We will not apply Theorem \ref{mainTechnical3} directly to the sets $A,B$, but rather to $A,(c_{0}B)_{\delta}$, where 
\begin{displaymath} (c_{0}B)_{\delta} = (\delta \cdot \Z) \cap (c_{0}B)(\delta) \subset (\delta \cdot \Z) \cap [0,1]. \end{displaymath}
Evidently $|(c_{0}B)_{\delta}| \gtrsim c_{0}|B| \geq \delta^{\epsilon_{0}}|B|$ by \eqref{form4}. It follows that
\begin{displaymath} |(c_{0}B)_{\delta} + (c_{0}B)_{\delta}| \lesssim |B + B| \leq \delta^{-\epsilon_{B}}|B| \lesssim \delta^{-\epsilon_{0} - \epsilon_{B}}|(c_{0}B)_{\delta}|.  \end{displaymath} 
Since $\epsilon_{0} + \epsilon_{B} \leq \bar{\epsilon}_{B}/2$ by \eqref{form80}, and if $\delta > 0$ is sufficiently small, we conclude that $(c_{0}B)_{\delta}$ satisfies the small doubling assumption in Theorem \ref{mainTechnical3} with constant $\bar{\epsilon}_{B}$. We moreover claim that $(c_{0}B)_{\delta}$ satisfies the Frostman condition $|(c_{0}B)_{\delta} \cap B(x,r)| \leq r^{\kappa/2}|(c_{0}B)_{\delta}|$ for all $\delta \leq r \leq \delta^{\bar{\epsilon}_{0}}$. To see this, fix $\delta \leq r \leq \delta^{\bar{\epsilon}_{0}} \leq \delta^{2\epsilon_{0}} \leq c_{0}\delta^{\epsilon_{0}}$ (by \eqref{form80} and \eqref{form4}), and note that
\begin{align*} |(c_{0}B)_{\delta} \cap B(x,r)| & \lesssim |B \cap B(x,c_{0}^{-1}r)|\\
& \leq (c_{0}^{-1}r)^{\kappa} \cdot |B|\\
& \lesssim c_{0}^{-2} \cdot r^{\kappa} \cdot |(c_{0}B)_{\delta}|\\
& \leq \delta^{-2\epsilon_{0}} \cdot r^{\kappa/2} \cdot r^{\kappa/2} \cdot |(c_{0}B)_{\delta}|\\
& \stackrel{\eqref{form80}}{\leq} \delta^{2\epsilon_{0}} \cdot r^{\kappa/2} \cdot |(c_{0}B)_{\delta}|. \end{align*} 
This implies $|(c_{0}B)_{\delta} \cap B(x,r)| \leq r^{\kappa/2} |(c_{0}B)_{\delta}|$, provided that $\delta_{0} \geq \delta$ is sufficiently small. We have now shown that Theorem \ref{mainTechnical3} is applicable with the parameters $\alpha,\beta,\gamma,\kappa/2$ to the the sets $A,(c_{0}B)_{\delta}$, and the measure $\bar{\nu}^{c_{0}}$. 

Since $\delta \leq \delta_{0} \leq \bar{\delta}_{0}$, Theorem \ref{mainTechnical3} implies the existence of a point $c \in \spt (\bar{\nu}^{c_{0}}) \subset [-1,1] \cap c_{0}^{-1} \cdot (\spt (\nu) - \spt (\nu))$ such that
\begin{equation}\label{form73} |A' + c(c_{0}B)_{\delta}| \geq \delta^{-\bar{\epsilon}}|A| \end{equation}
for all subsets $A' \subset A$ with $|A'| \geq (1 - \bar{\rho})|A|$. Note that the point $c \in \spt (\bar{\nu}^{c_{0}})$ in \eqref{form73} can be written as $c = c_{0}^{-1} \cdot (c_{1} - c_{2})$ for certain points $c_{1},c_{2} \in \spt (\nu)$. Therefore
\begin{equation}\label{form70} |A' + (c_{1} - c_{2}) B|_{\delta} = |A' + \tfrac{c_{1} - c_{2}}{c_{0}} \cdot c_{0}B|_{\delta} \gtrsim |A' + c (c_{0}B)_{\delta}|_{\delta} \geq \delta^{-\bar{\epsilon}}|A| \end{equation}
for all $A' \subset A$ with $|A'| \geq (1 - \bar{\rho})|A|$. We now claim that there exists $\bar{c} \in \{c_{1},c_{2}\}$ such that
\begin{equation}\label{form118} |A' + \bar{c}B|_{\delta} \geq \delta^{-\epsilon}|A|, \qquad A' \subset A, \, |A'| \geq (1 - \rho)|A|, \end{equation} 
assuming that $\delta_{0} \geq \delta$ is small enough, depending on $\epsilon,\rho$. This will prove Theorem \ref{mainTechnical2}. 

Assume that \eqref{form118} fails for both $\bar{c} \in \{c_{1},c_{2}\}$, and let $A_{1}',A_{2}' \subset A$ be subsets of cardinalities $|A_{j}'| \geq (1 - \rho)$, $j \in \{1,2\}$, such that
\begin{equation}\label{form71} |A_{1}' + c_{1}B| < \delta^{-\epsilon}|A| \quad \text{and} \quad |A_{2}' + c_{2}B| < \delta^{-\epsilon}|A|. \end{equation}
We first observe from the second inequality in \eqref{form71} that
\begin{displaymath} |A_{2}' + c_{2}B|_{\delta} \leq \delta^{-\epsilon}|A| \leq 2\delta^{-\epsilon}|A_{2}'|. \end{displaymath}
By Lemma \ref{TVLemma}, for $N \geq 1$ there exists a subset $A_{2}'' \subset A_{2}'$ of cardinality $|A_{2}''| \geq (1 - \rho)|A_{2}'|$ such that
\begin{equation}\label{form74} |A_{2}'' - c_{2}B|_{\delta} \lesssim_{\rho,N} (\delta^{-\epsilon})^{2^{N}}|A_{2}'|^{1 + 1/N} \lesssim \delta^{-\epsilon \cdot 2^{N + 2} - 1/N}|A_{2}''|. \end{equation}
We apply this with $N \sim \log_{2} (1/\epsilon)$ satisfying $\epsilon \cdot 2^{N + 2} \sim \sqrt{\epsilon}$. Since with this choice $\epsilon \cdot 2^{N + 2}  \sim \sqrt{\epsilon} \ll 1/\log_{2}(1/\epsilon) \sim 1/N$, we have $\epsilon \cdot 2^{N + 1} + 1/N \leq 2/N \sim 1/\log_{2}(1/\epsilon)$, and we deduce from \eqref{form74} that
\begin{displaymath} |A_{2}'' - c_{2}B|_{\delta} \lesssim_{\epsilon,\rho} \delta^{-C_{0}/\log_{2}(1/\epsilon)}|A_{2}''| \end{displaymath}
for some absolute constant $C_{0} > 0$.

Now, recall that $|A_{1}'| \geq (1 - \rho)|A|$ and $|A_{2}''| \geq (1 - \rho)|A_{2}'| \geq (1 - \sqrt{\rho})|A|$. Consequently, the intersection $A' := A_{1}' \cap A_{2}''$ satisfies $|A'| \geq (1 - \rho - \sqrt{\rho})|A|$. Evidently, 
\begin{displaymath} |A' + c_{1}B| \leq \delta^{-\epsilon}|A| \lesssim \delta^{-C_{0}/\log_{2}(1/\epsilon)}|A'| \quad \text{and} \quad |A' - c_{2}B| \lesssim_{\rho,\epsilon} \delta^{-C_{0}/\log_{2}(1/\epsilon)}|A'|. \end{displaymath} 
By Lemma \ref{PRIneq}, there exists a further subset $A'' \subset A'$ with 
\begin{displaymath} |A''| \geq (1 - \rho)|A'| \geq (1 - \rho)(1 - \rho - \sqrt{\rho})|A| \stackrel{\eqref{form117}}{\geq} (1 - \bar{\rho})|A| \end{displaymath}
such that
\begin{displaymath} |A'' + c_{1}B - c_{2}B|_{\delta} \lesssim_{\epsilon,\rho} \delta^{-2C_{0}/\log_{2}(1/\epsilon)}|A''| \stackrel{\eqref{form96}}{\leq} \delta^{-\bar{\epsilon}/2}|A|. \end{displaymath}
This contradicts \eqref{form70} for $\delta > 0$ small enough, depending on $\epsilon,\rho$, and proves \eqref{form118}. The proof of Theorem \ref{mainTechnical2} is complete. \end{proof}

\subsection{Removing reference to subsets}\label{s:subsetReduction} In the previous reductions, we have upgraded the assumptions of Theorem \ref{mainTechnical} in two ways: we have arranged the set $B$ to have small doubling, and the Frostman constant of $\nu$ to be $20$. However, there has been a price: whereas Theorem \ref{mainTechnical} only claims that $|A + cB|_{\delta} \geq \delta^{-\epsilon}|A|$ for some $c \in \spt(\nu)$, Theorem \ref{mainTechnical3} claims the existence of $c \in \spt(\nu)$ such that $|A' + cB|_{\delta} \geq \delta^{-\epsilon}|A'|$ for all $A' \subset A$ with $|A'| \geq (1 - \rho)|A|$. It turns out that this innocent-looking difference makes Theorem \ref{mainTechnical3} difficult to prove with a direct assault. Therefore, we need a final reduction to the following statement:

\begin{thm}\label{mainTechnical4} Let $0 < \beta \leq \alpha < 1$ and $\kappa > 0$. Then, for every $\gamma \in ((\alpha - \beta)/(1 - \beta),1]$, there exist $\epsilon,\epsilon_{0},\epsilon_{B},\delta_{0} \in (0,\tfrac{1}{2}]$, depending only on $\alpha,\beta,\gamma,\kappa$, such that the following holds. Let $\delta \in 2^{-\N}$ with $\delta \in (0,\delta_{0}]$, and let $A,B \subset (\delta \cdot \Z) \cap [0,1]$ satisfy the following hypotheses:
\begin{enumerate}
\item[(A)] \label{A} $|A| \leq \delta^{-\alpha}$.
\item[(B)] \label{B} $|B| \geq \delta^{-\beta}$, and $B$ satisfies the following Frostman condition: 
\begin{displaymath} |B \cap B(x,r)| \leq r^{\kappa}|B|, \qquad \delta \leq r \leq \delta^{\epsilon_{0}}. \end{displaymath}
Assume moreover that $|B + B| \leq \delta^{-\epsilon_{B}}|B|$. 
\end{enumerate}
Further, let $\nu$ be a Borel probability measure with $\spt (\nu) \subset [-1,1]$ satisfying the Frostman condition $\nu(B(x,r)) \leq 40 \cdot r^{\gamma}$ for $x \in \R$ and $r \geq \delta$. Then, there exists a point $c \in \spt (\nu)$ such that $|A + cB|_{\delta} \geq \delta^{-\epsilon}|A|$. \end{thm}

Theorem \ref{mainTechnical4} only differs from Theorem \ref{mainTechnical3} in its (superficially) weaker conclusion, and in that the Frostman constant of $\nu$ has increased from $20$ to $40$.

\begin{proof}[Proof of Theorem \ref{mainTechnical3} assuming Theorem \ref{mainTechnical4}] Fix the parameters $0 < \beta \leq \alpha < 1$, $\kappa$, and $\gamma > (\alpha - \beta)/(1 - \beta)$ from Theorem \ref{mainTechnical3}. As usual, our task is to find the parameters $\epsilon,\epsilon_{0},\epsilon_{B},\delta_{0},\rho$ such that Theorem \ref{mainTechnical3} is satisfied. In doing so, we apply Theorem \ref{mainTechnical4} to the parameters $0 < \beta \leq \bar{\alpha} < 1$ and $\kappa$, where $\bar{\alpha} \in (\alpha,1)$ is arbitrary such that the key inequality $\gamma > (\bar{\alpha} - \beta)/(1 - \beta)$ still holds. Then, we let 
\begin{equation}\label{a21} \bar{\epsilon},\bar{\epsilon}_{0},\bar{\epsilon}_{B},\bar{\delta}_{0} > 0 \end{equation}
be the constants given by Theorem \ref{mainTechnical4} with parameters $\bar{\alpha},\beta,\kappa,\gamma$. We now begin defining the parameters $\epsilon,\epsilon_{0},\epsilon_{B},\delta_{0},\rho$. We set
\begin{equation}\label{a22} \epsilon_{0} := \bar{\epsilon}_{0} \quad \text{and} \quad \epsilon_{B} = \bar{\epsilon}_{B}. \end{equation}
We will need that $\delta_{0} \leq \bar{\delta}_{0}$, and there will be an additional (simple) dependences on the allowed parameters, which will be explained when they arise. To define the parameters $\epsilon,\rho$, fix a natural number $N \sim 1/\bar{\epsilon}$, so that the following holds:
\begin{equation}\label{a23} (N - 1)^{-1} < \bar{\epsilon}/2. \end{equation}
Then, let
\begin{equation}\label{b9} \epsilon := \frac{\bar{\alpha} - \alpha}{2^{N + 1}}. \end{equation}
Finally, define $\rho > 0$, depending only on $\bar{\epsilon}$, so small that 
\begin{equation}\label{form120}  \left(2(1 - (1 - \rho)^{N}) \right)^{1/2^{N}} \leq \tfrac{1}{2}. \end{equation}
This is possible, since the inequality is clearly true for $\rho = 0$. 

We now make the counter assumption that Theorem \ref{mainTechnical3} fails for certain $\delta \in (0,\delta_{0}]$, $A,B \subset (\delta \cdot \Z) \cap [0,1]$, and a Borel probability measure $\nu$ on $[-1,1]$, satisfying the hypotheses of Theorem \ref{mainTechnical3} with parameters $\alpha,\beta,\kappa,\gamma$, and the constants $\epsilon_{0},\epsilon,\delta_{0}$ described above. This means that for every $c \in C = \spt(\nu)$, there exists a subset $A_{c} \subset A$ with the properties
\begin{equation}\label{b5} |A_{c}| \geq (1 - \rho)|A| \quad \text{and} \quad |A_{c} + cB|_{\delta} \leq \delta^{-\epsilon}|A|. \end{equation}
The plan is to use this information to construct a new set $\bar{A} \subset (\delta \cdot \Z) \cap [0,1]$, and a new probability measure $\bar{\nu}$ on $[-1,1]$, such that the triple $\bar{A},B,\bar{\nu}$ satisfies the hypotheses of Theorem \ref{mainTechnical4} with parameters $\bar{\alpha},\beta,\kappa,\gamma$ and constants $\bar{\epsilon}_{0},\bar{\epsilon}_{B}$, but nevertheless $|\bar{A} + cB| < \delta^{-\bar{\epsilon}}|\bar{A}|$ for all $c \in \spt(\bar{\nu})$. This contradiction will complete the proof of Theorem \ref{mainTechnical3}. 

Given such a set $A_{c} \subset A$ for every $c \in C$, we observe that
\begin{equation}\label{form82} \int \ldots \int |A_{c_{1}} \cap \ldots \cap A_{c_{N}}| \, d\nu(c_{1})\cdots d\nu(c_{N}) \geq (1 - \rho)^{N}|A| \end{equation}
by H\"older's inequality. Consider the set
\begin{displaymath} \Omega := \{(c_{1},\ldots,c_{N}) \in C^{N} : |A_{c_{1}} \cap \ldots \cap A_{c_{N}}| \geq \tfrac{1}{2}|A|\}. \end{displaymath}
If "$I$" temporarily stands for the integral in \eqref{form82}, we have
\begin{displaymath} (1 - \rho)^{N}|A| \leq I  \leq \nu^{N}(\Omega^{c}) \cdot \tfrac{1}{2}|A| + (1 - \nu^{N}(\Omega^{c})) \cdot |A|, \end{displaymath}
which can be rearranged to $\nu^{N}(\Omega^{c}) \leq 2(1 - (1 - \rho)^{N})$. Consequently
\begin{equation}\label{form119} \nu^{N}(\Omega) \geq 1 - 2(1 - (1 - \rho)^{N}) =: 1 - \theta_{0}. \end{equation}
For $c_{1},\ldots,c_{n} \in C$ fixed, we define
\begin{displaymath} \Omega_{c_{1}\cdots c_{n}} := \{(c_{n + 1},\ldots,c_{N}) \in C^{N - n} : (c_{1},\ldots,c_{N}) \in \Omega\}. \end{displaymath}
It follows from Fubini's theorem that
\begin{equation}\label{b6} \nu^{N - n}(\Omega_{c_{1}\cdots c_{n}}) = \int \nu^{N - n - 1}(\Omega_{c_{1}\cdots c_{n}c}) \, d\nu(c) \end{equation}
for all $c_{1},\ldots,c_{n} \in C$, and $1 \leq n \leq N - 2$. The same remains true for $n = 0$, if the left hand side is interpreted as $\nu^{N}(\Omega)$. Equation \eqref{b6} also remains valid for $n = N - 1$ if we define the notation $\nu^{N - n - 1} = \nu^{0}$ as follows:
\begin{equation}\label{b7} \nu^{0}(\Omega_{c_{1}\cdots c_{N - 1}c}) := \mathbf{1}_{\Omega}(c_{1},\ldots,c_{N - 1},c). \end{equation}
We will use this notation in the sequel. 

For $(c_{1},\ldots,c_{N}) \in \Omega$ fixed, we write
\begin{equation}\label{form121} A_{c_{1}\cdots c_{N}} := A_{c_{1}} \cap \ldots \cap A_{c_{N}} \quad \Longrightarrow \quad |A_{c_{1}\cdots c_{N}}| \geq \tfrac{1}{2}|A|. \end{equation}
We now construct a sequence of sets $H_{n} \subset \delta \cdot \Z$, $1 \leq n \leq N$. At the same time, we will construct subsets $C_{1},\ldots,C_{N} \subset C$, and points $c_{n} \in C_{n}$, $1 \leq n \leq N$, with the properties  
\begin{equation}\label{b3} \nu^{N - n}(\Omega_{c_{1}\cdots c_{n}}) \geq 1 - \theta_{n} \quad \text{and} \quad \nu(C_{n}) \geq 1 - \theta_{n}, \quad 1 \leq n \leq N, \end{equation}
where we define inductively
\begin{displaymath} \theta_{n} := \sqrt{\theta_{n - 1}} \geq \theta_{n - 1}. \end{displaymath}
 In particular, the first part of \eqref{b3} with $n = N$ shows that $(c_{1},\ldots,c_{N}) \in \Omega$, recall the notation \eqref{b7}. As a second remark, recalling the definition of $\theta_{0} = 2(1 - (1 - \rho)^{N})$, and combining this with the definition of $\rho$ in \eqref{form120}, one sees that $\theta_{n} \leq \tfrac{1}{2}$ for all $1 \leq n \leq N$. To begin with, we define
\begin{displaymath} C_{1} := \{c \in C : \nu^{N - 1}(\Omega_{c}) \geq 1 - \theta_{1}\}, \end{displaymath}
and we choose an arbitrary element $c_{1} \in C_{1}$. Since
\begin{displaymath} 1 - \theta_{0} \leq \nu^{N}(\Omega) = \int \nu^{N - 1}(\Omega_{c}) \, d\nu(c) \leq \nu(C_{1}^{c}) \cdot (1 - \theta_{1}) + (1 - \nu(C_{1}^{c})) = -\theta_{1} \cdot \nu(C_{1}^{c}) + 1 \end{displaymath}
by \eqref{form119}, we observe that $\nu(C_{1}^{c}) \leq \theta_{0}/\theta_{1} = \theta_{1}$, and consequently $\nu(C_{1}) \geq 1 - \theta_{1}$. In particular $C_{1} \neq \emptyset$. We then define
\begin{displaymath} H_{1} := (c_{1}B)_{\delta}. \end{displaymath}
Assume inductively that $H_{1},\ldots,H_{n}$ and $C_{1},\ldots,C_{n} \subset C$, and $c_{j} \in C_{j}$, $1 \leq j \leq n \leq N - 1$, have already been constructed, and satisfy \eqref{b3}. We pick an element $c_{n + 1} \in C_{n + 1}$, where
\begin{displaymath} C_{n + 1} := \{c \in C : \nu^{N - n - 1}(\Omega_{c_{1}\cdots c_{n} c}) \geq 1 - \theta_{n + 1}\}, \quad 1 \leq n \leq N - 1. \end{displaymath}
For $n = N - 1$, the notation $\nu^{N - n - 1}(\Omega_{c_{1}\cdots c_{n}c})$ should be interpreted as in \eqref{b7}, so 
\begin{displaymath} C_{N} = \{c \in C : \mathbf{1}_{\Omega}(c_{1},\ldots,c_{N - 1},c) \geq 1 - \theta_{N}\} = \{c \in C : (c_{1},\ldots,c_{N - 1},c) \in \Omega\}. \end{displaymath}
For an arbitrary choice $c_{n + 1} \in C_{n + 1}$, we note that the first part of \eqref{b3} is satisfied with index "$n + 1$", by the definition of $C_{n + 1}$. 

The set $C_{n + 1}$ also satisfies the second part of \eqref{b3} with index "$n + 1$", since
\begin{displaymath} 1 - \theta_{n} \stackrel{\eqref{b3}}{\leq} \nu^{N - n}(\Omega_{c_{1}\cdots c_{n}}) \stackrel{\eqref{b6}}{=} \int \nu^{N - n - 1}(\Omega_{c_{1}\cdots c_{n}c}) \, d\nu(c) \leq -\theta_{n + 1} \cdot \nu(C_{n + 1}^{c}) + 1, \end{displaymath}
and consequently $\nu(C_{n + 1}^{c}) \leq \theta_{n}/\theta_{n + 1} = \theta_{n + 1}$, and $\nu(C_{n + 1}) \geq 1 - \theta_{n + 1}$.

Whereas $c_{1} \in C_{1}$ was chosen arbitrarily, the element $c_{n + 1} \in C_{n + 1}$ is chosen in such a way that the quantity $|H_{n} + c_{n + 1}B|_{\delta}$ is maximised, among all possible choices $c_{n + 1} \in C_{n + 1}$. We then define
\begin{displaymath} H_{n + 1} := H_{n} + (c_{n + 1}B)_{\delta}. \end{displaymath}
Continuing in this manner produces a distinguished sequence $(c_{1},\ldots,c_{N}) \in \Omega$, which we fix for the remainder of the argument, and a sequence of sets $H_{1},\ldots,H_{N}$.

Note that $H_{n} \subset (\delta \cdot \Z) \cap [0,N]$ for all $1 \leq n \leq N$ by a straightforward induction, so $|H_{n}| \leq 2N\delta^{-1}$. Therefore, by the pigeonhole principle, there exists an index $n \in \{1,\ldots,N - 1\}$ such that
\begin{equation}\label{a4} |H_{n + 1}| \leq (2N\delta^{-1})^{1/(N - 1)}|H_{n}| \leq 4\delta^{-1/(N - 1)}|H_{n}|. \end{equation}
For this particular index $n \in \{1,\ldots,N - 1\}$, we then have $|H_{n} + cB|_{\delta} \lesssim |H_{n + 1}| \leq 4\delta^{-1/(N - 1)}|H_{n}|$ for all $c \in C_{n + 1}$ by the definition of $H_{n + 1}$, and therefore
\begin{equation}\label{a5} |H_{n} + cB|_{\delta} \leq \delta^{-\bar{\epsilon}/2}|H_{n}|, \qquad c \in C_{n + 1}, \end{equation}
recalling \eqref{a23}, and assuming that $\delta > 0$ is small enough.

We now claim that \eqref{a5} violates Theorem \ref{mainTechnical4} with parameters $\bar{\alpha},\beta,\kappa,\gamma$, and with the objects
\begin{equation}\label{b10} \bar{A} := H_{n}, \quad B, \quad \text{and} \quad \bar{\nu} := \nu(C_{n + 1})^{-1} \cdot \nu|_{C_{n + 1}}. \end{equation}
We need to check the following items to contradict Theorem \ref{mainTechnical4}:
\begin{itemize}
\item[(a)] $|\bar{A}| \leq \delta^{-\bar{\alpha}}$,
\item[(b)] $|B| \geq \delta^{-\beta}$ and $|B + B| \leq \delta^{-\bar{\epsilon}_{B}}|B|$, and $B$ satisfies a Frostman condition with exponents $\kappa$ and $\bar{\epsilon}_{0}$, 
\item[(c)] $\bar{\nu}$ satisfies a Frostman condition with exponent $\gamma$ and constant $40$.
\end{itemize}
Point (b) is true by assumption (and since we chose $\epsilon_{B} = \bar{\epsilon}_{B}$ and $\epsilon_{0} = \bar{\epsilon}_{0}$ in \eqref{a22}), so only (a) and (c) need to be verified. We first use the Pl\"unnecke-Ruzsa inequality to establish (a), assuming that $\delta > 0$ is sufficiently small in terms of $N,\bar{\alpha}$. Clearly $\bar{A}$ can be written as a sum of $n \leq N$ sets of the form $(c_{m}B)_{\delta}$, for some $1 \leq m \leq n$, where $c_{m}$ is an index in the (fixed) sequence $(c_{1},\ldots,c_{N}) \in \Omega$. Noting that $A_{c_{1}\cdots c_{N}} \subset A_{c_{m}} \subset A$, each of these sets individually satisfies
\begin{displaymath} |A_{c_{1}\cdots c_{N}} + (c_{m}B)_{\delta}| \lesssim |A_{c_{m}} + c_{m}B|_{\delta} \stackrel{\eqref{b5}}{\leq} \delta^{-\epsilon}|A| \stackrel{\eqref{form121}}{\leq} 2\delta^{-\epsilon}|A_{c_{1}\cdots c_{N}}|. \end{displaymath}
We may therefore infer that
\begin{displaymath} |\bar{A}| \lesssim_{N,\rho} \delta^{-2^{N}\epsilon}|A| \leq \delta^{-2^{N}\epsilon - \alpha}. \end{displaymath}
from Lemma \ref{PRIneq}. This inequality implies $|\bar{A}| \leq \delta^{-\bar{\alpha}}$ for small enough $\delta > 0$, recalling our choice of $\epsilon$ at \eqref{b9}. 

We move to (c). Recalling \eqref{b10}, and from \eqref{b3} that $\nu(C_{n + 1}) \geq 1 - \theta_{N} \geq \tfrac{1}{2}$, we have
\begin{displaymath} \bar{\nu}(B(x,r)) \leq 2 \cdot \nu(B(x,r)) \leq 40 \cdot r^{\gamma}, \qquad x \in \R, \, r \geq \delta. \end{displaymath}
We have now reached a situation which violates Theorem \ref{mainTechnical4} for the choice of parameters $\bar{\alpha},\beta,\kappa,\gamma$: the objects $\bar{A},B,\bar{\nu}$ satisfy all the hypotheses (by (a)-(c)), but nevertheless we have $|\bar{A} + cB|_{\delta} \leq \delta^{-\bar{\epsilon}}|\bar{A}|$ for all $c \in C_{n + 1}$, a set of full $\bar{\nu}$ measure, by \eqref{a5}. This violates Theorem \ref{mainTechnical4}, since $\bar{\epsilon} > 0$ was the constant associated to $\bar{\alpha},\beta,\kappa,\gamma$. Therefore the counter assumption \eqref{b5} is false, and the proof of Theorem \ref{mainTechnical3} is complete.

To be precise, we have ignored that $\bar{A} \subset [0,N]$ instead of $\bar{A} \subset [0,1]$. This can be dealt with as in the proof of Corollary \ref{longSumCor}, or below \eqref{form59}. We leave this to the reader. \end{proof}

\subsection{Bonus reduction}\label{s:bonusReduction} We have now reduced the proof of Theorem \ref{mainTechnical} to the proof of Theorem \ref{mainTechnical4}. For notational convenience in the future, we mention one final reduction: we may assume that $1 \in \spt (\nu)$. Indeed, assume that Theorem \ref{mainTechnical4} is known under this extra assumption. Then, let $A,B,\nu$ be a general triple as in Theorem \ref{mainTechnical4}. Since $\nu$ is a probability measure, $\spt (\nu) \subset [-1,1]$ and $\nu(B(x,r)) \leq 40 \cdot r^{\gamma}$, the point $c_{0} \in \spt(\nu) \cap [-1,1]$ with maximal absolute value satisfies
\begin{displaymath} |c_{0}| \geq 40^{-1/\gamma}. \end{displaymath}
Consider the measure $\bar{\nu}(A) := \nu(c_{0}A)$. Observe that $\bar{\nu}(B(x,r)) \leq 40 \cdot r^{\gamma}$ and $\spt(\bar{\nu}) = c_{0}^{-1}\spt(\nu)$. Therefore $1 \in \spt(\bar{\nu}) \subset [-1,1]$, so $\bar{\nu}$ satisfies the extra assumption. We then apply the (assumedly known) version of Theorem \ref{mainTechnical4} to $A,(c_{0}B)_{\delta},\bar{\nu}$. The set $(c_{0}B)_{\delta}$ will have slightly worse constants than $B$, in a manner depending on $\gamma$ only, so the theorem needs to be applied with appropriately modified parameters. Once this has been done, we find a point $c = c_{0}^{-1}c_{1} \in \spt(\bar{\nu})$, where $c_{1} \in \spt(\nu)$, such that
\begin{displaymath} |A + cB|_{\delta} \gtrsim |A + (c_{1}/c_{0}) \cdot (c_{0}B)_{\delta}|_{\delta} = |A + c(c_{0}B)_{\delta}|_{\delta} \geq \delta^{-\epsilon}|A|, \end{displaymath}
and the proof of Theorem \ref{mainTechnical4} (without the extra assumption) is complete.

\section{Proof of Theorem \ref{mainTechnical4}}

\subsection{Preliminaries} We have now reduced the proof of Theorem \ref{mainTechnical} to the proof of Theorem \ref{mainTechnical4}. We fix the parameters $\alpha,\beta,\gamma,\kappa$, with $0 < \beta \leq \alpha < 1$ and $(\alpha - \beta)/(1 - \beta) < \gamma \leq 1$. We also fix sets $A,B \subset (\delta \cdot \Z) \cap [0,1]$ and a Borel probability measure $\nu$ with $\spt(\nu) \subset [-1,1]$, satisfying all the hypotheses of Theorem \ref{mainTechnical4} with sufficiently small constants $\epsilon_{0},\epsilon_{B} > 0$ to be determined later. For future reference, we write
\begin{equation}\label{barAlpha} |A| =: \delta^{-\bar{\alpha}}, \qquad 0 \leq \bar{\alpha} \leq \alpha. \end{equation}
We make a counter assumption: $|A + cB|_{\delta} < \delta^{-\epsilon}|A|$ for all $c \in \spt(\nu)$. Since we may assume that $1 \in \spt(\nu)$ by Section \ref{s:bonusReduction}, we have the assumptions
\begin{equation}\label{counterAss} |A + B| \leq  \delta^{-\epsilon}|A| \quad \text{and} \quad |B + B| \leq \delta^{-\epsilon_{B}}|B|. \end{equation} 
If $\epsilon,\epsilon_{B} > 0$ in \eqref{counterAss} are small enough, depending only on $\alpha,\beta,\kappa,\gamma$, we will be able to find a point $c \in \spt(\nu)$ such that $|A + cB|_{\delta} \geq \delta^{-\epsilon}|A|$. This will violate the counter assumption, and prove Theorem \ref{mainTechnical4}. The necessary values of $\epsilon = \epsilon(\alpha,\beta,\gamma,\kappa) > 0$ and $\epsilon_{B} = \epsilon_{B}(\alpha,\beta,\gamma,\kappa) > 0$ in \eqref{counterAss} will be fixed during the proof of Proposition \ref{prop1}. 

\subsection{Shmerkin's inverse theorem} In the case $A = B$, Bourgain \cite{Bourgain10} used an assumption of the form \eqref{counterAss} to obtain, up to passing to a subset, a special multi-scale structure inside $A$: informally speaking, when passing from one scale to the next, either $A$ has \emph{full branching}, or then \emph{no branching}. Similar statements have, after Bourgain's work, been proved by Hochman \cite{Ho} and Shmerkin \cite{Sh} in the case where $A \neq B$, and where $A$ and $B$ may have completely different sizes. This is our situation, and we will apply Shmerkin's theorem, which we state in Theorem \ref{shmerkin}. 

\begin{definition}[$\delta$-sets and measures, $L^{2}$-norms]\label{def:deltaMeasures} Let $\delta \in 2^{-\N}$ be a dyadic rational. A subset of $(\delta \cdot \Z) \cap [0,1)$ is called a \emph{$\delta$-set}. A probability measure supported on a $\delta$-set is called a \emph{$\delta$-measure}. The $L^{2}$-norm of a $\delta$-measure $\mu$ is defined by
\begin{displaymath} \|\mu\|_{L^{2}} := \left(\sum_{z \in \delta \cdot \Z} \mu(\{z\})^{2} \right)^{1/2}. \end{displaymath}
\end{definition}

We will only be concerned with $\delta$-measures of the form $\mu = |A|^{-1}\mathcal{H}^{0}|_{A}$, where $A \subset [0,1)$ is a $\delta$-set. Then $\|\mu\|_{L^{2}} = |A|^{-1/2}$. 

\begin{definition}[Uniform sets] Let $m,N \in \N$, and set $\delta := 2^{-mN} \in 2^{-\N}$. For $A \subset [0,1)$ and $s \in \{0,\ldots,N - 1\}$, write $\mathcal{I}_{ms}(A) := \{I \in \mathcal{D}_{ms} : A \cap I \neq \emptyset\}$ for the collection of dyadic intervals of side-length $2^{-ms}$ (these are denoted $\mathcal{D}_{ms}$) with non-empty intersection with $A$. We say that $A$ is \emph{$(m,N)$-uniform} if
\begin{displaymath} R_{A}(s) := |I \cap A|_{2^{-m(s + 1)}}, \qquad I \in \mathcal{I}_{ms}(A), \end{displaymath}
is independent of the choice of $I \in \mathcal{I}_{ms}(A)$. We may also write that $A$ is \emph{$(m,N,R_{A})$-uniform} if the \emph{branching numbers} $R_{A}$ need emphasising.
\end{definition} 

In the definition of $R_{A}(s)$, is it important to remember that $|H|_{r}$ is, by definition, the number of dyadic $r$-intervals intersecting $H$ -- instead of the $r$-covering number. This distinction has hardly mattered earlier in the paper. 

As in \cite{Sh}, we will only consider uniform sets which are also $\delta$-sets. It was observed by Bourgain \cite{Bourgain10} that every $\delta$-set contains a uniform subset of "comparable" cardinality. Thus, the possibility of finding uniform subsets has nothing to do, yet, with an assumption like \eqref{counterAss}. To explain what \eqref{counterAss} implies, we introduce the following terminology:
\begin{definition}[$\eta$-polarised pair] Let $m,N \in \N$, $\delta = 2^{-mN}$, and $\eta > 0$. A pair of $(m,N)$-uniform sets $(A,B)$ is \emph{$(\eta,m,N)$-polarised}, if 
\begin{displaymath} R_{B}(s) > 1 \quad \Longrightarrow \quad R_{A}(s) \geq 2^{(1 - \eta)m}, \qquad s \in \{0,\ldots,N - 1\}. \end{displaymath}
If $A = B$, we say that $A$ (instead of $(A,A)$) is $(\eta,m,N)$-polarised. 
\end{definition}
Note that $R_{A}(s) \leq 2^{m}$ for all $s \in \{0,\ldots,N - 1\}$, so $R_{A}(s) \geq 2^{(1 - \eta)m}$ means that $R_{A}(s)$ is nearly maximal. Bourgain \cite{Bourgain10} proved that if $A$ is a $\delta$-set with $|A + A| \leq \delta^{-\epsilon}|A|$, then $A$ contains a uniform subset $A'$ such that $|A'| \geq \delta^{\eta}|A|$, and $A'$ is $\eta$-polarised, where $\eta = o_{\epsilon}(1)$. This means that either $R_{A'}(s) = 1$ or $R_{A'}(s) \geq 2^{(1 - \eta)m}$ for all scales "$s$". 

Versions of Bourgain's "polarisation theorem", explained above, for two different sets were found by Hochman \cite{Ho} and Shmerkin \cite{Sh}. Hochman first showed that if $\mu,\nu$ are probability measures on $[0,1)$, then the entropy inequality $H(\mu \ast \nu) \leq H(\mu) + \epsilon$ implies a measure-theoretic version of the polarisation phenomenon for $\mu,\nu$. The set version, below, was established by Shmerkin \cite{Sh} (with a proof very different from \cite{Ho}):

\begin{thm}[Shmerkin]\label{shmerkin} Let $\eta > 0$, and let $m(\eta) \in \N$ be sufficiently large, depending on $\eta$. Then, for all $m \geq m(\eta)$ there exists $\epsilon = \epsilon(\eta,m) > 0$ such that the following holds for all large enough $N \in \N$. Let $\delta = (2^{-m})^{N}$, and let $A,B \subset [0,1]$ be $\delta$-sets such that
\begin{displaymath} |A + B| \leq \delta^{-\epsilon}|A|. \end{displaymath}
Then, there exist $(m,N)$-uniform sets $A' \subset A$ and $B' \subset B$ such that $|A'| \geq \delta^{\eta}|A|$, $|B'| \geq \delta^{\eta}|B|$, and $(A',B')$ is $(\eta,m,N)$-polarised. \end{thm}

\begin{remark}\label{rem2} To be accurate, Theorem \ref{shmerkin} is a slight refinement of Shmerkin's theorem: \cite[Theorem 2.1]{Sh} literally contains the following statement: \emph{if $\eta > 0$ and $m_{0} \in \N$, then there exists $m = m(\eta,m_{0}) \geq m_{0}$ and $\epsilon = \epsilon(\eta,m_{0}) > 0$ as in Theorem \ref{shmerkin}.} However, if one inspects the proof of \cite[Theorem 2.1]{Sh}, one observes that the only dependence of $m = m(\eta,m_{0})$ on $m_{0}$ is "$m \geq m_{0}$", and any choice of $m \geq m_{0}$ works, provided that $m$ is also sufficiently large in terms of $\eta$. This is precisely what Theorem \ref{shmerkin} says. \end{remark}

As another remark, Shmerkin's theorem actually concerns a pair of $\delta$-measures $\mu_{1},\mu_{2}$ instead of $\delta$-sets: the measures of interest for our application are simply $\mu_{1} = |A|^{-1}\mathcal{H}^{0}|_{A}$ and $\mu_{2} = |B|^{-1}\mathcal{H}^{0}|_{B}$, and with such choices \cite[Theorem 2.1]{Sh} implies Theorem \ref{shmerkin}.

\begin{remark} We will be applying Theorem \ref{shmerkin} to dyadic scales of the form $\delta = 2^{-\ell m N}$, where $\ell,m,N \in \N$. Since $\delta = (2^{-m})^{\ell N} = (2^{-\ell m})^{N}$, a $\delta$-set $A \subset [0,1)$ may be $(m,\ell N)$-uniform, $(\ell m,N)$-uniform, or both. The former condition means that the branching numbers $R_{A}^{m}(s) = |A \cap I|_{2^{-m(s + 1)}}$ are well-defined for $m \in [\ell N]$, whereas the latter means that the branching numbers $R_{A}^{\ell m}(\sigma) = |A \cap I|_{2^{-\ell m(\sigma + 1)}}$ are well-defined for $\sigma \in [N]$. It is clear that every $(m,\ell N)$-uniform $2^{-\ell m N}$-set is $(\ell m,N)$-uniform, and indeed
\begin{displaymath} R_{A}^{\ell m}(\sigma) = \prod_{s = \ell \sigma}^{\ell(\sigma + 1) - 1} R_{A}^{m}(s), \qquad \sigma \in [N]. \end{displaymath}
The converse is generally not true, so $(m,\ell N)$-uniformity is a strictly stronger property than $(\ell m,N)$-uniformity. We will also be interested in pairs $(A,B)$ which are sometimes $(\eta,m,\ell N)$-polarised, and sometimes $(\eta,\ell m,N)$-polarised. In contrast to uniformity, there is no simple implication between these two properties. 

\end{remark}
In addition to Shmerkin's theorem, we will also need a lemma from its proof:

\begin{lemma}\label{ShLemma3} Let $m,\ell,N \in \N$, $\delta = 2^{-\ell mN}$, and let $A \subset [0,1)$ be an $(m,\ell N)$-uniform $\delta$-set. Then $A$ is also $(\ell m,N,R^{\ell m}_{A})$-uniform for some $R_{A} \colon [N] \to \{1,\ldots,2^{\ell m}\}$. If $\mathcal{S} \subset [N]$ is arbitrary, there exists $A' \subset A$ which is $(m,\ell N)$-uniform, and also $(\ell m,N,R_{A'}^{\ell m})$-uniform with 
\begin{displaymath} |A'| \geq |A| \cdot \prod_{\sigma \in \mathcal{S}} R_{A}^{\ell m}(\sigma)^{-1}, \quad \text{and} \quad R_{A'}^{\ell m}(\sigma) = \begin{cases} 1, & \sigma \in \mathcal{S}, \\ R_{A}^{\ell m}(\sigma), & \sigma \notin \mathcal{S}. \end{cases} \end{displaymath} 
A similar statement holds true if $\mathcal{S} \subset [\ell N]$, with the only difference that "$R_{A}^{\ell m}(\sigma)$" and "$R_{A'}^{\ell m}(\sigma)$" should be replaced by "$R_{A}^{m}(s)$" and "$R_{A'}^{m}(s)$" for $s \in [\ell N]$. \end{lemma}

The lemma above is \cite[Lemma 3.7]{Sh}. To be accurate, the statement about $A'$ remaining $(m,\ell N)$-uniform is not part of the statement of \cite[Lemma 3.7]{Sh}, but the $3.8$-line proof quickly reveals that $(m,\ell N)$-uniformity is not violated when passing between $A$ and $A'$; the only point is to "collapse" all the branching of $A$ for levels corresponding to $\sigma \in \mathcal{S}$, or equivalently for $s \in \{\ell \sigma,\ell(\sigma + 1) - 1\}$ for all $\sigma \in \mathcal{S}$.
  
\subsection{Applying the inverse theorem} We start by fixing the following parameters:
\begin{equation}\label{allParameters} \begin{cases} \ell = \ell(\alpha,\beta,\gamma,\kappa) \in \N,\\ \eta = \eta(\alpha,\beta,\gamma,\kappa) \in (0,1), \\ m_{0} \in \N \text{ with } m_{0} \geq (40 + C_{0})/\eta. \end{cases}
\end{equation}
Here $C_{0} > 0$ is an absolute constant to be specified later. In fact, the values of all these constants will be specified later, but as indicated above, all of them only depend on $\alpha,\beta,\gamma,\kappa$. For the reader interested in seeing specific choices, we refer to \eqref{parameters} and the discussion afterwards. Recall that the set $B$ satisfies the Frostman condition $|B \cap B(x,r)| \leq r^{\kappa}|B|$ for all $\delta \leq r \leq \delta^{\epsilon_{0}}$, where we may freely choose $\epsilon_{0} = \epsilon_{0}(\alpha,\beta,\kappa,\gamma) > 0$. We choose
\begin{equation}\label{form58} \epsilon_{0} := \eta. \end{equation}
We will assume that $\eta,\epsilon,\epsilon_{B} < 1/1000$ in the sequel (but these upper bounds will generally not suffice). This section is devoted to the proof of the following proposition, whose proof will also finalise the choice of the parameters $\epsilon,\epsilon_{B} > 0$, relative to $\eta$:
\begin{proposition}\label{prop1} There exist $\epsilon,\epsilon_{B} > 0$ and $m \geq m_{0}$, depending on $\alpha,\beta,\gamma,\kappa$, such that the following holds for all $\delta \in 2^{-\N}$ of the form $\delta = 2^{-\ell m N}$, $N \in \N$. Assume that $A,B \subset [0,1]$ are $\delta$-sets satisfying the small doubling assumptions \eqref{counterAss}. Then there exist subsets $A' \subset A$ and $B' \subset B$ with the following properties:
\begin{enumerate}
\item $A'$ and $B'$ are $(m,\ell N)$-uniform with $|A'| \geq \delta^{\eta}|A|$ and $|B'| \geq \delta^{\eta/2}|B|$.
\item The pair $(A',B')$ is $(\eta,m,\ell N)$-polarised.
\item The set $B'$ is $(\eta/2,\ell m,N)$-polarised.
\end{enumerate}
\end{proposition}

\begin{remark} In the sequel, we will always work with scales of the form $\delta = 2^{-\ell m N}$ with the fixed parameters $\ell,m$, which depend on $\alpha,\beta,\gamma,\kappa$. In other words, we initially prove Theorem \ref{mainTechnical4} (and find the constants $\epsilon,\epsilon_{0},\epsilon_{B}$) for only scales of this special form. After this has been accomplished, it is easy to check that the case of general scales $\delta \in 2^{-\N}$ is a corollary, assuming that the upper bound $\delta_{0} = \delta_{0}(\alpha,\beta,\gamma,\kappa) > 0$ for $\delta$ is sufficiently small. The reason is that if $\delta \in 2^{-\N}$ is arbitrary, then there exists a scale of the form $\bar{\delta} = 2^{-\ell m N}$ with $\delta \leq \bar{\delta} \lesssim_{\alpha,\beta,\gamma,\kappa} \delta$. We leave the rest of this reduction to the reader.

As another remark, we will later in the paper need to assume that $\epsilon = \epsilon(\alpha,\gamma) > 0$ is sufficiently small that
\begin{equation}\label{form98} \frac{\epsilon}{1 - \alpha - \epsilon} \leq \frac{\gamma}{2}. \end{equation}
This requirement should be combined with the one coming from Proposition \ref{prop1}.  \end{remark}

\begin{proof}[Proof of Proposition \ref{prop1}] We begin by applying Theorem \ref{shmerkin} with constant $\eta^{3} > 0$ to the pair $(B,B)$, for which we assumed in \eqref{counterAss} that $|B + B| \leq \delta^{-\epsilon_{B}}|B|$. Assume that $m \geq m(\eta^{3}) \in \N$ is sufficiently large that Theorem \ref{shmerkin} applies. Assume additionally that $m \geq m_{0}$, where $m_{0}$ is the constant from \eqref{allParameters}. Then, if $\epsilon_{B} = \epsilon_{B}(\eta^{3},\ell m) = \epsilon_{B}(\alpha,\beta,\gamma,\kappa) > 0$ and $\delta = (2^{-\ell m})^{N}$ are sufficiently small, we find an $(\ell m,N)$-uniform subset $B' \subset B$ such that $|B'| \geq \delta^{\eta^{3}}|B|$, and $B'$ is $(\eta^{3},\ell m,N)$-polarised. \emph{We have now fixed the value of the parameter $\epsilon_{B} > 0$ in \eqref{allParameters} (and hence in Theorem \ref{mainTechnical4})!}

Next, note that $|A + B'| \leq |A + B| \leq \delta^{-\epsilon}|A|$. We therefore may apply Theorem \ref{shmerkin} again to the pair $(A,B')$, again with parameter $\eta^{3} > 0$. If $\epsilon = \epsilon(\eta^{3},m) > 0$ is sufficiently small, we find an $(m,\ell N)$-uniform subset $A' \subset A$ with $|A'| \geq \delta^{\eta^{3}}|A| \geq \delta^{\eta}|A|$, and an $(m,\ell N)$-uniform subset $B'' \subset B'$ such that 
\begin{equation}\label{form66} |B''| \geq \delta^{\eta^{3}}|B'|, \end{equation}
and  $(A',B'')$ is $(\eta^{3},m,\ell N)$-polarised. In particular $(A',B'')$ is $(\eta,m,\ell N)$-polarised. \emph{We have now fixed the value of the parameter $\epsilon > 0$ in \eqref{counterAss}!}

Are we done with properties (1)-(3) in Proposition \ref{prop1}? Not quite: while passing from $B'$ to $B''$, we might have lost the $(\eta^{3},\ell m,N)$-polarisation of $B'$. The plan will be to pass to a final $(m,\ell N)$-uniform subset $B''' \subset B''$ which is $(\eta/2,\ell m,N)$-polarised, and such that $|B'''| \geq \delta^{\eta/4}|B''|$. Then finally
\begin{displaymath} |B'''| \geq \delta^{\eta/4}|B''| \geq \delta^{\eta/4 + \eta^{3}}|B'| \geq \delta^{\eta/4 + 2\eta^{3}}|B| \geq \delta^{\eta/2}|B|. \end{displaymath}
Also $(A',B''')$ remains $(\eta,m,\ell N)$-polarised, since this property is not violated by replacing $B''$ by an $(m,\ell N)$-uniform subset, for example $B'''$.

Write 
\begin{displaymath} \mathcal{S}_{0} := \{\sigma \in [N] : R_{B'}^{\ell m}(\sigma) = 1\} \quad \text{and} \quad \mathcal{S}_{1} := \{\sigma \in [N] : R_{B'}^{\ell m}(\sigma) \geq 2^{(1 - \eta^{3})\ell m}\}. \end{displaymath} 
Since $B'$ was constructed to be $(\eta^{3},\ell m,N)$-polarised, we have $[N] = \mathcal{S}_{0} \cup \mathcal{S}_{1}$, and
\begin{displaymath} |B'| = \prod_{\sigma \in \mathcal{S}_{1}} R_{B'}^{\ell m}(\sigma) \geq 2^{(1 - \eta^{3})\ell m |\mathcal{S}_{1}|}. \end{displaymath}
Now, let $\mathcal{S}_{\mathrm{bad}} := \{\sigma \in \mathcal{S}_{1} : R_{B''}^{\ell m}(\sigma) < 2^{(1 - \eta/2)\ell m}\}$, and $\mathcal{S}_{\mathrm{good}} := \mathcal{S}_{1} \, \setminus \, \mathcal{S}_{\mathrm{bad}}$. (Note that the numbers $R_{B''}^{\ell m}(\sigma)$ are well-defined, since $B''$ is $(m,\ell N)$-uniform, hence $(\ell m,N)$-uniform.) Then, since evidently $R_{B''}^{\ell m}(\sigma) \leq R_{B'}^{\ell m}(\sigma) = 1$ for all $\sigma \in \mathcal{S}_{0}$, we have
\begin{align*} |B''| = \prod_{\sigma \in \mathcal{S}_{\mathrm{bad}}} R_{B''}^{\ell m}(\sigma) \cdot \prod_{\sigma \in \mathcal{S}_{\mathrm{good}}} R_{B''}^{\ell m}(\sigma) & \leq 2^{(1 - \eta/2)\ell m |\mathcal{S}_{\mathrm{bad}}|} \cdot 2^{\ell m (|\mathcal{S}_{1}| - |\mathcal{S}_{\mathrm{bad}}|)}\\
& = 2^{\ell m |\mathcal{S}_{1}| - (\eta/2) \ell m |\mathcal{S}_{\mathrm{bad}}|}. \end{align*}
On the other hand, recalling that $\delta = 2^{\ell m N}$, we have
\begin{displaymath} |B''| \stackrel{\eqref{form66}}{\geq} \delta^{\eta^{3}}|B'| = \delta^{\eta^{3}} \cdot \prod_{\sigma \in \mathcal{S}_{1}} R_{B'}^{\ell m}(\sigma) \geq 2^{-\eta^{3}\ell m N + (1 - \eta^{3})\ell m|\mathcal{S}_{1}|}. \end{displaymath}
Combining these inequalities and dividing both sides by $2^{\ell m |\mathcal{S}_{1}|}$ yields
\begin{displaymath} 2^{-(\eta/2) \ell m |\mathcal{S}_{\mathrm{bad}}|} \geq 2^{-\eta^{3}\ell m N - \eta^{3} \ell m |\mathcal{S}_{1}|} \quad \Longrightarrow \quad |\mathcal{S}_{\mathrm{bad}}| \leq 2\eta^{2} N + 2\eta^{2} |\mathcal{S}_{1}| \leq (\eta/4) N, \end{displaymath}
noting that $16\eta^{2} \leq \eta$ since $\eta < 1/1000$. At this point, we simply apply the "collapsing" Lemma \ref{ShLemma3} to the the $(m,\ell N)$-uniform set $B''$, and the set of scales $\mathcal{S} := \mathcal{S}_{\mathrm{bad}}$. The result is an $(m,\ell N)$-uniform subset $B''' \subset B''$, which is also $(\ell m,N,R_{B'''}^{\ell m})$-uniform, with
\begin{displaymath} R_{B'''}^{\ell m}(\sigma) = \begin{cases} R_{B''}^{\ell m}(\sigma), & \sigma \notin \mathcal{S}_{\mathrm{bad}}, \\ 1, & \sigma \in \mathcal{S}_{\mathrm{bad}}. \end{cases}, \end{displaymath}
and
\begin{displaymath} |B'''| \geq |B''| \cdot \prod_{\sigma \in \mathcal{S}_{\mathrm{bad}}} R_{B''}^{\ell m}(\sigma)^{-1} \geq |B''| \cdot 2^{-\ell m |\mathcal{S}_{\mathrm{bad}}|} \geq \delta^{\eta/4} \cdot |B''|. \end{displaymath}
Then $B'''$ is $(\eta/2,\ell m,N)$-polarised, since if $R_{B'''}^{\ell m}(\sigma) > 1$, then necessarily $\sigma \in \mathcal{S}_{\mathrm{good}}$, hence $R_{B'''}^{\ell m}(\sigma) = R_{B''}^{\ell m}(\sigma) \geq 2^{(1 - \eta/2)\ell m}$ by definition. Now the pair of sets $A',B'''$ (in place of $A',B'$) satisfies all the requirements (1)-(3) in Proposition \ref{prop1}. \end{proof}

\subsection{Pruning $B'$ to improve separation I} It will be useful to reduce $B'$ to a further subset, in order to gain a little extra separation. We prove the following proposition:
\begin{proposition}\label{prop2} Let $B' \subset B$ be the $(m,\ell N)$-uniform $(\eta/2,\ell m,N)$-polarised set found in Proposition \ref{prop1}. Then, there exists an $(m,\ell N)$-uniform $(\eta,\ell m,N)$-polarised subset $B'' \subset B'$ with $|B''| \geq \delta^{\eta/2}|B'|$, and which satisfies the following separation property:
\begin{equation}\label{form8} s \in [\ell N], \, I_{1},I_{2} \in \mathcal{I}_{ms}(B''), \, I_{1} \neq I_{2} \quad \Longrightarrow \quad \dist(I_{1},I_{2}) \geq 2^{-ms}. \end{equation}
\end{proposition}
\begin{proof} We perform a straightforward "top down" reduction of $B'$. At scale $s = 0$, there is only one interval $[0,1) \in \mathcal{I}_{0}(B')$, so \eqref{form8} is clear. At scale $s = 1$, remove $I \cap B'$ from $B'$ for at most every second interval $I \in \mathcal{I}_{m}(B')$. This results in a new set $B_{1}' \subset B'$ with $|B_{1}'| \geq \tfrac{1}{2}|B'|$ points, and 
\begin{displaymath} R_{B_{1}'}^{m}(0) := |B_{1}'|_{2^{-m}} \geq \tfrac{1}{2}R_{B'}^{m}(0). \end{displaymath}
Next, for $s = 2$, consider every (remaining) interval in $\mathcal{I}_{m}(B_{1}')$. For each $I \in \mathcal{I}_{m}(B_{1}')$, at most $\tfrac{1}{2}$ of the intervals from $\mathcal{I}_{2m}(B_{1}')$ contained in $I$ need removal to ensure \eqref{form8} at level $s = 2$. However, the (minimal) number may vary depending on the choice of $I \in \mathcal{I}_{m}(B_{1}')$. Fortunately, by removing some extra $\mathcal{I}_{2m}(B_{1}')$-intervals if necessary, we may make the number independent of $I \in \mathcal{I}_{m}(B_{1}')$. This way, the number of remaining points in $B_{1}'$ again gets reduced by at most a factor of $\tfrac{1}{2}$. For the remaining points, say $B_{2}'$, the new branching numbers $R^{m}_{B_{2}'}(2) := |B_{2}' \cap I|_{2^{-2m}}$ are independent of the choice of $I \in \mathcal{I}_{m}(B_{2}')$, and of course $R_{B_{2}'}^{m}(0) = R^{m}_{B_{1}'}(0)$ (in general, the removal process at level $s$ never alters the branching numbers at levels $< s - 1$). Hence $B_{2}'$ is again $(m,\ell N)$-uniform.

Once the deletion process has been executed successively at all levels $s \in \{1,\ldots,\ell N\}$, the remaining set $B'' := B_{N}'$ satisfies 
\begin{displaymath} |B''| \geq 2^{-\ell N}|B'| = \delta^{1/m} \cdot |B'| \geq \delta^{\eta/2} \cdot |B'|, \end{displaymath} 
recalling that $m \geq m_{0} \geq 2/\eta \geq 1/\eta$ by Proposition \ref{prop1} and \eqref{allParameters}. Also, $B''$ remains $(m,\ell N)$-uniform, with $R_{B''}^{m}(s) \geq \tfrac{1}{2}R_{B'}^{m}(s)$ for $s \in [\ell N]$. This implies that if $R_{B''}^{\ell m}(\sigma) > 1$, then
\begin{displaymath} R_{B''}^{\ell m}(\sigma) \geq 2^{-\ell} R_{B'}^{\ell m}(\sigma) \geq 2^{-\ell + (1 - \eta/2)\ell m} = 2^{(1 - \eta/2 - 1/m)\ell m} \geq 2^{(1 - \eta)\ell m}, \end{displaymath}
by the $(\eta/2,\ell m,N)$-polarisation of $B'$, and using that $m \geq m_{0} \geq 2/\eta$, recall \eqref{allParameters}. In other words, $B''$ is $(\eta,\ell m,N)$-polarised, as claimed. \end{proof}

Since $|B'| \geq \delta^{\eta/2}|B|$ by Proposition \ref{prop2}, we have $|B''| \geq \delta^{\eta/2}|B'| \geq \delta^{\eta}|B|$. Also, since $(A',B')$ is $(\eta,m,\ell N)$-polarised, the same is true of $(A',B'')$. To simplify notation, we will remove one prime, that is, assume that $B'$ already satisfies the separation property \eqref{form8} constructed for $B''$ in the proposition above. To summarise the progress so far:
\begin{itemize} 
\item $|A'| \geq \delta^{\eta}|A|$ and $|B'| \geq \delta^{\eta}|B|$,
\item $(A',B')$ is $(\eta,m,\ell N)$-polarised and $B'$ is $(\eta,\ell m,N)$-polarised,
\item $B'$ satisfies the separation property \eqref{form8}.
\end{itemize}

\subsection{Intervals with small but non-zero $B'$-branching}\label{s:combining} For intervals $\mathcal{I} \subset [\ell N]$ and $\mathcal{J} \subset [N]$, and for $D' \in \{A',B'\}$, we define
\begin{displaymath} R_{D'}^{m}(\mathcal{I}) := \prod_{s \in \mathcal{I}} R_{D'}^{m}(s) \quad \text{and} \quad R_{D'}^{\ell m}(\mathcal{J}) := \prod_{\sigma \in \mathcal{J}} R_{D'}^{\ell m}(\sigma). \end{displaymath}
We also define $\mathcal{N}_{B'}$ to consist of the maximal intervals $\mathcal{I} \subset [N]$ such that $R_{B'}^{\ell m}(\mathcal{I}) = 1$. Thus the intervals in $\mathcal{N}_{B'}$ partition the set $\{\sigma \in [N] : R_{B'}^{\ell m}(\sigma) = 1\}$.

\begin{lemma}\label{lemma1} Write $|B| := \delta^{-\beta_{1}}$, where $\beta_{1} \geq \beta$ by the hypothesis of Theorem \ref{mainTechnical4}. Then the following lower bound holds for the total length of the intervals in $\mathcal{N}_{B'}$:
\begin{displaymath} \sum_{\mathcal{I} \in \mathcal{N}_{B'}} |\mathcal{I}| \geq (1 - \tfrac{\beta_{1}}{1 - \eta}) \cdot N. \end{displaymath} 
\end{lemma}

\begin{remark}\label{rem3} We remark that $\beta_{1} \leq \alpha + \epsilon < 1$ (recall \eqref{form98} for the second inequality). Otherwise $|A + B| \geq |B| = \delta^{-\beta_{1}} > \delta^{-\epsilon}|A|$, violating our counter assumption \eqref{counterAss}. Therefore $1 - \beta_{1} > 0$, which will be (tacitly) needed several times below. \end{remark}

\begin{proof} Let $N_{B'} := \cup \mathcal{N}_{B'} \subset [N]$. Note that if $\sigma \in [N] \, \setminus \, N_{B'}$, then $R_{B'}^{\ell m}(\sigma) > 1$, and consequently $R_{B'}^{\ell m}(\sigma) \geq 2^{(1 - \eta)\ell m}$, by the $(\eta,\ell m,N)$-polarisation of $B'$. It follows that
\begin{displaymath} 2^{\beta_{1}\ell m N} = \delta^{-\beta_{1}} \geq |B'| = \prod_{\sigma \in [N] \, \setminus \, N_{B'}} R_{B'}^{\ell m}(\sigma) \geq 2^{(1 - \eta) \ell m \cdot (N - |N_{B'}|)}. \end{displaymath}
Consequently $N - |N_{B'}| \leq (\beta_{1}/(1 - \eta)) \cdot N$, and finally
\begin{displaymath} |N_{B'}| \geq N - \tfrac{\beta_{1}}{1 - \eta} \cdot N = (1 - \tfrac{\beta_{1}}{1 - \eta}) \cdot N, \end{displaymath}
as claimed.  \end{proof}

We would next like to extend the intervals $\mathcal{I} \in \mathcal{N}_{B'}$ to the left in such a manner that the $B'$-branching numbers of the extended intervals are relatively small, but not vanishingly small; say, we keep extending left until the extension $\overline{\mathcal{I}} \supset \mathcal{I}$ satisfies $R_{B'}^{\ell m}(\overline{\mathcal{I}}) \sim 2^{\zeta \ell m|\overline{\mathcal{I}}|}$ for some small parameter $\zeta > 0$. Unfortunately, this is not always possible: consider for example a scenario where the elements in $\mathcal{N}_{B'}$ are singletons. Then, as soon as an interval (a singleton) $\mathcal{I} \in \mathcal{N}_{B'}$ is extended to the left by a single element $\sigma \in [N]$, we have $R_{B'}^{\ell m}(\sigma) \approx 2^{\ell m}$ by the polarisation of $B'$, and hence
\begin{displaymath} R_{B'}^{\ell m}(\mathcal{I} \cup \{\sigma\}) = R_{B'}^{\ell m}(\sigma) \approx 2^{\ell m} = 2^{\ell m |\mathcal{I} \cup \{\sigma\}|/2}. \end{displaymath}
The right hand side is far too large for our purposes. This issue is not possible to overcome as long as we are fixated with the single scale partition $[N]$, and here arises the need to play with the two different scale partitions $[N]$ and $[\ell N]$.

The solution is to identify every interval $\mathcal{I} \in \mathcal{N}_{B'}$ with another, $\ell$-times longer, sub-interval of $[\ell N]$. More precisely, for every $\mathcal{I} = \{\sigma,\ldots,\tau\} \subset [N]$, we define the interval 
\begin{displaymath} \ell \mathcal{I} := \{\ell \sigma,\ell \sigma + 1,\ldots,\ell (\tau + 1) - 1\} \subset [\ell N]. \end{displaymath}
Thus, for example $\ell \{0\} = [\ell]$, and $\ell [N] = [\ell N]$. It is easy to check that the branching numbers interact with this operation as follows:
\begin{displaymath} R^{m}_{D'}(\ell \mathcal{I}) = R^{\ell m}_{D'}(\mathcal{I}), \qquad D' \in \{A',B'\}, \, \mathcal{I} \subset [N]. \end{displaymath}
In particular:
\begin{equation}\label{form28} \mathcal{I} \in \mathcal{N}_{B'} \quad \Longrightarrow \quad R_{B'}^{\ell m}(\mathcal{I}) = 1 \quad \Longrightarrow \quad R_{B'}^{m}(\ell \mathcal{I}) = 1. \end{equation}
Motivated by this observation, we define 
\begin{displaymath} \ell \mathcal{N}_{B'} := \{\ell \mathcal{I} : \mathcal{I} \in \mathcal{N}_{B'}\}. \end{displaymath}
The intervals in $\mathcal{N}_{B'}$ partition $\{\sigma \in [N] : R_{B'}^{\ell m}(\sigma) = 1\}$ by definition. The intervals in $\ell \mathcal{N}_{B'}$ are contained in the set $\{s \in [\ell N] : R_{B'}^{m}(s) = 1\}$ by \eqref{form28}, but may not cover it. However, we may infer the following lower bound for their total length form Lemma \ref{lemma1}:
\begin{cor}\label{cor2} The following lower bound holds for the total length of the intervals in $\mathcal{N}_{B'}^{m}$:
\begin{displaymath} \sum_{\mathcal{J} \in \ell \mathcal{N}_{B'}} |\mathcal{J}| \geq (1 - \tfrac{\beta_{1}}{1 - \eta}) \cdot \ell N. \end{displaymath}
\end{cor}

\begin{proof} This follows immediately from Lemma \ref{lemma1} and the relation $|\ell \mathcal{I}| = \ell |\mathcal{I}|$ (noting also that if $\mathcal{I}_{1},\mathcal{I}_{2} \subset [N]$ with $\mathcal{I}_{1} \cap \mathcal{I}_{2} = \emptyset$, then $\ell \mathcal{I}_{1} \cap \ell \mathcal{I}_{2} = \emptyset$). \end{proof}

Next, instead of extending the intervals in $\mathcal{N}_{B'}$ to the left, as we first proposed, we do this to the intervals in $\ell \mathcal{N}_{B'}$. More precisely, fix a small parameter
\begin{equation}\label{zeta} \zeta = \zeta(\alpha,\beta,\gamma,\kappa) > 0 \quad \text{with} \quad \eta \ll_{\alpha,\beta,\gamma,\kappa} \zeta < \kappa/2. \end{equation}
It will eventually turn out that $\eta,\zeta > 0$ will need to be chosen so small that
\begin{equation}\label{parameters} \left(1 - \tfrac{\beta_{1}}{1 - \eta} - \tfrac{2\eta}{\kappa}\right) - \left(\tfrac{\alpha - (1 - \eta)(\beta_{1} - \eta - 2\zeta)}{\Gamma}\right) \geq \tfrac{1}{2} \cdot [(1 - \beta_{1}) - (\alpha - \beta_{1})/\Gamma], \end{equation}
where $(\alpha - \beta)/(1 - \beta) < \Gamma < \gamma$ is a constant to be fixed in \eqref{gammaA}, which satisfies $\Gamma \geq \gamma/2$. It might look suspicious that the requirement in \eqref{parameters} depends on $\beta_{1}$, and not just $\alpha,\beta,\kappa,\gamma$. To see that this is not a problem, denote the left hand side $L(\beta_{1})$ and the right hand side $\tfrac{1}{2} \cdot R(\beta_{1})$. With this notation, one can easily check that
\begin{displaymath} \left| L(\beta_{1}) - R(\beta_{1}) \right| \lesssim \frac{\eta}{\kappa} + \frac{\eta + \zeta}{\gamma} \quad \text{and} \quad R(\beta_{1}) \geq R(\beta) \gtrsim_{\alpha,\beta,\gamma} 1.  \end{displaymath}
The lower bound on $R(\beta)$ follows from $\Gamma > (\alpha - \beta)/(1 - \beta)$. From these estimates, one sees that if $\eta,\zeta > 0$ are chosen small enough depending only on $\alpha,\beta,\gamma,\kappa$, then $|L(\beta_{1}) - R(\beta_{1})| \leq \tfrac{1}{2}R(\beta) \leq \tfrac{1}{2}R(\beta_{1})$, and hence $L(\beta_{1}) \geq \tfrac{1}{2} \cdot R(\beta_{1})$, as in \eqref{parameters}.

In addition to the constraint in \eqref{parameters}, we will finally (after \eqref{form50}) need to take $\eta$ sufficiently small in terms of $\zeta(\gamma - \Gamma)/4$. Of course \eqref{parameters} is compatible with such a constraint.

Enumerate $\ell \mathcal{N}_{B'} = \{\mathcal{I}_{1},\mathcal{I}_{2},\ldots,\mathcal{I}_{k}\}$, where $\max \mathcal{I}_{j} < \min \mathcal{I}_{j + 1}$. Start with $\mathcal{I}_{k}$, and recall that $R_{B'}^{m}(\mathcal{I}_{k}) = 1$ by \eqref{form28}. Begin extending $\mathcal{I}_{k}$ to the left, adding elements of $[\ell N]$ one by one, until (the newly defined interval) $\mathcal{I}_{k}$ satisfies 
\begin{itemize}
\item[(a)] $R_{B'}^{m}(\mathcal{I}_{k}) \geq 2^{\zeta m |\mathcal{I}_{k}|}$, or
\item[(b)] $0 \in \mathcal{I}_{k}$, and $R_{B'}^{m}(\mathcal{I}_{k}) < 2^{\zeta m |\mathcal{I}_{k}|}$.
\end{itemize}
In both cases (a)-(b) we have the upper bound
\begin{equation}\label{form10} R_{B'}^{m}(\mathcal{I}_{k}) < 2^{\zeta m(|\mathcal{I}_{k}| - 1)} \cdot 2^{m} \leq 2^{m(\zeta |\mathcal{I}_{k}| + 1)} \leq 2^{m(2\zeta)|\mathcal{I}_{k}|},
\end{equation}
choosing here $\ell = \ell(\alpha,\beta,\gamma) \geq 1$ sufficient large that
\begin{displaymath} \ell \geq \zeta^{-1} \quad \Longrightarrow \quad \zeta |\mathcal{I}_{k}| \geq \ell \zeta \geq 1. \end{displaymath}
This will be legitimate, since $\zeta$ only depends on $\alpha,\beta,\gamma,\kappa$. Later, in \eqref{form27}, we will need to couple this requirement with $\ell \geq \epsilon^{-1}$, where $\epsilon = \epsilon(\alpha,\beta,\gamma) > 0$.

The extended interval $\mathcal{I}_{k}$ may have "swallowed" a number of the previous intervals $\mathcal{I}_{j} \in \ell \mathcal{N}_{B'}$: note however that if $\mathcal{I}_{j} \cap \mathcal{I}_{k} \neq \emptyset$ for some $1 \leq j < k$, then actually $\mathcal{I}_{j} \subset \mathcal{I}_{k}$, since $R_{B'}^{m}(\mathcal{I}_{j}) = 1$ (in other words, there is no reason why the extension algorithm would terminate in the middle of $\mathcal{I}_{j}$). Let $\mathcal{I}_{k_{1}} \in \ell \mathcal{N}_{B'}$ be the right-most interval which does not intersect (equivalently: is not contained in) the extension $\mathcal{I}_{k}$. If no such interval remains, the algorithm terminates. Otherwise, repeat the extension procedure with $\mathcal{I}_{k_{1}}$. Continue in this manner until all the intervals in $\ell \mathcal{N}_{B'}$ are contained in (precisely) one of the extensions. Then rename the extensions as $\{\mathcal{J}_{1},\ldots,\mathcal{J}_{l}\}$.

Since all the intervals in $\ell \mathcal{N}_{B'}$ are contained in one of the extensions $\mathcal{J}_{i}$, we have $\sum |\mathcal{J}_{i}| \geq (1 - \beta_{1}/(1 - \eta)) \cdot \ell N$ by Corollary \ref{cor2}. Unfortunately, the leftmost interval $\mathcal{J}_{1}$ will be "useless" to us in case it was generated by case (b) (of course the intervals $\mathcal{J}_{2},\ldots,\mathcal{J}_{l}$, if any exist, were generated by case (a)). We will next argue that $\mathcal{J}_{1}$ is so short in this case that its removal makes virtually no difference for the sum $\sum |\mathcal{J}_{i}|$. 

If the interval $\mathcal{J}_{1}$ was indeed generated by case (b), then $\mathcal{J}_{1} = \{0,\ldots,h\}$ for some $h \in \{0,\ldots,\ell N - 1\}$. Writing $r := 2^{-m(h + 1)} \geq \delta$, then
\begin{equation}\label{form10a} |B'|_{r} \leq R_{B'}^{m}(\mathcal{J}_{1}) < 2^{\zeta m |\mathcal{J}_{1}|} = 2^{\zeta m(h + 1)} = r^{-\zeta}. \end{equation}
If $r \geq \delta^{\epsilon_{0}} = 2^{-\epsilon_{0}\ell m N}$, then
\begin{equation}\label{form57} |\mathcal{J}_{1}| = h + 1 \leq \epsilon_{0} \cdot \ell N \stackrel{\eqref{form58}}{\leq} \eta \cdot \ell N. \end{equation}
On the other hand, if $\delta \leq r < \delta^{\epsilon_{0}}$, then the Frostman condition for $B$ yields
\begin{displaymath} |B' \cap B(x,r)| \leq |B \cap B(x,r)| \leq r^{\kappa}|B| \leq r^{\kappa} \cdot \delta^{-\eta} |B'|, \quad x \in \R, \end{displaymath}
and consequently
\begin{displaymath} r^{-\zeta} \stackrel{\eqref{form10a}}{>} |B'|_{r} \geq r^{-\kappa} \cdot \delta^{\eta}. \end{displaymath}
This yields $r^{\kappa - \zeta} \geq \delta^{\eta}$, and since $\zeta < \kappa/2$ by \eqref{zeta}, we have $r \geq \delta^{2\eta/\kappa}$. Recalling that $r = 2^{-m(h + 1)}$ and $\delta = 2^{-\ell mN}$, this rearranges to 
\begin{displaymath}  |\mathcal{J}_{1}| = h + 1 \leq \tfrac{2\eta}{\kappa} \cdot \ell N. \end{displaymath}
Combining this estimate with \eqref{form57}, we reach the following conclusion: if $\mathcal{N}_{+}$ stands for the intervals among $\{\mathcal{J}_{1},\ldots,\mathcal{J}_{l}\}$ which were generated by case (a), then
\begin{equation}\label{form11} \sum_{\mathcal{J} \in \mathcal{N}_{+}} |\mathcal{J}| \geq \sum_{i = 1}^{l} |\mathcal{J}_{i}| - |\mathcal{J}_{1}| \geq (1 - \tfrac{\beta_{1}}{1 - \eta} - \tfrac{2\eta}{\kappa}) \cdot \ell N. \end{equation}
To recap, the intervals in $\mathcal{N}_{+}$ are subsets of $[\ell N]$, they are roughly compositions of a few intervals in $\ell \mathcal{N}_{B'}$, plus a little extra, and 
\begin{equation}\label{form12} 2^{\zeta m|\mathcal{J}|} \leq R_{B'}^{m}(\mathcal{J}) \stackrel{\eqref{form10}}{\leq} 2^{m(2\zeta)|\mathcal{J}|}, \qquad \mathcal{J} \in \mathcal{N}_{+}. \end{equation}

\subsection{Branching of $A'$ on typical intervals in $\mathcal{N}_{+}$} In this section we are concerned with upper bounding the numbers $R_{A'}^{m}(\mathcal{J})$ for $\mathcal{J} \in \mathcal{N}_{+}$. We already sketched these computations in Section \ref{s:outline}. Recall that $\gamma > (\alpha - \beta)/(1 - \beta)$ is one of the parameters specified in the statement of Theorem \ref{mainTechnical4}. Write
\begin{equation}\label{gammaA} \Gamma := \tfrac{1}{2} \cdot \tfrac{\alpha - \beta}{1 - \beta} + \tfrac{1}{2} \cdot \gamma \in ((\alpha - \beta)/(1 - \beta),\gamma), \end{equation}
and decompose $\mathcal{N}_{+} = \mathcal{N}_{+}^{\textup{low}} \cup \mathcal{N}_{+}^{\textup{high}}$, where 
\begin{equation}\label{form29} \mathcal{N}_{+}^{\textup{low}} := \{\mathcal{J} \in \mathcal{N}_{+} : R_{A'}^{m}(\mathcal{J}) \leq 2^{\Gamma m |\mathcal{J}|}\} \quad \text{and} \quad \mathcal{N}_{+}^{\textup{high}} := \mathcal{N}_{+} \, \setminus \, \mathcal{N}_{+}^{\textup{low}}. \end{equation}
We remark that $\Gamma \geq \gamma/2 > 0$, since $\beta \leq \alpha$.

 We claim that the total length of intervals in $\mathcal{N}_{+}^{\textup{low}}$ must be reasonably large; in the sequel, these will be the only "useful" intervals for us.  More precisely, if $\eta,\zeta > 0$ are sufficiently small (as chosen in \eqref{parameters}), then
\begin{equation}\label{form30} \sum_{\mathcal{J} \in \mathcal{N}_{+}^{\mathrm{low}}} |\mathcal{J}| \geq \tfrac{1}{2} [(1 - \beta) - (\alpha - \beta)/\Gamma] \cdot \ell N. \end{equation}
Note that $(1 - \beta) - (\alpha - \beta)/\Gamma > 0$, since $\Gamma > (\alpha - \beta)/(1 - \beta)$. To prove \eqref{form30}, set
\begin{displaymath} \mathfrak{N} := \cup \mathcal{N}_{+} \cup \{s \in [\ell N] : R_{B'}^{m}(s) = 1\} \subset [\ell N]. \end{displaymath}
We start by claiming that
\begin{equation}\label{form1} |[\ell N] \, \setminus \, \mathfrak{N}| \geq (\beta_{1} - \eta - 2\zeta) \cdot \ell N. \end{equation}
To see this, start with the estimate
\begin{equation}\label{form14} \delta^{\eta - \beta_{1}} \leq |B'| \leq \prod_{s \in \mathfrak{N}} R_{B'}^{m}(s) \cdot \prod_{s \in [\ell N] \, \setminus \, \mathfrak{N}} 2^{m} = 2^{m |[\ell N] \, \setminus \mathfrak{N}|} \cdot \prod_{s \in \mathfrak{N}} R_{B'}^{m}(s). \end{equation}
The last factor can further be decomposed to those indices "$s$" with $R_{B'}^{m}(s) = 1$ (which in total contribute "$1$" to the product), and then a product over the intervals $\mathcal{J} \in \mathcal{N}_{+}$:
\begin{equation}\label{form15} \prod_{s \in \mathfrak{N}} R_{B'}^{m}(s) = \prod_{\mathcal{J} \in \mathcal{N}_{+}} R_{B'}^{m}(\mathcal{J}) \stackrel{\eqref{form12}}{\leq} \prod_{\mathcal{J} \in \mathcal{N}_{+}} 2^{m(2\zeta)|\mathcal{J}|} \leq 2^{m(2\zeta)\ell N}. \end{equation}
Recalling that $\delta = 2^{-\ell m N}$, and combining \eqref{form14}-\eqref{form15}, leads to
\begin{displaymath} 2^{m|[\ell N] \, \setminus \mathfrak{N}|} \geq 2^{\ell m N(\beta_{1} - \eta - 2\zeta)}, \end{displaymath}
which is equivalent to \eqref{form1}.

We continue with the proof of \eqref{form30}. Note that, by the $(\eta,m,\ell N)$-polarisation of $(A',B')$, we have $R_{A'}^{m}(s) \geq 2^{(1 - \eta)m}$ for all $m \in [\ell N] \, \setminus \, \mathfrak{N} \subset \{s \in [\ell N] : R_{B'}^{m}(s) > 1\}$. Consequently,
\begin{equation}\label{form44} \delta^{-\alpha} \geq |A'| = \prod_{s \in \mathfrak{N}} R_{A'}^{m}(s) \cdot \prod_{s \in [\ell N] \, \setminus \, \mathfrak{N}} 2^{(1 - \eta)m} \geq 2^{(1 - \eta)m|[\ell N] \, \setminus \, \mathfrak{N}|} \cdot \prod_{\mathcal{J} \in \mathcal{N}_{+}} R_{A'}^{m}(\mathcal{J}).  \end{equation} 
For the first factor, we will derive a lower bound from \eqref{form1}. Regarding the second factor, recall the high and low branching families from \eqref{form29}, and write $H := \cup \mathcal{N}_{+}^{\mathrm{high}}$. Then, 
\begin{displaymath} \prod_{\mathcal{J} \in \mathcal{N}_{+}} R_{A'}^{m}(\mathcal{J}) \geq \prod_{\mathcal{J} \in \mathcal{N}_{+}^{\mathrm{high}}} 2^{\Gamma m |\mathcal{J}|} = 2^{\Gamma m |H|}. \end{displaymath} 
Consequently, combining \eqref{form44} with \eqref{form1} and the estimate above, we find that
\begin{displaymath} 2^{\alpha \ell m N} \geq 2^{(1 - \eta)(\beta_{1} - \eta - 2\zeta) \ell m N} \cdot 2^{\Gamma m|H|}, \end{displaymath}
or equivalently
\begin{displaymath} |H| \leq \frac{\alpha - (1 - \eta)(\beta_{1} - \eta - 2\zeta)}{\Gamma} \cdot \ell N. \end{displaymath}
If $\eta,\zeta = 0$, then we would have just shown that $|H| \leq \Gamma^{-1}(\alpha - \beta_{1})\ell N$. Since, on the other hand, the intervals in $\mathcal{N}_{+}$ have total length at least $(1 - \beta_{1}) \ell N$ by \eqref{form11} (still assuming $\eta,\zeta = 0$), we may conclude that the intervals in $\mathcal{N}_{+}^{\mathrm{low}}$ have total length at least $[(1 - \beta_{1}) - (\alpha - \beta_{1})/\Gamma] \cdot \ell N$. Finally, if the parameters $\eta,\zeta$ are chosen appropriately, more precisely as in \eqref{parameters}, then the slightly weaker estimate \eqref{form30} holds, namely
\begin{displaymath} \sum_{\mathcal{J} \in \mathcal{N}_{+}^{\mathrm{low}}} |\mathcal{J}| \geq \tfrac{1}{2} \cdot [(1 - \beta_{1}) - (\alpha - \beta_{1})/\Gamma] \cdot \ell N \geq \tfrac{1}{2} \cdot [(1 - \beta) - (\alpha - \beta)/\Gamma] \cdot \ell N. \end{displaymath} 
The final inequality only uses $\beta \leq \beta_{1}$ and $\Gamma \leq \gamma \leq 1$.

\subsection{Pruning $B'$ to improve separation II}\label{s:pruningII} Fix an interval $\mathcal{J} = \{t - r,t - r + 1,\ldots,t\} \in \mathcal{N}_{+}$, as defined above \eqref{form12}. Then, \eqref{form12} means that if $I \in \mathcal{I}_{m(t - r)}(B')$ is a fixed interval of length $2^{-m(t - r)}$ intersecting $B'$, we have $2^{\zeta m r} \leq |\{J \in \mathcal{I}_{m(t + 1)}(B') : J \subset I\}| \leq 2^{(2\zeta)mr}$. How well are these intervals $J \subset I$ separated? By \eqref{form8}, we already know that any two distinct intervals in $\mathcal{I}_{m(t + 1)}(B')$ are separated by at least $2^{-m(t + 1)}$, but this is far too weak for our purposes: for purposes to become apparent later, we would like the intervals $J$ to be closer to $2^{-m(t - r)}$-separated, and the only control for "$r$" we have is the lower bound $r \geq \ell$ (recalling that each interval in $\mathcal{N}_{+}$ contains an interval in $\ell \mathcal{N}_{B'}$). 

The better separation is "morally true" for the following reason: the interval $\mathcal{J}$ was created by combining levels with \emph{almost} trivial branching, until \emph{roughly} the first moment we saw some non-trivial branching. If the words "almost" and "roughly" could be omitted, we would be done: then each interval $I \in \mathcal{I}_{m(t + 1)}(B')$ would be a "single child" of its parent in $\mathcal{I}_{m(t - r)}(B')$, and since the intervals in $\mathcal{I}_{m(t - r)}(B')$ are $2^{-m(t - r)}$-separated by \eqref{form8}, the same would be true of the intervals in $\mathcal{I}_{m(t + 1)}(B')$.

The words "almost" and "roughly" cannot be omitted, so we need to force the separation by trimming $B'$ to a further $(\eta,m,\ell N)$-uniform subset $B'' \subset B'$. Write
\begin{equation}\label{defEpsilon} \xi := (\gamma - \Gamma)/4, \end{equation}
where $\Gamma$ was defined in \eqref{gammaA}. In particular, $\xi \gtrsim_{\alpha,\beta,\gamma} 1$. We also impose the following additional condition on the constant $\ell \in \N$ selected at \eqref{allParameters}:
\begin{equation}\label{form27} \ell \geq \xi^{-1} = \tfrac{4}{\gamma - \Gamma}. \end{equation}
Recall that $\mathcal{J} = \{t - r,\ldots,t\} \in \mathcal{N}_{+}$ was the shortest extension (to the left) of a certain interval $\mathcal{J}_{0} \in \ell \mathcal{N}_{B'}$ with the property $R_{B'}^{m}(\mathcal{J}) \geq 2^{\zeta m |\mathcal{J}|} = 2^{\zeta m (r + 1)}$. Consequently, the subinterval $\mathcal{J}_{\xi} = \{t - \floor{(1 - 2\xi) r},\ldots,t\}$ does not yet have this property, that is,
\begin{equation}\label{form31} R_{B'}^{m}(\mathcal{J}_{\xi}) < 2^{\zeta m |\mathcal{J}_{\xi}|} = 2^{\zeta m(\floor{(1 - 2\xi)r} + 1)}. \end{equation}
Here we used that $\floor{(1 - 2\xi)r} < r$, which is true because $\xi r \geq \xi \ell \geq 1$ by \eqref{form27}. 

Now, it follows from a combination of \eqref{form31}, and $R_{B'}^{m}(\mathcal{J}) \geq 2^{\zeta m |\mathcal{J}|} = 2^{\zeta m(r + 1)}$, that
\begin{equation}\label{form32} R_{B'}^{m}(\mathcal{J} \, \setminus \, \mathcal{J}_{\xi}) \geq 2^{\zeta m (r - \floor{(1 - 2\xi)r})} \geq 2^{2\xi \zeta m r} \geq 2^{\xi \zeta m(r + 1)} = 2^{\xi \zeta m|\mathcal{J}|}. \end{equation} 
We are then prepared to define the desired subset $B'' \subset B'$. Let $\mathcal{S}_{\xi} := \cup \{\mathcal{J}_{\xi} : \mathcal{J} \in \mathcal{N}_{+}\}$, and apply the "collapsing" Lemma \ref{ShLemma3} to the $(m,\ell N)$-uniform set $B'$, and the set of scales $\mathcal{S}_{\xi} \subset [\ell N]$. The product is an $(m,\ell N)$-uniform subset $B'' \subset B'$ 
such that
\begin{displaymath} R_{B''}^{m}(s) = \begin{cases} R_{B'}^{m}(s), & s \notin \mathcal{S}_{\xi}, \\ 1, & s \in \mathcal{S}_{\xi}. \end{cases} \end{displaymath}
In particular, $R_{B''}^{m}(s) = R_{B'}^{m}(s)$ for all $s \in \mathcal{J} \, \setminus \, \mathcal{J}_{\xi}$, for $\mathcal{J} \in \mathcal{N}_{+}$, so \eqref{form32} remains valid for the set $B''$:
\begin{equation}\label{form33} R_{B''}^{m}(\mathcal{J}) \geq R_{B'}^{m}(\mathcal{J} \, \setminus \, \mathcal{J}_{\xi}) \geq 2^{\xi \zeta m |\mathcal{J}|}. \end{equation}
In fact, the first inequality is an equation, since $R_{B''}^{m}(s) = 1$ for all $s \in \mathcal{J}_{\xi} \subset \mathcal{S}_{\xi}$. Curiously, we will have no use for a "global" lower bound for $|B''|$, although it would be easy to deduce from \eqref{form12} that $|B''| \geq \delta^{2\zeta} \cdot |B'|$. From now on, only the "local" branching estimate \eqref{form33} will be needed, and "global" lower bound $|B'| \gtrapprox \delta^{-\beta}$ has already been fully exploited in previous sections (where the relation between $\gamma,\alpha$ and $\beta$ appeared).

The point of reducing $B'$ to $B''$ was to improve the $2^{-m(t + 1)}$-separation of distinct intervals $I \in \mathcal{I}_{m(t + 1)}(B')$ to something resembling $2^{-m(t - r)}$-separation. This has now been accomplished. More precisely, assume that $\mathcal{J} = \{t - r,\ldots,t\} \in \mathcal{N}_{+}$, let $I \in \mathcal{I}_{m(t - r)}(B'')$, and and let $I_{1},I_{2} \in \mathcal{I}_{m(t + 1)}(B'')$ be distinct. Then, since 
\begin{displaymath} R_{B''}^{m}(s) = 1, \qquad s \in \mathcal{J}_{\xi} = \{t - \floor{(1 - \xi)r},\ldots,t\}, \end{displaymath}
the intervals $I_{1},I_{2}$ are contained inside distinct intervals 
\begin{displaymath} \hat{I}_{1},\hat{I}_{2} \in \mathcal{I}_{m(t - \floor{(1 - \xi)r})}(B'') \subset \mathcal{I}_{m(t - \floor{(1 - \xi)r})}(B'). \end{displaymath}
Consequently, using also that $\floor{(1 - \xi)r} \geq (1 - \xi)r - 1$, and $\xi r \geq \xi \ell \geq 1$,
\begin{align} \dist(I_{1},I_{2}) \geq \dist(\hat{I}_{1},\hat{I}_{2}) & \stackrel{\eqref{form8}}{\geq} 2^{-m(t - \floor{(1 - \xi)r})} \notag\\
& \geq 2^{-\xi m r - m} \cdot 2^{-m(t - r)} \notag\\
&\label{form34} \geq 2^{-2\xi m(r + 1)} \cdot 2^{-m(t - r)}. \end{align} 
Inequality \eqref{form34} is more clearly phrased in the following way:
\begin{lemma}\label{lemma5} Let $\mathcal{J} = \{t - r,\ldots,t\} \in \mathcal{N}_{+}$, $\Delta_{\mathcal{J}} := 2^{-m(t - r)}$, and $\delta_{\mathcal{J}} := 2^{-m(t + 1)}$. Let $I \in \mathcal{I}_{m(t - r)}(B'')$ be a dyadic interval of length $\Delta_{\mathcal{J}}$ intersecting $B''$, and let $I_{1},I_{2} \in \mathcal{I}_{m(t + 1)}(B'')$ be distinct with $I_{1},I_{2} \subset I$. Then,
\begin{displaymath} \dist(I_{1},I_{2}) \geq \left(\frac{\delta_{\mathcal{J}}}{\Delta_{\mathcal{J}}} \right)^{2\xi} \cdot |\Delta_{\mathcal{J}}|. \end{displaymath} \end{lemma}
\begin{proof} Observing that $\delta_{\mathcal{J}}/\Delta_{\mathcal{J}} = 2^{-m(r + 1)}$, this inequality is just a rewording of \eqref{form34}. \end{proof}

\subsection{Elementary projection estimates} The plan is to prove lower bounds for $|A' + cB''|_{\delta}$ by, roughly speaking, establishing separately lower bounds for $|(A' \cap I) + c(B'' \cap J)|_{\delta}$, where $I,J \subset [0,1)$ are suitable dyadic intervals intersecting $A',B''$, and then combining the results. In this section, we will prove an auxiliary result which will imply the required lower bounds for $|(A' \cap I) + c(B'' \cap J)|_{\delta}$. To be more accurate, instead of proving lower bounds for $|(A' \cap I) + c(B'' \cap J)|_{\delta}$ directly, we prove (stronger) lower bounds for the entropies of suitable measures supported on $(A' \cap I) + c(B'' \cap J)$ (see \eqref{form47}). This is (only!) done for the reason that such "multi-scale" information about entropy is cleaner to combine than "multi-scale" information about cardinalities.

We introduce the following notation. Dyadic cubes in $\R^{d}$ of side-length $2^{-n}$ are denoted $\mathcal{D}_{n}$. If $\mu$ is a Borel probability measure on $\R^{d}$, and $n \in \N$, we write
\begin{displaymath} \mu^{(n)} := \sum_{Q \in \mathcal{D}_{n}} \frac{\mu(Q)}{\mathcal{L}^{d}(Q)} \cdot \mathcal{L}^{d}|_{Q}. \end{displaymath}
Thus $\mu^{(n)}$ is a "$2^{-n}$-discretisation of $\mu$". Note that $\mu^{(n)} \in L^{2}(\R^{d}) \cap L^{\infty}(\R^{d})$. We also define the projections $\pi_{c}(x,y) := x + cy$ for $(x,y) \in \R^{2}$ and $c \in \R$.

\begin{lemma}\label{lemma3} Let $\Delta = 2^{-n} \in 2^{-\N}$, and let $\gamma,\gamma_{A},\gamma_{B} \in (0,1]$, and $\mathbf{C} \geq 1$. Let $\mathcal{A},\mathcal{B} \subset \mathcal{D}_{n}$ be collections of dyadic $\Delta$-intervals with $|\mathcal{A}| = \Delta^{-\gamma_{A}}$ and $|\mathcal{B}| = \Delta^{-\gamma_{B}}$. We assume the following separation from $\mathcal{B}$, for some $\xi \in (0,1]$:
\begin{equation}\label{form36} \dist(I_{1},I_{2}) \geq \Delta^{\xi} \quad \text{for distinct} \quad I_{1},I_{2} \in \mathcal{B}. \end{equation}
Let
\begin{itemize}
\item Let $\mu$ be a probability measure with $\spt \mu \subset (\cup \mathcal{A}) \times (\cup \mathcal{B})$ with the property that $\mu(Q) \leq \mathbf{C}\Delta^{\gamma_{A} + \gamma_{B}}$ for $Q \in \mathcal{D}_{n}$. 
\item Let $\nu$ be a probability measure on $[-1,1]$ such that $\nu(I) \leq \mathbf{C}\Delta^{\gamma}$ for all $I \in \mathcal{D}_{n}$. 
\end{itemize}
Then,
\begin{equation}\label{form37} \int_{-1}^{1} \|(\pi_{c}\mu)^{(n)}\|_{2}^{2} \, d\nu(c) \lesssim \mathbf{C} \cdot \max\{\Delta^{\gamma_{A} + \gamma_{B} - 1},\Delta^{\gamma - 1 - \xi}\}. \end{equation} 
\end{lemma}

\begin{remark} To help interpreting the upper bound \eqref{form37}, let us mention the "trivial" estimate $\|(\pi_{c}\mu)^{(n)}\|_{2}^{2} \lesssim \Delta^{\gamma_{A} - 1}$ for \textbf{every} $c \in [0,1)$. This could be deduced rather easily from \eqref{form38} below. Therefore, \eqref{form37} beats the trivial bound whenever $\gamma > \gamma_{A} + \xi$.  \end{remark}

\begin{proof}[Proof of Lemma \ref{lemma3}] By definition,
\begin{displaymath} (\pi_{c}\mu)^{(n)} = \sum_{I \in \mathcal{D}_{n}} \frac{\pi_{c}\mu(I)}{\Delta} \cdot \mathcal{L}^{1}|_{I} = \sum_{I \in \mathcal{D}_{n}} \frac{\mu(\pi_{c}^{-1}(I))}{\Delta} \cdot \mathcal{L}^{1}|_{I}. \qquad c \in [-1,1]. \end{displaymath}
Consequently,
\begin{equation}\label{form38} \|(\pi_{c}\mu)^{(n)}\|_{2}^{2} = \sum_{I \in \mathcal{D}_{n}} \left(\frac{\mu(\pi_{c}^{-1}(I))}{\Delta} \right)^{2} \cdot \Delta = \frac{1}{\Delta} \cdot \sum_{I \in \mathcal{D}_{n}} (\mu \times \mu)(\{(p,q) : p,q \in \pi_{c}^{-1}(I)\}). \end{equation}
Therefore,
\begin{align*} \Delta \cdot \int_{-1}^{1} \|(\pi_{c}\mu)^{(n)}\|_{2}^{2} \, d\nu(c) & \sim \int_{-1}^{1} \sum_{I \in \mathcal{D}_{n}} (\mu \times \mu)(\{(p,q) : p,q \in \pi_{c}^{-1}(I)\}) \, d\nu(c)\\
& = \iint \int \sum_{I \in \mathcal{D}_{n}} \mathbf{1}_{\{p,q \in \pi_{c}^{-1}(I)\}}(c) \, d\nu(c) \, d\mu(p) \, d\mu(q). \end{align*}
We split the outer integration into 
\begin{displaymath} \Omega_{\mathrm{near}} := \{(p,q) : |p - q| < 10\Delta\} \quad \text{and} \quad \Omega_{\mathrm{far}} := \{(p,q) : |p - q| \geq 10\Delta\}. \end{displaymath}
Regarding $\Omega_{\mathrm{near}}$, we only use the observation that if $p,q \in A \times B$ and $c \in [0,1]$ are fixed, then there is at most one interval $I \in \mathcal{D}_{n}$ such that $p,q \in \pi_{c}^{-1}(I)$. Since $\mu,\nu$ are probability measures, and $\mu(B(x,10\Delta)) \lesssim \mathbf{C}\Delta^{\gamma_{A} + \gamma_{B}}$ for every $x \in \R^{2}$, this leads to
\begin{equation}\label{form35} \iint_{\Omega_{\mathrm{near}}} \sum_{I \in \mathcal{D}_{n}} \mathbf{1}_{\{p,q \in \pi_{c}^{-1}(I)\}}(c) \, d\nu(c) \, d\mu(p) \, d\mu(q) \lesssim (\mu \times \mu)(\Omega_{\mathrm{near}}) \lesssim \mathbf{C}\Delta^{\gamma_{A} + \gamma_{B}}. \end{equation}
We then consider integral over the domain $\Omega_{\mathrm{far}}$. A basic, easy to verify, observation is this: if $p,q \in \R^{2}$ are fixed and distinct, then the set
\begin{displaymath} I(p,q) := \{c \in [0,1] : p,q \in \pi_{c}^{-1}(I) \text{ for some } I \in \mathcal{D}_{n}\} \end{displaymath}
is contained in an interval of length $\lesssim \Delta/|p - q|$, and in particular can be covered by $\lesssim |p - q|^{-1}$ dyadic intervals of length $\Delta$.

We combine this with the following additional observation. Note that all the tubes $\pi_{c}^{-1}(I)$ make an angle $\leq \pi/4$ with the $y$-axis (this is attained for $c = 1$, and for $c = 0$, the tubes $\pi_{c}^{-1}(I)$ are vertical). Therefore, 
\begin{displaymath} p,q \in A \times B,\,  |p - q| \geq 10\Delta \text{ and } \exists \, c \in [-1,1] \text{ s.t. } p,q \in \pi_{c}^{-1}(I) \quad \Longrightarrow \quad |p_{y} - q_{y}| > \Delta. \end{displaymath}
Here $p_{y},q_{y} \in B$ refer to the second coordinates of $p,q$. Namely, if $|p - q| \geq 10\Delta$ and $|p_{y} - q_{y}| < \Delta$, then $|p_{x} - q_{x}| \geq 9\Delta$, which makes the pair $p,q$ too "horizontal" to be contained in any common tube $\pi_{c}^{-1}(I)$, with $c \in [-1,1]$ and $I \in \mathcal{D}_{n}$. Now, recalling our assumption \eqref{form36} that $\dist(I_{1},I_{2}) \geq \Delta^{\xi}$ for distinct $I_{1},I_{2} \in \mathcal{B}$, the conclusion $|p_{y} - q_{y}| > \Delta$ can be amplified substantially: $|p_{y} - q_{y}| > \Delta$ implies that $p_{y},q_{y}$ lie in distinct intervals in $\mathcal{B}$, hence $|p - q| \geq |p_{y} - q_{y}| \geq \Delta^{\xi}$. Therefore:
\begin{displaymath} \iint_{\Omega_{\mathrm{far}}} \int \sum_{I \in \mathcal{D}_{n}} \mathbf{1}_{\{p,q \in \pi_{c}^{-1}(I)\}}(c) \, d\nu(c) \, d\mu(p) \, d\mu(q) = \iint_{\Omega_{\mathrm{FAR}}} \int_{I(p,q)} \ldots \, d\nu(c) \, d\mu(p) \, d\mu(q),  \end{displaymath}
with $\Omega_{\mathrm{FAR}} = \{(p,q) : |p - q| \geq \Delta^{\xi}\}$. Now, for every pair $(p,q) \in \Omega_{\mathrm{FAR}}$, we note that the set $I(p,q) \subset [0,1]$ can be covered by $\lesssim |p - q|^{-1} \leq \Delta^{-\xi}$ dyadic intervals of length $\Delta$, and for each $c \in I(p,q)$, there is exactly one $I \in \mathcal{D}_{n}$ such that $p,q \in \pi_{c}^{-1}(I)$. Therefore,
\begin{displaymath} \int_{I(p,q)} \sum_{I \in \mathcal{D}_{n}} \mathbf{1}_{\{p,q \in \pi_{c}^{-1}(I)\}}(c) \, d\nu(v) = \nu(I(p,q)) \lesssim \mathbf{C}\Delta^{\gamma - \xi}, \qquad (p,q) \in \Omega_{\mathrm{FAR}}, \end{displaymath}
and consequently
\begin{displaymath} \iint_{\Omega_{\mathrm{far}}} \int \sum_{I \in \mathcal{D}_{n}} \mathbf{1}_{\{p,q \in \pi_{c}^{-1}(I)\}}(c) \, d\nu(c) \, d\mu(p) \, d\mu(q) \lesssim \mathbf{C}\Delta^{\gamma - \xi}. \end{displaymath}
Combining this estimate with \eqref{form35}, we arrive at \eqref{form37}. \end{proof}
 
We will next deduce, as a corollary, an entropy version of Lemma \ref{lemma3}. For this purpose, we record the following \cite[Lemma 3.6]{MR3940442} by Shmerkin:
\begin{lemma}\label{lemma4} Let $\mu$ be a Borel probability measure on $\R^{d}$. The following relation holds between the $\mathcal{D}_{n}$-entropy $H(\mu,\mathcal{D}_{n})$ of $\mu$, and the $L^{2}$-norm of $\mu^{(n)}$:
\begin{equation}\label{form42} H(\mu,\mathcal{D}_{n}) \geq dn - \log \|\mu^{(n)}\|_{2}^{2}. \end{equation} 
\end{lemma}

Here, and below, "$\log$" refers to logarithm in base $2$.

\begin{cor}\label{cor3} Let $\Delta = 2^{-n} \in 2^{-\N}$, and assume that $\mathcal{A},\mathcal{B},\mu,\nu,\gamma_{A},\gamma_{B},\gamma,\mathbf{C}$, and $\xi$ have the same meaning as in Lemma \ref{lemma3}. Then, 
\begin{equation}\label{form39} \int_{-1}^{1} H(\pi_{c}\mu,\mathcal{D}_{n}) \, d\nu(c) \geq n \cdot \min\{\gamma_{A} + \gamma_{B},\gamma - \xi\} - \log \mathbf{C} - \log C_{0}, \end{equation}
where $C_{0} > 0$ is an absolute constant. 
\end{cor} 

\begin{proof} First combine \eqref{form42} (with $d = 1$) and Jensen's inequality to deduce that
\begin{displaymath} \int_{-1}^{1} H(\pi_{c}\mu,\mathcal{D}_{n}) \, d\nu(c) \geq n -  \int_{-1}^{1} \log \|(\pi_{c}\mu)^{(n)}\|_{2}^{2} \, d\nu(c) \geq n - \log \left( \int_{0}^{1} \|(\pi_{c}\mu)^{(n)}\|_{2}^{2} \, d\nu(c) \right). \end{displaymath}
Here,
\begin{displaymath} \int_{-1}^{1} \|(\pi_{c}\mu)^{(n)}\|_{2}^{2} \, d\nu(c) \leq C_{0}\mathbf{C}\max\{2^{n(1 - \gamma_{A} - \gamma_{B})},2^{n(1 + \xi - \gamma)}\} \end{displaymath}
for some absolute constant $C_{0} > 0$, by Lemma \ref{lemma3}. These inequalities give \eqref{form39}. \end{proof}

\subsection{Projecting pieces of $A' \times B''$} We next put Corollary \ref{cor3} to work in our "real-world" situation. We recall the following notation from Section \ref{s:notation}. Assume that $\mu$ is a Borel probability measure on $\R^{d}$ (we will use this for both $d = 1$ and $d = 2$), and let $Q \in \mathcal{D}_{n}$ be a dyadic cube of side-length $2^{-n}$ such that $\mu(Q) > 0$. Let $T_{Q} \colon Q \to [0,1)^{d}$ be the rescaling map with $T_{Q}(Q) = [0,1)^{d}$. We define the measures
\begin{equation}\label{form45} \mu_{Q} := \tfrac{1}{\mu(Q)} \cdot \mu|_{Q} \quad \text{and} \quad \mu^{Q} := T_{Q}\mu_{Q}. \end{equation}
In this section, $\mu = \mu_{A'} \times \mu_{B''}$, where $\mu_{A'}$ is the normalised counting measure on $A'$, and $\mu_{B''}$ is the normalised counting measure on $B''$ (defined in Section \ref{s:pruningII}). For $s \in [\ell N]$, we will write 
\begin{displaymath} \mathcal{D}_{ms}(\mu) := \{I \times J : I \in \mathcal{I}_{ms}(A') \text{ and } J \in \mathcal{I}_{ms}(B'')\} = \{Q \in \mathcal{D}_{ms} : \mu(Q) > 0\}. \end{displaymath} 

Fix 
\begin{displaymath} \mathcal{J} = \{t - r,\ldots,t\} \in \mathcal{N}_{+}^{\mathrm{low}}. \end{displaymath}
As defined in \eqref{form29}, this means that $R_{A'}^{m}(\mathcal{J}) \leq 2^{\Gamma m|\mathcal{J}|}$, where $\Gamma \in ((\alpha - \beta)/(1 - \beta),\gamma) \subset [\gamma/2,\gamma)$ was the parameter specified in \eqref{gammaA}. For now, it is only important to remember that $\gamma - \Gamma \gtrsim_{\alpha,\beta,\gamma} 1$. Fix intervals $I_{0} \in \mathcal{I}_{m(t - r)}(A')$ and $J_{0} \in \mathcal{I}_{m(t - r)}(B'')$. Write
\begin{displaymath} \mathcal{A}_{I_{0}} := \{I' \in \mathcal{I}_{m(t + 1)}(A') : I' \subset I_{0}\} \quad \text{and} \quad \mathcal{B}_{J_{0}} := \{J' \in \mathcal{I}_{m(t + 1)}(B'') : J' \subset J_{0}\}. \end{displaymath}
Then
\begin{equation}\label{form43} |\mathcal{A}_{I_{0}}| = R_{A'}^{m}(\mathcal{J}) \leq 2^{\Gamma m|\mathcal{J}|} \quad \text{and} \quad |\mathcal{B}_{J_{0}}| = R_{B''}^{m}(\mathcal{J}) \stackrel{\eqref{form33}}{\geq} 2^{\xi \zeta m|\mathcal{J}|}. \end{equation}
In particular, we may write
\begin{equation}\label{form48} R_{A'}^{m}(\mathcal{J}) = |\mathcal{A}_{I_{0}}| = 2^{\gamma_{A}m|\mathcal{J}|} \quad \text{and} \quad |\mathcal{B}_{J_{0}}| = 2^{\gamma_{B}m|\mathcal{J}|} \end{equation}
for some $0 \leq \gamma_{A} \leq \Gamma$ and $\gamma_{B} \geq \xi \zeta$. Then, write 
\begin{displaymath} Q_{0} := I \times J \in \mathcal{D}_{m(t - r)}(\mu), \, n := m|\mathcal{J}| = m(r + 1) \quad \text{and} \quad \Delta := 2^{-n} = \frac{2^{-m(t + 1)}}{2^{-m(t - r)}} \geq \delta. \end{displaymath}
Consider the normalised measure $\mu^{Q_{0}}$, as in \eqref{form45}. The measure $\mu^{Q_{0}}$ is supported on a product of the form $(\cup \mathcal{A}) \times (\cup \mathcal{B}) = T_{Q_{0}}((\cup \mathcal{A}_{I_{0}}) \times (\cup \mathcal{B}_{J_{0}}))$, where $\mathcal{A},\mathcal{B}$ are the families of $\Delta$-intervals obtained by normalising the intervals in $\mathcal{A}_{I_{0}}$ and $\mathcal{B}_{J_{0}}$ by a factor of $2^{m(t - r)}$. It follows from the $(m,\ell N)$-uniformity of $A'$ and $B''$ that 
\begin{displaymath} \mu^{Q_{0}}(Q) \leq (|\mathcal{A}||\mathcal{B}|)^{-1}  = (|\mathcal{A}_{I_{0}}||\mathcal{B}_{J_{0}}|)^{-1} \stackrel{\eqref{form48}}{=} \Delta^{\gamma_{A} + \gamma_{B}}, \qquad Q \in \mathcal{D}_{n}. \end{displaymath}
Moreover, the intervals in $\mathcal{B}$ satisfy the following separation property by Lemma \ref{lemma5}:
\begin{displaymath} I_{1},I_{2} \in \mathcal{B}, \, I_{1} \neq I_{2} \quad \Longrightarrow \quad \dist(I_{1},I_{2}) \geq \Delta^{2\xi}. \end{displaymath}
These facts place us in a position to apply Corollary \ref{cor3} to the measure $\mu^{Q_{0}}$:
\begin{equation}\label{form46} \int_{-1}^{1} H(\pi_{c}\mu^{Q_{0}},\mathcal{D}_{n}) \, d\nu(c) \geq n \cdot \min\{\gamma_{A} + \gamma_{B},\gamma - 2\xi\} - \log 40 - \log C_{0}. \end{equation} 
The parameter "$\xi$" was initially chosen (see \eqref{defEpsilon}) so that $2\xi \leq (\gamma - \Gamma)/2$. Since $\gamma_{A} \leq \Gamma$, this leads to
\begin{displaymath} \gamma - 2\xi \geq \gamma_{A} + (\gamma - \Gamma) - 2\xi \geq \gamma_{A} + (\gamma - \Gamma)/2. \end{displaymath}
Recalling also that $\gamma_{B} \geq \xi \zeta$ by \eqref{form43}, and $\zeta \in (0,1)$ (see \eqref{zeta} for a reminder), we find
\begin{equation}\label{form49} \min\{\gamma_{A} + \gamma_{B},\gamma - 2\xi\} \geq \min\{\gamma_{A} + \xi \zeta,\gamma_{A} + (\gamma - \Gamma)/2\} = \gamma_{A} + \xi \zeta.  \end{equation}
Before the final conclusion, let us recall that $n = m(r + 1) = m|\mathcal{J}|$, and observe that
\begin{displaymath} \gamma_{A} \cdot n = \gamma_{A} \cdot m|\mathcal{J}| \stackrel{\eqref{form48}}{=} \log R_{A'}^{m}(\mathcal{J}). \end{displaymath}
Therefore, \eqref{form46}-\eqref{form49} yield
\begin{align} \int_{-1}^{1} H(\pi_{c}\mu^{Q_{0}},\mathcal{D}_{m|\mathcal{J}|}) \, d\nu(c) & \geq n \cdot (\gamma_{A} + \xi \zeta) - \log 40 - \log C_{0} \notag\\
&\label{form47} = \log R_{A'}^{m}(\mathcal{J}) + \xi \zeta \cdot m|\mathcal{J}| - \log 40 - \log C_{0} \end{align}
for all $\mathcal{J} = \{t - r,\ldots,t\} \in \mathcal{N}_{+}^{\mathrm{low}}$ and for all $Q_{0} = I \times J \in \mathcal{D}_{m(t - r)}(\mu)$.

\subsection{Final multiscale argument} As in the previous section, let $\mu_{A'}$ be the normalised counting measure on the set $A'$, let $\mu_{B''}$ be the normalised counting measure on the set $B''$, and let $\mu = \mu_{A'} \times \mu_{B''}$. Recall also that $\mathcal{D}_{ms}(\mu) = \{Q \in \mathcal{D}_{ms} : \mu(Q) > 0\}$ for $s \in [\ell N]$. We warn the reader that the notation "$\mathcal{D}_{ms}$" will in this section refer to both dyadic squares in $\R^{2}$, and dyadic intervals in $\R$. The meaning should always be clear from context. 

The purpose fo this section is to show that there exists $c \in \spt(\nu)$ such that
\begin{equation}\label{form50} \tfrac{1}{\ell m N} \cdot H(\pi_{c}\mu,\mathcal{D}_{\ell m N}) \geq \bar{\alpha} + \xi \zeta \cdot \tfrac{1}{2}[(1 - \beta) - (\alpha - \beta)/\Gamma] - 2\eta. \end{equation}
Here $\bar{\alpha}$ was the constant (defined in \eqref{barAlpha}) such that $|A| = \delta^{-\bar{\alpha}}$. The lower bound in \eqref{form50} yields a lower bound for $|A' + cB''|_{\delta}$, and consequently $|A + cB|_{\delta}$: since $H(\pi_{c}\mu,\mathcal{D}_{\ell m N}) \leq \log |A' + cB''|_{\delta} \leq \log |A + cB|_{\delta}$, and $\ell m N = -\log \delta$, we deduce from \eqref{form50} that 
\begin{displaymath} \frac{\log |A + cB|_{\delta}}{-\log \delta} \geq  \bar{\alpha} + \xi \zeta \cdot \tfrac{1}{2}[(1 - \beta) - (\alpha - \beta)/\Gamma] - 2\eta. \end{displaymath}
If $\eta > 0$ is sufficiently small, depending only on $\alpha,\beta,\gamma,\kappa$, this implies $|A + cB|_{\delta} \geq \delta^{-\bar{\alpha} - \eta} = \delta^{-\eta}|A|$. Of course it is important here that the values of $\xi = (\gamma - \Gamma)/4$ (see \eqref{defEpsilon}) and $\zeta > 0$ (see \eqref{zeta}) are independent of $\bar{\alpha}$, although they may depend on $\alpha$. This proves Theorem \ref{mainTechnical4}: either \eqref{counterAss} fails, and $|A + cB|_{\delta} \geq \delta^{-\epsilon}|A|$ with $c = 1 \in \spt(\nu)$, or \eqref{counterAss} holds, and in this case $|A + cB|_{\delta} \geq \delta^{-\eta}|A|$ for the point $c \in \spt(\nu)$ provided by \eqref{form50}.

It remains to prove \eqref{form50}. This will be accomplished by combining \eqref{form47} with the following uniform lower bound:
\begin{lemma}\label{lemma6} Let $\mathcal{J} = \{s,\ldots,t\} \subset [\ell N]$, and let $Q_{0} \in \mathcal{D}_{ms}(\mu)$. Then,
\begin{equation}\label{form41} H(\pi_{c}\mu^{Q_{0}},\mathcal{D}_{m|\mathcal{J}|}) \geq \log R_{A'}^{m}(\mathcal{J}) - 1, \qquad c \in [0,1]. \end{equation} \end{lemma}
\begin{proof} Let $I \in \mathcal{I}_{ms}(A')$ and $J \in \mathcal{I}_{ms}(B'')$ such that $Q_{0} = I \times J$. Then $\mu^{Q_{0}} = \mu_{A'}^{I} \times \mu_{B''}^{J}$, hence $\pi_{c}\mu^{Q} = \mu_{A'}^{I} \ast \mu_{B''}^{J}$, and finally
\begin{equation}\label{form68} H(\pi_{c}\mu^{Q_{0}},\mathcal{D}_{m|\mathcal{J}|}) = H(\mu_{A'}^{I} \ast \mu_{B''}^{J},\mathcal{D}_{m|\mathcal{J}|}) \geq \int H((\mu_{A'}^{I})_{x},\mathcal{D}_{m|\mathcal{J}|}) \, d\mu_{B''}^{J}(x), \end{equation} 
where the inequality follows from the concavity of entropy (we discussed this at \eqref{form84}), and where $(\mu_{A'})^{I}_{x}(H):= \mu_{A'}^{I}(H - x)$ for $H \subset \R$. From the definition of entropy, one has
\begin{displaymath} H((\mu_{A'}^{I})_{x},\mathcal{D}_{m|\mathcal{J}|}) = H(\mu_{A'}^{I},\mathcal{D}_{m|\mathcal{J}|} - x), \end{displaymath}
where $\mathcal{D}_{m|\mathcal{J}|} - x$ refers to the family of $(-x)$-translated dyadic intervals. Now, for $x \in \R$ fixed, every intervals in $\mathcal{D}_{m|\mathcal{J}|}$ can be covered by $2$ intervals in $\mathcal{D}_{m|\mathcal{J}|} - x$ and vice versa. This implies that
\begin{equation}\label{form69} |H(\mu_{A'}^{I},\mathcal{D}_{m|\mathcal{J}|} - x) - H(\mu_{A'}^{I},\mathcal{D}_{m|\mathcal{J}|})| \leq \log 2 = 1, \qquad x \in \R. \end{equation}
Furthermore, by definition,
\begin{displaymath} H(\mu_{A'}^{I},\mathcal{D}_{m|\mathcal{J}|}) = - \sum_{L \in \mathcal{D}_{m|\mathcal{J}|}} \mu_{A'}^{I}(L) \log \mu_{A'}^{I}(L). \end{displaymath}
Since $A'$ is $(m,\ell N,R_{A'}^{m})$-uniform, either $\mu_{A'}^{I}(L) = 0$, or then $\mu_{A'}^{I}(L) = R_{A'}^{m}(\mathcal{J})^{-1}$ for every $L \in \mathcal{D}_{m|\mathcal{J}|}$. Therefore
\begin{displaymath} H(\mu_{A'}^{I},\mathcal{D}_{m|\mathcal{J}}) = R_{A'}^{m}(\mathcal{J})^{-1}. \end{displaymath}
In combination with \eqref{form68}-\eqref{form69}, this yields \eqref{form41}. \end{proof}

Recall the intervals $\mathcal{N}_{+}^{\mathrm{low}} \subset \mathcal{N}_{+}$, defined in \eqref{form29}. In this section, the properties of these intervals will be used via the formula \eqref{form47}, and we additionally need to recall that
\begin{equation}\label{form40} \sum_{\mathcal{J} \in \mathcal{N}_{+}^{\mathrm{low}}} |\mathcal{J}| \geq \tfrac{1}{2} [(1 - \beta) - (\alpha - \beta)/\Gamma]  \cdot \ell N \end{equation}
by \eqref{form30}. Let $\mathcal{P}$ be the partition of $[\ell N]$ which is induced by the intervals in $\mathcal{N}_{+}^{\mathrm{low}} $. In other words, $\mathcal{P}$ consists of the intervals in $\mathcal{N}_{+}^{\mathrm{low}}$, and the maximal complementary intervals. We write
\begin{displaymath} \mathcal{P}_{\mathrm{useless}} := \mathcal{P} \, \setminus \, \mathcal{N}_{+}^{\mathrm{low}},  \end{displaymath}
and we enumerate $\mathcal{P} = \{\mathcal{J}_{1},\mathcal{J}_{2},\ldots,\mathcal{J}_{h}\}$, where $1 \leq h \leq \ell N$. We write $\mathcal{J}_{j} = \{s_{j},\ldots,t_{j}\}$ for $1 \leq j \leq h$, so $s_{1} = 0$, $t_{h} + 1 = \ell N$, and $s_{j + 1} = t_{j} + 1$ for all $1 \leq j < h$. We artificially define $s_{h + 1} := \ell N$, so the relation $s_{j + 1} = t_{j} + 1$ also remains valid for $j = h$.  

We abbreviate
\begin{displaymath} \mathcal{D}_{j} := \mathcal{D}_{ms_{j}}(\mu) := \{I \times J : I \in \mathcal{I}_{ms_{j}}(A') \text{ and } J \in \mathcal{I}_{ms_{j}}(B'')\}. \end{displaymath}
We then apply the entropy lower bound in Lemma \ref{entropyLemma}, and its corollary \eqref{entropyIneq}, to the partition $0 = ms_{1} < \ldots < ms_{h} < ms_{h + 1} = \ell m N$ of $\{0,\ldots,\ell m N\}$, and the $2$-Lipschitz maps $\pi_{c} \colon \R^{2} \to \R$ with $c \in [-1,1]$:
\begin{align} \int_{-1}^{1} H(\pi_{c}\mu, \mathcal{D}_{\ell m N}) & \, d\nu(c) = \sum_{j = 1}^{h} \sum_{Q \in \mathcal{D}_{j}} \mu(Q) \int_{-1}^{1} H(\pi_{c}\mu^{Q},\mathcal{D}_{ms_{j + 1} - ms_{j}} \mid \mathcal{D}_{0}) \, d\nu(c) \notag \\
&\label{form51} \geq -C_{0}h + \sum_{j = 1}^{h} \sum_{Q \in \mathcal{D}_{j}} \mu(Q) \int_{-1}^{1} H(\pi_{c}\mu^{Q},\mathcal{D}_{m(t_{j} - s_{j} + 1)}) \, d\nu(c). \end{align} 
Above, $t_{j} - s_{j} + 1 = |\mathcal{J}_{j}|$. For $\mathcal{J}_{j} \in \mathcal{N}_{+}^{\mathrm{low}}$, and $Q \in \mathcal{D}_{j}$, we recall from \eqref{form47} that
\begin{displaymath} \int_{-1}^{1} H(\pi_{c}\mu^{Q},\mathcal{D}_{m|\mathcal{J}_{j}|}) \, d\nu(c) \geq \log R_{A'}^{m}(\mathcal{J}_{j}) + \xi \zeta \cdot m|\mathcal{J}_{j}| - \log 40 - \log C_{0}. \end{displaymath}
For $\mathcal{J} \in \mathcal{P}_{\mathrm{useless}}$ we have to settle with the estimate
\begin{displaymath} \int_{0}^{1} H(\pi_{c}\mu^{Q},\mathcal{D}_{m|\mathcal{J}_{j}|}) \, d\nu(c) \geq \log R_{A'}^{m}(\mathcal{J}_{j}) - 1 \end{displaymath}
from Lemma \ref{lemma6}. Plugging these bounds into \eqref{form51} (and redefining $C_{0}$ as $C_{0} + 1$) yields
\begin{align*} \int_{-1}^{1} H(\pi_{c}\mu,\mathcal{D}_{\ell m N}) \, d\nu(c) & \geq -(40 + C_{0})h + \sum_{\mathcal{J} \in \mathcal{P}} \log R_{A'}^{m}(\mathcal{J}) + \xi \zeta \sum_{\mathcal{J} \in \mathcal{N}_{+}^{\mathrm{low}}} |\mathcal{J}|\\
& \stackrel{\eqref{form40}}{\geq} \log |A'| + \xi \zeta \cdot \tfrac{1}{2} [(1 - \alpha) - (\alpha - \beta)/\Gamma] \cdot \ell m N - h(40 + C_{0}). \end{align*}
Recalling that $|A'| \geq \delta^{\eta}|A| \geq 2^{(\bar{\alpha} - \eta)\ell m N}$, there exists $c \in \spt(\nu)$ with 
\begin{displaymath} H_{\ell m N}(\pi_{c}\mu) \geq \left( \bar{\alpha} + \xi \zeta \cdot \tfrac{1}{2} [(1 - \alpha) - (\alpha - \beta)/\Gamma] - \eta \right) - \tfrac{h(40 + C_{0})}{\ell m N} \end{displaymath}
Here $h(40 + C_{0})/(\ell m N) \leq (40 + C_{0})/m_{0} \leq \eta$ by the choice of $m_{0}$ at \eqref{allParameters}, and since we chose $m \geq m_{0}$ in Proposition \ref{prop1}. Therefore we have established \eqref{form50}, and completed the proof of Theorem \ref{mainTechnical4}.

\section{Hausdorff dimension estimates}\label{appA}

The purpose of this final section is to reduce Theorem \ref{main} to Theorem \ref{mainTechnical}, and to use Theorem \ref{main} to prove the Hausdorff dimension result, Corollary \ref{hausdorffCor}. 

\begin{remark} The threshold $\gamma > (\alpha - \beta)/(1 - \beta)$ familiar from Theorems \ref{main} and \ref{mainTechnical} plays no particular role in this section: if we knew that Theorem \ref{mainTechnical} holds for all $\gamma \in (\tau,1]$ for some parameter $\tau = \tau(\alpha,\beta) \in (0,1)$, then the argument would below would show that Theorem \ref{main} also holds for $\gamma > \tau$. This is relevant to know if one eventually manages to solve Conjecture \ref{mainConjecture}, and proves Theorem \ref{mainTechnical} with threshold $\tau(\alpha,\beta) = \alpha - \beta$.

\end{remark}

\subsection{Reducing Theorem \ref{main} to Theorem \ref{mainTechnical}: outline}\label{appB}

The reduction from Theorem \ref{main} to Theorem \ref{mainTechnical} proceeds in several stages. First, in Section \ref{s:toy}, we prove the following toy version of Theorem \ref{main}: instead of allowing for general subsets of the form $G \subset A \times B$ with $|G| \geq \delta^{\epsilon}|A||B|$, this version (Theorem \ref{mainSubset1}) only treats subsets of the form $G = A \times B'$ with $|B'| \geq \delta^{\epsilon}|B|$. The conclusion is that there exists $c \in \spt(\nu)$ such that $|A + cB'| \geq \delta^{-\epsilon}|A|$ for all $B' \subset B$ with $|B'| \geq \delta^{\epsilon}|B|$. 

Even the toy version, Theorem \ref{mainSubset1}, is not proved directly: we will pass through a toy-toy version, Theorem \ref{mainSubset2}, where we are first allowed to replace $A \times B$ by a subset of the form $A \times \bar{B}$, and then the conclusion explained above is established for $A \times \bar{B}$ in place of $A \times B$. Fortunately, the passage between the toy and toy-toy versions can be accomplished by a formal exhaustion argument, which I learned from He's paper \cite{MR4148151}.

The toy-toy version is eventually deduced, in Section \ref{s:subsetReductionB}, by a direct argument from the main Theorem \ref{mainTechnical}. This is the heart of the matter. Instead of giving details here, I mention a key difficulty: this reduction, and various other steps of the argument would be simpler if we \emph{a priori} knew that 
\begin{equation}\label{form24} |A + A| \approx |A| \quad \text{and} \quad |B + B| \approx |B|. \end{equation}
(In this heuristic discussion, I will leave the meaning of "$\approx$" to the reader's imagination.) In the case $|A| \approx |B|$, treated by Bourgain in \cite{Bourgain10}, this is automatic: if $|A + cB|_{\delta} \approx |A| \approx |B|$ for some $c \in [\tfrac{1}{2},1]$, then \eqref{form24} holds by Pl\"unnecke's inequality. However, in our situation $B$ is typically much smaller than $A$, and now the property $|A + cB|_{\delta} \approx |A|$ implies neither property in \eqref{form24}. Nevertheless, \eqref{form24} is needed, technically because Lemma \ref{OVLemma} is useless without \eqref{form24}. Roughly speaking, Theorem \ref{mainSubset2} is proved by making a counter assumption, and using it to generate new sets $\bar{A} \neq A$ and $\bar{B} \neq B$ which satisfy the original hypotheses, and additionally \eqref{form24}. At some level, this argument is reminiscent of the proof of the asymmetric Balog-Szemer\'edi-Gowers theorem in \cite{MR2289012} (see Theorem \ref{BSG}).

Once we have the toy version, Theorem \ref{mainSubset1}, at our disposal, it remains to deduce Theorem \ref{main} from Theorem \ref{mainSubset1}. This step is based on the asymmetric Balog-Szemer\'edi-Gowers theorem -- unlike the other steps. We make a counter assumption that for every $c \in \spt(\nu)$ there exists a subset $G_{c} \subset A \times B$ with $|G| \gtrapprox |A||B|$ such that $|\pi_{c}(G)|_{\delta} \lessapprox |A|$. By the B-S-G theorem, this yields for every $c \in \spt(\nu)$ subsets $A_{c} \subset A$ and $B_{c} \subset B$ such that $|A_{c}| \gtrapprox |A|$, $|B_{c}| \gtrapprox |B|$, and $|A_{c} + cB_{c}|_{\delta} \lessapprox |A|$. With the help of probabilistic arguments, and the Pl\"unnecke-Ruzsa inequality (Lemma \ref{PRIneq}), this allows us to construct a new $\delta$-separated set $H \subset [0,1]$ with $|H| \lessapprox |A|$, and a subset $C \subset \spt(\nu)$ with $\nu(C) \gtrapprox 1$, such that $|H + cB_{c}|_{\delta} \lessapprox |H|$ for all $c \in C$. This violates the first toy version, Theorem \ref{mainSubset1}, applied to $H,B$ and finally concludes the proof of Theorem \ref{main}.

\subsection{A toy version}\label{s:toy}

Theorem \ref{main} claims the existence of $c \in \spt(\nu)$ such that $|\pi_{c}(G)| \geq \delta^{-\epsilon}|A|$ for all $G \subset A \times B$ with $|G| \geq \delta^{\epsilon}|A||B|$. A toy problem is to find $c \in \spt(\nu)$ such that $|A + cB'|_{\delta} \geq \delta^{-\epsilon}|A|$ for all $B' \subset B$ with $|B'| \geq \delta^{\epsilon}|B|$. Instead of approaching Theorem \ref{main} directly, we will first solve this toy problem:

\begin{thm}\label{mainSubset1} Let $0 < \beta \leq \alpha < 1$ and $\kappa > 0$. Then, for every $\gamma \in ((\alpha - \beta)/(1 - \beta),1]$, there exist $\epsilon_{0},\epsilon,\delta_{0} \in (0,\tfrac{1}{2}]$, depending only on $\alpha,\beta,\gamma,\kappa$, such that the following holds. Let $\delta \in 2^{-\N}$ with $\delta \in (0,\delta_{0}]$, and let $A,B \subset (\delta \cdot \Z) \cap [0,1]$ satisfy the following hypotheses:
\begin{enumerate}
\item[(A)] \label{A} $|A| \leq \delta^{-\alpha}$.
\item[(B)] \label{B} $|B| \geq \delta^{-\beta}$, and $B$ satisfies the following Frostman condition: 
\begin{displaymath} |B \cap B(x,r)| \leq r^{\kappa}|B|, \qquad \delta \leq r \leq \delta^{\epsilon_{0}}. \end{displaymath} 
\end{enumerate}
Further, let $\nu$ be a Borel probability measure with $\spt (\nu) \subset [0,1]$, and satisfying the Frostman condition $\nu(B(x,r)) \leq r^{\gamma}$ for $x \in \R$ and $0 < r \leq \delta^{\epsilon_{0}}$. Then, there exists $c \in \spt(\nu)$ such that if $B' \subset B$ satisfies $|B'| \geq \delta^{\epsilon}|B|$, then $|A + cB'| \geq \delta^{-\epsilon}|A|$. \end{thm}

\subsection{Reduction to a weaker toy theorem}\label{s:toytoyReduction} Even Theorem \ref{mainSubset1} is hard to prove with a direct assault. We will first need to reduce it to an even weaker version. In the statement, we use the following notation (slightly adapted) from He's paper \cite{MR4148151}. Given two sets $A,B \subset [0,1] \cap (\delta \cdot \Z)$, we write
\begin{displaymath} \mathcal{E}(A \mid B,\epsilon) := \{c \in \R : \exists \, B' \subset B \text{ such that } |B'| \geq \delta^{\epsilon}|B| \text{ and } |A + cB'|_{\delta} < \delta^{-\epsilon}|A|\}. \end{displaymath} 

\begin{thm}\label{mainSubset2} Let $0 < \beta \leq \alpha < 1$ and $\kappa,\theta > 0$. Then, for every $\gamma \in ((\alpha - \beta)/(1 - \beta),1]$, there exist $\epsilon_{0},\epsilon,\delta_{0} \in (0,\tfrac{1}{2}]$, depending only on $\alpha,\beta,\gamma,\kappa$, such that the following holds. Let $\delta \in 2^{-\N}$ with $\delta \in (0,\delta_{0}]$, and let $A,B \subset (\delta \cdot \Z) \cap [0,1]$ satisfy the following hypotheses:
\begin{enumerate}
\item[(A)] \label{A} $|A| \leq \delta^{-\alpha}$.
\item[(B)] \label{B} $|B| \geq \delta^{-\beta}$, and $B$ satisfies the following Frostman condition: 
\begin{displaymath} |B \cap B(x,r)| \leq r^{\kappa}|B|, \qquad \delta \leq r \leq \delta^{\epsilon_{0}}. \end{displaymath} 
\end{enumerate}
Further, let $\nu$ be a Borel probability measure with $\spt (\nu) \subset [0,1]$, and satisfying the Frostman condition $\nu(B(x,r)) \leq r^{\gamma}$ for $x \in \R$ and $0 < r \leq \delta^{\epsilon_{0}}$. Then, there exists a subset $B' \subset B$ such that $\nu(\mathcal{E}(A \mid B',\epsilon)) \leq \delta^{\epsilon}$. \end{thm}

I learned this reduction from the paper of He \cite[Proposition 25]{MR4148151}, and his proof works here, up to modifying the notation. The full details are recorded below nonetheless.

\begin{proof}[Proof of Theorem \ref{mainSubset1} assuming Theorem \ref{mainSubset2}] Let $\alpha,\beta,\gamma,\kappa$ be the parameters given in Theorem \ref{mainSubset1}, so that $\gamma > (\alpha - \beta)/(1 - \beta)$. Our task is to find the constants $\epsilon,\epsilon_{0},\delta_{0} \in (0,\tfrac{1}{2}]$, depending only on $\alpha,\beta,\gamma,\kappa$. Start by applying Theorem \ref{mainSubset2} with parameters $\alpha,\bar{\beta},\gamma,\bar{\kappa}$, where $\bar{\kappa} \in (0,\kappa)$ is arbitrary, and also and $\bar{\beta} < \beta$ is arbitrary with the property that the key inequality
\begin{displaymath} \gamma > (\alpha - \bar{\beta})/(1 - \bar{\beta}) \end{displaymath}
remains valid. Let $\bar{\epsilon},\bar{\epsilon}_{0},\bar{\delta}_{0} \in (0,\tfrac{1}{2}]$ be the constants given by Theorem \ref{mainSubset2}, associated to the parameters $\alpha,\bar{\beta},\gamma,\bar{\kappa}$. We define
\begin{equation}\label{form6} \epsilon_{0} := \bar{\epsilon}_{0} \quad \text{and} \quad \epsilon := \min\left\{\frac{\bar{\epsilon}}{2},\frac{(\kappa - \bar{\kappa})\bar{\epsilon}_{0}}{4},\frac{\beta - \bar{\beta}}{2} \right\}. \end{equation}
We assume that $\delta_{0} \leq \bar{\delta}_{0}$, and there will be a few additional requirements, where for example $\delta \leq \delta_{0}$ needs to be taken small enough relative to the difference $\bar{\epsilon} - \epsilon$. I will not gather these requirements together; they will be pointed out where they appear.

Let $\delta \in 2^{-\N}$ with $\delta \leq \delta_{0}$, and let $A,B,\nu$ be the objects from Theorem \ref{mainSubset1}, satisfying the assumptions of that theorem with constants $\alpha,\beta,\kappa,\gamma$, and $\epsilon_{0},\delta_{0}$ as above. In particular,
\begin{equation}\label{form7} |B| \geq \delta^{-\beta} \quad \text{and} \quad |B \cap B(x,r)| \leq r^{\kappa}|B| \text{ for } x \in \R \text{ and } \delta \leq r \leq \delta^{\epsilon_{0}}. \end{equation}
Evidently $A,B,\nu$ also satisfy the hypotheses of Theorem \ref{mainSubset2} with constants $\alpha,\bar{\beta},\gamma,\kappa/2$, and $\bar{\epsilon}_{0}$. We now perform an "exhaustion" argument to construct a finite sequence of disjoint subsets $B_{1},\ldots,B_{N} \subset B$ with the property
\begin{equation}\label{form20} \nu(\mathcal{E}(A \mid B_{j},\bar{\epsilon})) \leq \delta^{\bar{\epsilon}}, \qquad 1 \leq j \leq N. \end{equation}
Let $B_{1} \subset B$ be the set given initially by Theorem \ref{mainSubset2}. We then assume inductively that we have already constructed disjoint $B_{1},\ldots,B_{n} \subset B$ for some $n \geq 1$. There are two options:
\begin{equation}\label{form1a} \Big| B \, \setminus \, \bigcup_{j = 1}^{n} B_{j} \Big| < \delta^{2\epsilon}|B| \quad \text{or} \quad \Big| B \, \setminus \, \bigcup_{j = 1}^{n} B_{j} \Big| \geq \delta^{2\epsilon}|B|.  \end{equation} 
In the former case, the inductive construction terminates, and we define $N := n$. In the latter case, we apply Theorem \ref{mainSubset2} to the objects $A,\nu$, and $B' := B \, \setminus \, \bigcup_{j = 1}^{n} B_{j}$. This is legitimate, because $|B'| \geq \delta^{2\epsilon}|B| \geq \delta^{-\beta - 2\epsilon} \geq \delta^{-\bar{\beta}}$, and
\begin{displaymath} |B' \cap B(x,r)| \stackrel{\eqref{form7}}{\leq} r^{\kappa}|B| \leq \delta^{-2\epsilon}r^{\kappa}|B'| \stackrel{\eqref{form6}}{\leq} r^{\bar{\kappa}}|B'|, \qquad x \in \R, \, \delta \leq r \leq \delta^{\epsilon_{0}} = \delta^{\bar{\epsilon}_{0}}. \end{displaymath}
Therefore $A,B',\nu$ satisfy the hypotheses of Theorem \ref{mainSubset2} with constants $\alpha,\bar{\beta},\bar{\kappa},\gamma,\bar{\epsilon}_{0}$. Consequently, there exists a further subset $B_{n + 1} \subset B' = B \, \setminus \, \bigcup_{j = 1}^{n} B_{j}$ with the property $\nu(\mathcal{E}(A \mid B_{n + 1},\bar{\epsilon})) \leq \delta^{\bar{\epsilon}}$. This completes the inductive construction of the sequence $B_{1},\ldots,B_{N}$. The construction terminates in $\leq \delta^{-\bar{\epsilon}}$ steps, because the sets $B_{j}$ satisfy $|B_{j}| \geq \delta^{-\bar{\epsilon}}$. Indeed, since $\nu(\mathcal{E}(A \mid B_{j},\bar{\epsilon})) < 1$, there exists $c \in \spt(\nu) \, \setminus \, \mathcal{E}(A \mid B_{j},\bar{\epsilon})$, and then $|A||B_{j}| \geq |A + cB_{j}|_{\delta} \geq \delta^{-\bar{\epsilon}}|A|$.

When the inductive procedure eventually terminates, we write $B_{0} := \bigcup_{j = 1}^{N} B_{j}$. By \eqref{form1a}, we have $|B \, \setminus \, B_{0}| < \delta^{2\epsilon}|B|$. Now, note that the claim of Theorem \ref{mainSubset1} is equivalent to proving that $\spt(\nu) \, \setminus \, \mathcal{E}(A \mid B,\epsilon) \neq \emptyset$. We will prove this by showing that $\mathcal{E}(A \mid B,\epsilon)$ has small $\nu$ measure. The first step is to establish the following inclusion:
\begin{equation}\label{form19} \mathcal{E}(A \mid B,\epsilon) \subset \bigcup_{\mathcal{J}} \bigcap_{j \in \mathcal{J}} \mathcal{E}(A \mid B_{j},\bar{\epsilon}), \end{equation}
where the index set $\mathcal{J}$ runs over all subsets of $\{1,\ldots,N\}$ with $\sum_{j \in \mathcal{J}} |B_{j}| \geq \delta^{\epsilon}|B|/4$. The proof is nearly verbatim the same as in \cite[Proposition 25]{MR4148151}, but I record the details here for completeness. If $c \in \mathcal{E}(A \mid B,\epsilon)$, then by definition there exists a subset $B_{c} \subset B$ with $|B_{c}| \geq \delta^{\epsilon}|B|$ and $|A + cB_{c}|_{\delta} < \delta^{-\epsilon}|A|$. Let $\mathcal{J} := \{1 \leq j \leq N : |B_{c} \cap B_{j}| \geq \delta^{\bar{\epsilon}}|B_{j}|\}$. Then $c \in \mathcal{E}(A \mid B_{j},\bar{\epsilon})$ for all $j \in \mathcal{J}$, since $B_{j}' := B_{c} \cap B_{j} \subset B_{j}$ satisfies $|B_{j}'| \geq \delta^{\bar{\epsilon}}|B_{j}|$ and $|A + cB_{j}'|_{\delta} < \delta^{-\epsilon}|A| \leq \delta^{-\bar{\epsilon}}|A|$. This proves \eqref{form19}, once we verify that $\sum_{j \in \mathcal{J}} |B_{j}| \geq \delta^{\epsilon}|B|/4$. 

To see this, recall that $|B \, \setminus \, B_{0}|  \leq \delta^{2\epsilon}|B|$. This implies that $B_{c}$ has large intersection with $B_{0}$ (assuming that $\delta > 0$ is sufficiently small):
\begin{displaymath} |B_{c} \cap B_{0}| \geq \tfrac{1}{2} \cdot \delta^{\epsilon}|B|. \end{displaymath}
Then, if $\delta > 0$ is small enough, and recalling that $\epsilon \leq \bar{\epsilon}/2$, we have
\begin{align*} \tfrac{1}{2} \cdot \delta^{\epsilon}|B| \leq |B_{c} \cap B_{0}| = \sum_{j = 1}^{N} |B_{c} \cap B_{j}| \leq \sum_{j \notin \mathcal{J}} \delta^{\bar{\epsilon}}|B_{j}| + \sum_{j \in \mathcal{J}} |B_{j}| \leq \tfrac{1}{4} \cdot \delta^{\epsilon}|B| + \sum_{j \in \mathcal{J}} |B_{j}|. \end{align*}
Rearranging, $\sum_{j \in \mathcal{J}} |B_{j}| \geq \delta^{\epsilon}|B|/4$. We have now established the inclusion \eqref{form19}.

Finally, it follows from \eqref{form19} and \cite[Lemma 20]{MR4148151} that
\begin{equation}\label{form25} \nu(\mathcal{E}(A \mid B,\epsilon)) \leq \nu \left( \bigcup_{\mathcal{J}} \bigcap_{j \in \mathcal{J}} \mathcal{E}(A \mid B_{j},\bar{\epsilon}) \right) \leq 4\delta^{\bar{\epsilon} - \epsilon} < 1, \end{equation}
assuming once more that $\delta > 0$ is small enough in the final inequality. The proof of \cite[Lemma 20]{MR4148151} is, again, so short that we provide the details for the reader's convenience. If $c \in \bigcup_{\mathcal{J}} \bigcap_{j \in \mathcal{J}} \mathcal{E}(A \mid B_{j},\bar{\epsilon})$, then $\sum_{j = 1}^{N} (|B_{j}|/|B|) \cdot \mathbf{1}_{\mathcal{E}(A \mid B_{j},\bar{\epsilon})}(c) \geq \delta^{\epsilon}/4$. Consequently,
\begin{align*} \nu \left( \bigcup_{\mathcal{J}} \bigcap_{j \in \mathcal{J}} \mathcal{E}(A \mid B_{j},\bar{\epsilon}) \right) & \leq 4\delta^{-\epsilon} \sum_{j = 1}^{N} \frac{|B_{j}|}{|B|} \cdot \nu(\mathcal{E}(A \mid B_{j},\bar{\epsilon}))\\
& \leq 4\delta^{-\epsilon} \max_{1 \leq j \leq N} \nu(\mathcal{E}(A \mid B_{j},\bar{\epsilon})) \stackrel{\eqref{form20}}{\leq} 4\delta^{-\epsilon + \bar{\epsilon}}. \end{align*}
This concludes the proof of Theorem \ref{mainSubset1}. \end{proof}

\begin{remark}\label{rem4} The proof above demonstrated that in order to deduce Theorem \ref{mainSubset1} for a fixed constant "$\kappa$" from Theorem \ref{mainSubset2}, one only needs to apply Theorem \ref{mainSubset2} for any value $\bar{\kappa} < \kappa$ arbitrarily close to $\kappa$ (but the value of the constant $\epsilon \to 0$ as $\bar{\kappa} \nearrow \kappa$).

\end{remark}

\subsection{Proof of the weaker toy theorem}\label{s:subsetReductionB} In this section, we prove Theorem \ref{mainSubset2} by reducing it to Theorem \ref{mainTechnical}, which we have already established. We will need a few auxiliary results. One is the Pl\"unnecke-Ruzsa inequality for different summands, Lemma \ref{PRIneq} (only with $\epsilon = \tfrac{1}{2}$). Another auxiliary result will concern the existence of \emph{tight} subsets:

\begin{definition} For $\tau,T > 0$ and $N \in \N$, a set $A \subset \delta \cdot \Z$ is called $(\tau,T,N)$-tight if
\begin{displaymath} \max_{1 \leq k \leq N} \frac{|kA|}{|kA'|} \leq T \qquad \text{for all } A' \subset A \text{ with } |A'| \geq \delta^{\tau}|A|. \end{displaymath}
\end{definition}

It will be useful to observe that if $A$ is $(\tau,T,N)$-tight, and $0 < \tau' \leq \tau$, then $A$ is also $(\tau',T,N)$-tight, simply because there are fewer sets $A' \subset A$ to consider.

\begin{lemma}\label{lemma7} Let $\tau > 0$, $N \in \N$, and let $A \subset (\delta \cdot \Z) \cap [0,1]$ be a set with $|A| \geq \delta^{-N^{2}\tau}$. Then, there exists a $(\tau,2\delta^{-1/N},N)$-tight subset $A' \subset A$ of cardinality $|A'| \geq \delta^{N^{2}\tau}|A|$. \end{lemma}

\begin{proof} We find a sequence $A =: A_{0} \supset A_{1} \supset \ldots \supset A_{N^{2}}$ as follows. Assuming that $A_{j - 1}$ has already been selected, and $1 \leq j \leq N^{2}$, we let $A_{j} \subset A_{j - 1}$ be a subset with $|A_{j}| \geq \delta^{\tau}|A_{j - 1}|$ such that the quantity
\begin{displaymath} \max_{1 \leq k \leq N} |kA_{j - 1}|/|kA'| \end{displaymath}
is maximised among all subsets $A' \subset A_{j - 1}$ with $|A'| \geq \delta^{\tau}|A_{j - 1}|$. Thus, we see that if $\max_{1 \leq k \leq N} |kA_{j - 1}|/|kA_{j}| \leq T$, then $A_{j - 1}$ is $(\tau,T,N)$-tight.

Observe that 
\begin{equation}\label{a3}|A_{N^{2}}| \geq \delta^{N^{2}\tau}|A| \geq 1 \quad \text{and} \quad |NA| \leq N\delta^{-1}. \end{equation}
Writing $T := 2\delta^{-1/N}$, we now claim that there exists an index $j \in \{1,\ldots,N^{2}\}$ with
\begin{equation}\label{a2} \max_{1 \leq k \leq N} |kA_{j - 1}|/|kA_{j}| \leq T. \end{equation}
Indeed, if this fails, then by the pigeonhole principle there exists a fixed choice $k \in \{1,\ldots,N\}$, and $n$ indices $j_{1},\ldots,j_{N} \in \{1,\ldots,N^{2}\}$ such that the converse inequality
\begin{displaymath} |kA_{j_{i}}| < T^{-1}|kA_{j_{i} - 1}|, \qquad 1 \leq i \leq N, \end{displaymath}
holds. Since $A_{j} \subset A_{j - 1}$, the inequality $|kA_{j}| \leq |kA_{j - 1}|$ holds for every index $j \in \{1,\ldots,N^{2}\}$, and $|kA_{j_{i}}| < T^{-1}|kA_{j_{i} - 1}|$ for the $n$ special indices $j_{i} \in \{1,\ldots,N^{2}\}$. This forces
\begin{displaymath} 1 \stackrel{\eqref{a3}}{\leq} |kA_{N^{2}}| \leq |kA_{j_{N}}| < T^{-N}|kA| \stackrel{\eqref{a3}}{\leq} (2^{-N}\delta) \cdot (N\delta^{-1}) \leq 1, \end{displaymath}
a contradiction. Now $A' := A_{j - 1} \subset A$, as in \eqref{a2}, is $(\tau,T,N)$-tight, and $|A'| \geq |A_{N^{2}}| \geq \delta^{N^{2}\tau}|A|$. This completes the proof of the lemma. \end{proof}

Finally, we will need the following lemma, which is a $\delta$-discretised version of \cite[Lemma 3.1]{OV18}, or alternatively a version of Bourgain's computations \cite[(7.18)-(7.19)]{Bourgain10} for two different sets (the presence of two different sets adds no difficulties):

\begin{lemma}\label{OVLemma} Let $C_{1},C_{2},C_{3} > 0$, and assume that $A,B \subset \delta \cdot \Z$ are sets with $|A + A| \leq C_{1}|A|$ and $|B + B| \leq C_{2}|B|$. Let moreover $c \in \R$, and let $G \subset A \times B$ be an arbitrary subset with $|G| \geq |A||B|/C_{3}$. Then $|A + cB|_{\delta} \lesssim C_{1}C_{2}C_{3}|\pi_{c}(G)|_{\delta}$.
\end{lemma} 
\begin{proof} Note that
\begin{displaymath} |A + cB|_{\delta} \lesssim \sum_{t \in \delta \cdot \Z} \mathbf{1}_{(A + cB)(\delta)}(t). \end{displaymath}
Fix $t \in (\delta \cdot \Z) \cap (A + cB)(\delta)$ and find $(a,b) \in A \times B$ such that $\dist(t,\pi_{c}(a,b)) \leq \delta$. Then
\begin{displaymath} \dist(t,\pi_{c}(-G + (x,y))) \leq \dist(t,\pi_{c}(a,b)) \leq \delta, \qquad (x,y) \in G + (a,b). \end{displaymath}
Moreover, any candidates $(x,y) \in G + (a,b)$ satisfy
\begin{displaymath} (x,y) \in G + (a,b) \subset (A \times B) + (A \times B) = (A + A) \times (B + B), \end{displaymath}
so there are $\geq |G + (a,b)| = |G|$ points $(x,y) \in (A + A) \times (B + B)$ with the property $\dist(t,\pi_{c}(-G + (x,y))) \leq \delta$. It follows that
\begin{displaymath} \mathbf{1}_{(A + cB)(\delta)}(t) = 1 \leq \frac{1}{|G|} \sum_{(x,y) \in (A + A) \times (B + B)} \mathbf{1}_{\pi_{c}(-G + (x,y))(\delta)}(t). \end{displaymath}
Since $|\pi_{c}(G)|_{\delta} \sim |\pi_{c}(-G + (x,y))|_{\delta}$ for every $(x,y) \in (\delta \cdot \Z)^{2}$, we have
\begin{align*} |A + cB|_{\delta} & \lesssim \frac{1}{|G|} \sum_{(x,y) \in (A + A) \times (B + B)} \sum_{t \in \delta \cdot \Z} \mathbf{1}_{\pi_{c}(-G + (x,y))(\delta)}(t)\\
& \lesssim \frac{|A + A||B + B||\pi_{c}(G)|_{\delta}}{|G|} \leq C_{1}C_{2}C_{3}|\pi_{c}(G)|_{\delta}, \end{align*}
as claimed. \end{proof}

We are now ready to carry out the main task in this section, namely reducing the proof of Theorem \ref{mainSubset2} to Theorem \ref{mainTechnical}, which we repeat here for the reader's convenience:

\begin{thm}\label{mainTechnicalApp} Let $0 < \beta \leq \alpha < 1$ and $\kappa > 0$. Then, for every $\gamma \in ((\alpha - \beta)/(1 - \beta),1]$, there exist $\epsilon_{0},\epsilon,\delta_{0} \in (0,\tfrac{1}{2}]$, depending only on $\alpha,\beta,\gamma,\kappa$, such that the following holds. Let $\delta \in 2^{-\N}$ with $\delta \in (0,\delta_{0}]$, and let $A,B \subset (\delta \cdot \Z) \cap [0,1]$ satisfy the following hypotheses:
\begin{enumerate}
\item[(A)] \label{A} $|A| \leq \delta^{-\alpha}$.
\item[(B)] \label{B} $|B| \geq \delta^{-\beta}$, and $B$ satisfies the following Frostman condition: 
\begin{displaymath} |B \cap B(x,r)| \leq r^{\kappa}|B|, \qquad \delta \leq r \leq \delta^{\epsilon_{0}}. \end{displaymath} 
\end{enumerate}
Further, let $\nu$ be a Borel probability measure with $\spt (\nu) \subset [0,1]$, and satisfying the Frostman condition $\nu(B(x,r)) \leq r^{\gamma}$ for $x \in \R$ and $\delta \leq r \leq \delta^{\epsilon_{0}}$. Then, there exists a point $c \in \spt (\nu)$ such that
\begin{displaymath} |A + cB|_{\delta} \geq \delta^{-\epsilon}|A|. \end{displaymath}
\end{thm}

\begin{proof}[Proof of Theorem \ref{mainSubset2} assuming Theorem \ref{mainTechnicalApp}] Fix the parameters $0 < \beta \leq \alpha < 1$, $\kappa > 0$, and $\gamma \in ((\alpha - \beta)/(1 - \beta),1]$, as in Theorem \ref{mainSubset2}. Start by applying Theorem \ref{mainTechnicalApp} with the following slightly modified parameters:
\begin{displaymath} 0 < \bar{\beta} \leq \bar{\alpha} < 1, \, \bar{\kappa} \in (0,\kappa), \quad \text{and} \quad \bar{\gamma} > (\bar{\alpha} - \bar{\beta})/(1 - \bar{\beta}), \end{displaymath}
where $\bar{\alpha} > \alpha$ and $\bar{\beta} < \beta$ and $\bar{\gamma} < \gamma$ are arbitrary choices such that the final inequality is valid. The choice of $\bar{\kappa} \in (0,\kappa)$ is arbitrary. As usual, the parameters $\bar{\alpha},\bar{\beta},\bar{\gamma}$ should be viewed as functions of $\alpha,\beta,\gamma$, but we leave finding explicit expressions to the reader. Any future dependence on $\bar{\alpha},\bar{\beta},\bar{\gamma}$ will, in fact, be a dependence on $\alpha,\beta,\gamma$. Then, let 
\begin{displaymath} \bar{\epsilon}_{0},\bar{\epsilon},\bar{\delta}_{0} > 0 \end{displaymath}
be the constants given by Theorem \ref{mainTechnicalApp}, which only depend on $\bar{\alpha},\bar{\beta},\bar{\kappa},\bar{\gamma}$, and such that the conclusion of Theorem \ref{mainTechnicalApp} is valid. Our task is to find constants $\epsilon,\epsilon_{0},\delta_{0} > 0$, which may depend on all of the constants $\alpha,\bar{\alpha},\beta,\bar{\beta},\gamma,\bar{\gamma},\kappa,\bar{\kappa},\bar{\epsilon}_{0},\bar{\epsilon},\bar{\delta}_{0}$, such that Theorem \ref{mainSubset2} is valid with constants $\alpha,\beta,\gamma,\kappa$. The choice of $\epsilon_{0}$ is particularly simple:
\begin{equation}\label{aa22} \epsilon_{0} = \bar{\epsilon}_{0}. \end{equation}
For $\delta_{0}$, we will need that $\delta_{0} \leq \bar{\delta}_{0}$, and there will be an additional dependence on $\bar{\alpha},\bar{\epsilon}$, which will be clarified during the proofs of \eqref{aa5} and \eqref{a19}. To define the constant $\epsilon \in (0,\tfrac{1}{2}]$, we first introduce an auxiliary natural number $N \in \N$ satisfying
\begin{equation}\label{aa23} \frac{3}{N} + \frac{1}{\log_{2} N} \leq \bar{\epsilon}/2. \end{equation}
Then, we choose $\epsilon > 0$ so small that
\begin{equation}\label{a14} N^{4N + 1}\epsilon \leq \min\left\{\bar{\epsilon}_{0}(\kappa - \bar{\kappa}),\beta - \bar{\beta},\bar{\alpha} - \alpha\right\} \quad \text{and} \quad \epsilon \leq \bar{\epsilon}_{0}(\gamma - \bar{\gamma}). \end{equation}
We now claim that Theorem \ref{mainSubset2} holds with the constants $\epsilon,\epsilon_{0},\delta_{0}$ (given the parameters $\alpha,\beta,\gamma,\kappa$). Let $\delta \in 2^{-\N}$ with $\delta \leq \delta_{0}$, and let $A,B,\nu$ be objects satisfying the hypotheses of Theorem \ref{mainSubset2} with constants $\alpha,\beta,\gamma,\kappa,\epsilon_{0}$. Thus $A,B \subset [0,1] \cap (\delta \cdot \Z)$, $|A| \leq \delta^{-\alpha}$, and $\nu(B(x,r)) \leq r^{\gamma}$ for $x \in \R$ and $\delta \leq r \leq \delta^{\epsilon_{0}} = \delta^{\bar{\epsilon}_{0}}$. Further, $|B| \geq \delta^{-\beta}$, and
\begin{equation}\label{form17} |B \cap B(x,r)| \leq r^{\kappa}|B|, \qquad x \in \R, \, \delta \leq r \leq \delta^{\epsilon_{0}} = \delta^{\bar{\epsilon}_{0}}. \end{equation}
The claim is that there exists a subset $B' \subset B$ such that $\nu(\mathcal{E}(A \mid B',\epsilon)) \leq \delta^{\epsilon}$. We proceed by making a counter assumption:

\begin{counter} $\nu(\mathcal{E}(A \mid B',\epsilon)) > \delta^{\epsilon}$ for all $B' \subset B$. \end{counter}

We will use our \textbf{Counter assumption} and Lemmas \ref{lemma7} and \ref{OVLemma} to construct a sequence $\{H_{n}\}_{n = 1}^{N} \subset \delta \cdot \Z$ with $|H_{n}| \leq \delta^{-\bar{\alpha}}$. The point will be, omitting all technical details, that once this sequence has been constructed, we will find an index $n \in \{0,\ldots,N - 1\}$ with the property that $|H_{n} + cB| < \delta^{-\bar{\epsilon}}|H_{n}|$ for all $c \in \spt(\nu)$. This (or the more technical version of it) will violate Theorem \ref{mainTechnicalApp}, and show that the \textbf{Counter assumption} is false.

Let $\{\tau_{n}\}_{n = 0}^{N}$ be the finite decreasing sequence
\begin{equation}\label{defTau} \tau_{n} := N^{4N - 3n}\epsilon, \qquad 0 \leq n \leq N. \end{equation}
While we construct the sets $H_{n}$, we will simultaneously find elements $c_{1},c_{2},\ldots,c_{N} \in C = \spt(\nu)$, subsets $C_{n} \subset C$ of measure $\nu(C_{n}) \geq \delta^{\epsilon}$, and a decreasing sequence $B =: B_{0} \supset B_{1} \supset \ldots \supset B_{N}$ with the following three properties:
\begin{enumerate}
\item $|B_{n + 1}| \geq \delta^{\tau_{n}}|B_{n}|$ for $0 \leq n \leq N - 1$,
\item $|A + c_{n}B_{n}| < \delta^{-\epsilon}|A|$ for $1 \leq n \leq N$,
\item $B_{n}$ is $(\tau_{n},2\delta^{1/N},N)$-tight for $1 \leq n \leq N$.
\end{enumerate}
In particular, it follows from property (1) that
\begin{equation}\label{a13} |B_{n}| \geq \delta^{N\tau_{0}}|B| \geq \delta^{N^{4N + 1}\epsilon}|B|, \qquad 0 \leq n \leq N. \end{equation} 
To initialise the definition of the sets $B_{n},C_{n},H_{n}$, and the elements $c_{n} \in C$, set $B_{0} := B$ and $H_{0} := \emptyset$. (the properties (2)-(3) do not concern the case $n = 0$). Assume that the sets $B_{n}$ have already been constructed for some $0 \leq n \leq N - 1$, and recall the notation $(H)_{\delta} := (\delta \cdot \Z) \cap H(\delta)$ for arbitrary $H \subset \R$. By the \textbf{Counter assumption} applied to the set $B' := B_{n} \subset B$, there now corresponds a subset 
\begin{equation}\label{form21} C_{n + 1} := \mathcal{E}(A \mid B_{n},\epsilon) \cap C \subset C \end{equation}
of measure $\nu(C_{n + 1}) \geq \delta^{\epsilon}$ with the property that for all $c \in C_{n + 1}$, there is a further subset $\bar{B}_{c} \subset B_{n}$ of cardinality $|\bar{B}_{c}| \geq \delta^{\epsilon}|B_{n}|$ such that $|A + c\bar{B}_{c}| < \delta^{-\epsilon}|A|$. We will define $H_{n + 1}$ as either $H_{n + 1} := H_{n} + H_{n} \subset \delta \cdot \Z$, or 
\begin{equation}\label{a11} H_{n + 1} := H_{n} + (c \cdot NB_{c})_{\delta} \subset \delta \cdot \Z, \end{equation}
where $c \in C_{n + 1}$, and $B_{c} \subset \bar{B}_{c} \subset B_{n}$ is a certain set satisfying the constraints (1)-(3). It turns out that subsets of this kind exist for all $c \in C_{n + 1}$: this will be proved shortly, but should be taken for granted for now. For every $c \in C_{n + 1}$, we then pick the subset $B_{c} \subset \bar{B}_{c} \subset B_{n}$ which satisfies (1)-(3), and maximises the number $|H_{n} + c \cdot NB_{c}|_{\delta}$, among all possible $c \in C_{n + 1}$, and subsets $B_{c} \subset \bar{B}_{c}$ satisfying (1)-(3). Once the optimal $c \in C_{n + 1}$ and $B_{c} \subset \bar{B}_{c} \subset B_{n}$ have been located, we finally check if 
\begin{displaymath} |H_{n} + c \cdot NB_{c}|_{\delta} \geq |H_{n} + H_{n}|. \end{displaymath}
If this happens, then $H_{n + 1}$ is defined as in \eqref{a11}. Otherwise $H_{n + 1} := H_{n} + H_{n}$. Note that in both cases
\begin{equation}\label{a12} |H_{n} + H_{n}| \leq |H_{n + 1}|. \end{equation}
If $H_{n}$ was defined by \eqref{a11}, for some 
\begin{displaymath} c_{n + 1} := c \in C_{n + 1}, \end{displaymath}
then we set $B_{n + 1} := B_{c} \subset B_{n}$, where $B_{c}$ is the maximising set found above. If $H_{n + 1} = H_{n} + H_{n}$, we simply define $B_{n + 1} := B_{n}$, and $c_{n + 1} := c_{n}$. Note that in all cases the properties (1)-(3) are satisfied, and $B_{n + 1}$ is $(\tau_{n + 1},2\delta^{1/N},N)$-tight. This is even true if $B_{n + 1}$ was defined via the "second scenario" as $B_{n + 1} = B_{n}$: indeed, since $H_{0} = \emptyset$, this is only possible if $n \geq 1$, and then we already know that $B_{n}$ is $(\tau_{n},2\delta^{1/N},N)$-tight. Then $B_{n + 1} = B_{n}$ is also $(\tau_{n + 1},2\delta^{1/N},N)$-tight simply because $\tau_{n + 1} \leq \tau_{n}$.

This completes the inductive definition of the sets $B_{n},C_{n},H_{n}$, and elements $c_{n} \in C$, for $1 \leq n \leq N$. Note that $H_{n} \subset (\delta \cdot \Z) \cap [0,N^{n}]$ by a straightforward induction, so $|H_{n}| \leq 2N^{n} \delta^{-1}$. Therefore, by the pigeonhole principle, there exists an index $n \in \{0,\ldots,N - 1\}$ such that
\begin{equation}\label{a4} |H_{n} + H_{n}| \stackrel{\eqref{a12}}{\leq} |H_{n + 1}| \leq (2N^{n}\delta^{-1})^{1/N}|H_{n}| \leq 2N\delta^{-1/N}|H_{n}|. \end{equation}
Since $H_{0} = \emptyset \neq H_{1}$, the middle inequality cannot be satisfied with $n = 0$, and we see that actually $n \in \{1,\ldots,N - 1\}$. We now claim that, for this particular index $n$, fixed for the remainder of the argument, it holds that
\begin{equation}\label{aa5} |H_{n} + cB_{n}|_{\delta} < \delta^{-\bar{\epsilon}}|H_{n}|, \qquad c \in C_{n + 1}, \end{equation}
assuming that the upper bound $\delta_{0} > 0$ for $\delta$ is sufficiently small, depending only on $N$ (hence "$\bar{\epsilon}$" by our choice \eqref{aa23}). To see this, we first record that $2^{k}B_{n} \subset 2^{N}B \subset (\delta \cdot \Z) \cap [0,2^{N}]$ for all $1 \leq k \leq N$, so by another application of the pigeonhole principle, there exists an index $0 \leq k \leq \log_{2} (N - 1)$ such that
\begin{equation}\label{a8} |2^{k + 1}B_{n}| \leq (2^{N + 1}\delta^{-1})^{1/\log_{2} N}|2^{k}B_{n}|. \end{equation} 
We also fix this index $k \in \{0,\ldots,\log_{2}(N - 1)\}$ for the remainder of the argument.

Now, to prove \eqref{aa5}, fix $c \in C_{n + 1} \subset [0,1]$, and recall the subset $\bar{B}_{c} \subset B_{n}$ defined right below \eqref{form21}, satisfying $|\bar{B}_{c}| \geq \delta^{\epsilon}|B_{n}|$ and $|A + c\bar{B}_{c}| < \delta^{-\epsilon}|A|$. We use Lemma \ref{lemma7} to find a $(\tau_{n + 1},2\delta^{-1/N},N)$-tight subset $B_{c} \subset \bar{B}_{c} \subset B_{n}$ of cardinality 
\begin{align} |B_{c}| \geq \delta^{N^{2}\tau_{n + 1}}|\bar{B}_{c}| & \stackrel{\eqref{defTau}}{\geq} \delta^{N^{2}N^{4N - 3(n + 1)}\epsilon + \epsilon}|B_{n}| \notag\\
& \,\,\, \geq \delta^{(N^{4N - 3n - 1} + 1)\epsilon}|B_{n}|\notag\\
&\label{a7} \,\,\, \geq \delta^{N^{4N - 3n}\epsilon}|B_{n}| = \delta^{\tau_{n}}|B_{n}|.\end{align}
We used the elementary inequality $N^{4N - 3n - 1} + 1 \leq N^{4N - 3n}$, for $N \geq 2$ and $0 \leq n \leq N$.

A combination of \eqref{a7}, the tightness of $B_{c}$, and the inequality $|A + cB_{c}|_{\delta} < \delta^{-\epsilon}|A|$, shows that $B_{c} \subset \bar{B}_{c} \subset B_{n}$ satisfies all the requirements (1)-(3), and is therefore a competitor in the definition of $H_{n + 1}$. In particular, now we have shown that such competitors exist for all $c \in C_{n + 1}$. Moreover, since $2^{k} \leq N$ (as in \eqref{a8}), it follows that
\begin{equation}\label{a15} |H_{n} + c \cdot 2^{k}B_{c}|_{\delta} \lesssim |H_{n} + c \cdot NB_{c}|_{\delta} \lesssim |H_{n + 1}| \stackrel{\eqref{a4}}{\leq} N\delta^{-1/N} |H_{n}|. \end{equation}
With this bound in mind, we next plan to show that $|H_{n} + cB_{n}|_{\delta}$ (the left hand side of \eqref{aa5}) is controlled by $|H_{n} + c \cdot 2^{k}B_{c}|_{\delta}$. To achieve this, we apply Lemma \ref{OVLemma} to the sets $H_{n},2^{k}B_{n} \subset \delta \cdot \Z$, and the subset $G = H_{n} \times 2^{k}B_{c} \subset H_{n} \times 2^{k}B_{n}$ which satisfies $|G| = |H_{n}||2^{k}B_{n}|/C_{3}$ with constant $C_{3} = |2^{k}B_{n}|/|2^{k}B_{c}|$: 
\begin{align} |H_{n} + c B_{n}|_{\delta} & \lesssim |H_{n} + c \cdot 2^{k}B_{n}|_{\delta} \notag \\
&\label{a9} \lesssim \frac{|H_{n} + H_{n}|}{|H_{n}|} \cdot \frac{|2^{k + 1}B_{n}|}{|2^{k}B_{n}|} \cdot \frac{|2^{k}B_{n}|}{|2^{k}B_{c}|} \cdot |H_{n} + c \cdot 2^{k}B_{c}|_{\delta}.  \end{align} 
Apart from \eqref{a15}, the individual factors are bounded from above as follows:
\begin{itemize}
\item $|H_{n} + H_{n}|/|H_{n}| \leq 2N\delta^{-1/N}$ by \eqref{a4},
\item $|2^{k + 1}B_{n}|/|2^{k}B_{n}| \leq (2^{N + 1}\delta^{-1})^{1/\log_{2} N}$ by \eqref{a8},
\item $|2^{k}B_{n}|/|2^{k}B_{c}| \leq 2\delta^{-1/N}$ by the $(\tau_{n},2\delta^{-1/N},N)$-tightness of $B_{n}$, and by \eqref{a7}.
\end{itemize}
Plugging these estimates into \eqref{a9} yields
\begin{equation}\label{form23} |H_{n} + cB_{n}|_{\delta} \lesssim_{N} \delta^{-3/N - 1/\log_{2} N }|H_{n}| \stackrel{\eqref{aa23}}{\leq} \delta^{-\bar{\epsilon}/2}|H_{n}|, \qquad c \in C_{n + 1}. \end{equation}
This completes the proof of \eqref{aa5}, if $\delta > 0$ is small enough depending on $\bar{\epsilon},N$, both of which only depend on $\alpha,\beta,\gamma,\kappa$.

We next plan to use \eqref{aa5} to contradict Theorem \ref{mainTechnicalApp} with parameters $\bar{\alpha},\bar{\beta},\bar{\gamma},\bar{\kappa}$, and the objects $H_{n},B_{n}$, and $\bar{\nu} = \nu(C_{n + 1})^{-1} \cdot \nu|_{C_{n + 1}}$. The first task it to use the Pl\"unnecke-Ruzsa inequality, Lemma \ref{PRIneq}, to show that
\begin{equation}\label{a19} |H_{n}| \leq \delta^{-\bar{\alpha}}, \end{equation}
assuming that $\delta > 0$ is sufficiently small in terms of $N,\bar{\alpha}$. Indeed, note that $H_{n}$ can be written as a sum of $\leq N^{n} \leq N^{N}$ sets of the form $(c_{m}B_{m})_{\delta}$, for some $1 \leq m \leq n$. Each of these sets individually satisfies $|A + c_{m}B_{m}|_{\delta} < \delta^{-\epsilon}|A|$. We may therefore infer that
\begin{displaymath} |H_{n}| \lesssim_{N} \delta^{-2N^{N}\epsilon}|A| \leq \delta^{-2N^{N}\epsilon - \alpha}. \end{displaymath}
from Lemma \ref{PRIneq}. This inequality implies \eqref{a19} for small enough $\delta > 0$, recalling our choice of constants at \eqref{a14}.

Recall from \eqref{form17} that the set $B$ satisfies a Frostman condition with exponent $\kappa$:
\begin{displaymath} |B \cap B(x,r)| \leq r^{\kappa}|B|, \qquad x \in \R, \, \delta \leq r \leq \delta^{\epsilon_{0}}. \end{displaymath}
Since $B_{n} \subset B$, and $|B_{n}| \geq \delta^{N^{4N + 1}\epsilon}|B|$ by \eqref{a13} we deduce that $B_{n}$ satisfies a Frostman condition with parameters $\bar{\epsilon}_{0} = \epsilon_{0}$ (recall \eqref{aa22}) and $\bar{\kappa}$:
\begin{equation}\label{a24} |B_{n} \cap B(x,r)| \leq \delta^{-N^{4N + 1}\epsilon}r^{\kappa}|B_{n}| \stackrel{\eqref{a14}}{\leq} r^{\bar{\kappa}}|B_{n}|, \qquad x \in \R, \, \delta \leq r \leq \delta^{\bar{\epsilon}_{0}}. \end{equation}
Moreover, since $|B| \geq \delta^{-\beta}$ by assumption (see above \eqref{form17}), we have
\begin{equation}\label{a25} |B_{n}| \geq \delta^{N^{4N + 1}\epsilon}|B| \stackrel{\eqref{a14}}{\geq} \delta^{-\bar{\beta}}. \end{equation}
Finally, we verify that the probability measure $\bar{\nu} = \nu(C_{n + 1})^{-1} \cdot \nu|_{C_{n + 1}}$ satisfies a Frostman condition with exponent $\bar{\gamma}$. Since $\nu$ itself satisfies the Frostman condition $\nu(B(x,r)) \leq r^{\gamma}$ for all $\delta \leq r \leq \delta^{\epsilon_{0}} = \delta^{\bar{\epsilon}_{0}}$, and $\nu(C_{n + 1}) \geq \delta^{\epsilon}$, we see that
\begin{displaymath} \bar{\nu}(B(x,r)) \leq \delta^{-\epsilon}\nu(B(x,r)) \leq \delta^{-\epsilon}r^{\gamma} \stackrel{\eqref{a14}}{\leq} r^{\bar{\gamma}}, \qquad x \in \R, \, \delta \leq r \leq \delta^{\bar{\epsilon}_{0}}. \end{displaymath} 
We have now reached a situation which violates Theorem \ref{mainTechnicalApp} for the choice of parameters $\bar{\alpha},\bar{\beta},\kappa/2,\bar{\gamma}$. The objects $H_{n},B_{n},\bar{\nu}$ satisfy all the hypotheses by \eqref{a19}-\eqref{a25}, but nevertheless $|H_{n} + cB_{n}|_{\delta} < \delta^{-\bar{\epsilon}}|H_{n}|$ for all $c \in C_{n + 1}$ according to \eqref{aa5}, where $C_{n + 1}$ is a set of full $\bar{\nu}$ measure. Therefore the \textbf{Counter assumption} is false, and the proof of Theorem \ref{mainSubset2} is complete.

To be precise, we have cut one corner: $H_{n}$ may not be a subset of $[0,1]$: we only know that $H_{n} \subset [0,N^{n}] \subset [0,N^{N}]$. However, one can easily fix this by picking the most $H_{n}$-populous unit interval $[r,r + 1] \subset [0,N^{n}]$, which contains $\geq N^{-N}|H_{n}| \geq \delta^{\bar{\epsilon}/4}|H_{n}|$ points of $H_{n}$ if $\delta > 0$ is small enough, and replacing $H_{n}$ by $\bar{H}_{n} := H_{n} \cap [r,r + 1] - \{r\} \subset [0,1]$. After replacing $H_{n}$ with $\bar{H}_{n}$, the estimate \eqref{form23} remains valid with constant $\delta^{-3\bar{\epsilon}/4}$ instead of $\delta^{-\bar{\epsilon}/2}$. This is still good enough to imply \eqref{aa5}. \end{proof}

\begin{remark}\label{rem5} In order to deduce Theorem \ref{mainSubset2} for a fixed exponent "$\kappa$" from Theorem \ref{mainTechnicalApp}, the argument above only needed to apply Theorem \ref{mainTechnicalApp} with exponent $\bar{\kappa} < \kappa$ arbitrarily close to $\kappa$ (but $\epsilon \to 0$ in Theorem \ref{mainSubset2} as $\bar{\kappa} \nearrow \kappa$).

\end{remark}

\subsection{Proof of the main theorem}\label{s:mainProof}

In this section, we finally prove Theorem \ref{main} by reducing it to its toy version, Theorem \ref{mainSubset1}. We will need the asymmetric Balog-Szemer\'edi-Gowers theorem, see the book of Tao and Vu, \cite[Theorem 2.35]{MR2289012}. We state the result in the following slightly weaker form (following Shmerkin's paper \cite[Theorem 3.2]{Sh}):
\begin{thm}[Asymmetric Balog-Szemer\'edi-Gowers theorem]\label{BSG} Given $\zeta > 0$, there exists $\xi > 0$ such that the following holds for $\delta \in 2^{-\N}$ small enough. Let $A,B \subset (\delta \cdot \Z) \cap [0,1]$ be finite sets, and assume that there exist $c \in [\tfrac{1}{2},1]$ and $G \subset A \times B$ satisfying 
\begin{equation}\label{bsg1} |G| \geq \delta^{\xi}|A||B| \quad \text{and} \quad |\{x + cy : (x,y) \in G\}|_{\delta} = |\pi_{c}(G)|_{\delta} \leq \delta^{-\xi}|A|. \end{equation}
Then there exist subsets $A' \subset A$ and $B' \subset B$ with the properties
\begin{equation}\label{bsg2} |A'||B'| \geq \delta^{\zeta}|A||B| \quad \text{and} \quad |A' + cB'|_{\delta} \leq \delta^{-\zeta}|A|. \end{equation}
\end{thm}

\begin{remark} In the references for Theorem \ref{BSG} cited above, the assumption $|\pi_{c}(G)|_{\delta} \leq \delta^{-\xi}|A|$ in \eqref{bsg1} is replaced by $|\pi_{1}(G)| \leq \delta^{-\xi}|A|$, and the conclusion \eqref{bsg2} is replaced by $|A' + B'| \leq \delta^{-\zeta}|A|$. For $c \in [\tfrac{1}{2},1]$, it is easy to see that the two variants of the theorem are formally equivalent. The details are left to the reader. The idea is to begin by applying the standard version of Theorem \ref{BSG} to the sets $B_{c} := (cB)_{\delta} \subset \delta \cdot \Z$ and $G_{c} := \{(x,(cy)_{\delta}) : (x,y) \in G\} \subset A \times B_{c}$, which satisfy $|B_{c}| \sim |B|$, $|G_{c}| \sim |G|$, and $|\pi_{1}(G_{c})| \lesssim \delta^{-\xi}|A|$. \end{remark}

\begin{proof}[Proof of Theorem \ref{main} assuming Theorem \ref{mainSubset1}] Let $\alpha,\beta,\gamma,\kappa$ be the constants for which we are supposed to prove Theorem \ref{main}. Thus $\gamma > (\alpha - \beta)/(1 - \beta)$. Our task is to find the constants $\epsilon,\epsilon_{0},\delta_{0} \in (0,\tfrac{1}{2}]$ such that the conclusion of Theorem \ref{main} holds. To this end, pick $\bar{\kappa} \in (0,\kappa)$ arbitrarily, and $\bar{\alpha} > \alpha$, $\bar{\beta} < \beta$, and $\bar{\gamma} < \gamma$ in such a way that the key inequality
\begin{displaymath} \bar{\gamma} > (\bar{\alpha} - \bar{\beta})/(1 - \bar{\beta}) \end{displaymath}
persists. This can be done explicitly in such a way that $\bar{\alpha},\bar{\beta},\bar{\gamma}$ are functions of $\alpha,\beta,\gamma$: therefore, any future dependence on $\bar{\alpha},\bar{\beta},\bar{\gamma}$ will, in fact, be a dependence on $\alpha,\beta,\gamma$.

Let $\bar{\epsilon},\bar{\epsilon}_{0},\bar{\delta}_{0} \in (0,\tfrac{1}{2}]$ be the constants given by Theorem \ref{mainSubset1} applied with parameters $\bar{\alpha},\bar{\beta},\bar{\gamma},\bar{\kappa}$. We now define $\epsilon,\epsilon_{0},\delta_{0}$ based on $\bar{\epsilon},\bar{\epsilon}_{0},\bar{\delta}_{0}$. First, we set $\epsilon_{0} := \bar{\epsilon}_{0}$. We also fix $\delta_{0} \in (0,\bar{\delta}_{0}]$. There will be a few additional requirements on $\delta_{0}$, depending on $\alpha,\beta,\gamma,\kappa$ only. These will be clarified when they arise. We then finally determine the constant $\epsilon$. First, we fix a natural number $N \sim 1/\bar{\epsilon}$, sufficiently large that the following holds:
\begin{equation}\label{b23} (N - 1)^{-1} < \bar{\epsilon}/2. \end{equation}
Then, we fix the auxiliary constant
\begin{equation}\label{bb9} \zeta := \min\left\{\frac{\bar{\epsilon}}{20 N},\frac{\bar{\epsilon}_{0}(\kappa - \bar{\kappa})}{2N}, \frac{\bar{\alpha} - \alpha}{2N(N + 1)},\frac{\beta - \bar{\beta}}{2N},\frac{\epsilon_{0}(\gamma - \bar{\gamma})}{2N} \right\}. \end{equation}
Now, let $\epsilon := \xi(\zeta) > 0$ be the constant given by Theorem \ref{BSG} applied with the constant $\zeta > 0$ from \eqref{bb9}. This means that if $c \in [\tfrac{1}{2},1]$, and $G \subset A \times B$ satisfies $|G| \geq \delta^{\epsilon}|A||B|$ and $|\pi_{c}(G)|_{\delta} \leq \delta^{-\epsilon}|A|$, then there exist $A' \subset A$ and $B' \subset B$ as in \eqref{bsg2}. 

Armed with these choices of parameters, we are prepared to prove Theorem \ref{main}. Fix $\delta \in 2^{-\N}$ with $\delta \leq \delta_{0}$, and let $A,B,\nu$ be a triple satisfying the hypotheses of Theorem \ref{main} with constants $\alpha,\beta,\gamma,\kappa$. To recap once more, $|A| \leq \delta^{\alpha}$, and $|B| \geq \delta^{-\beta}$, and 
\begin{equation}\label{form10aa} |B \cap B(x,r)| \leq r^{\kappa}|B|, \qquad x \in \R, \, \delta \leq r \leq \delta^{\epsilon_{0}} = \delta^{\bar{\epsilon}_{0}}. \end{equation}
Also, $\nu$ is a probability measure on $[\tfrac{1}{2},1]$ satisfying $\nu(B(x,r)) \leq r^{\gamma}$ for all $\delta \leq r \leq \delta^{\epsilon_{0}}$. We claim that there exists $c \in C := \spt(\nu)$ such that whenever $G \subset A \times B$ is a subset with $|G| \geq \delta^{\epsilon}|A||B|$, then $|\pi_{c}(G)|_{\delta} \geq \delta^{-\epsilon}|A|$. 

We make a counter assumption: the property above fails for every $c \in C$. Then, by the choice $\epsilon = \xi(\zeta)$, and Theorem \ref{BSG}, for every $c \in C$ there exist subsets $A_{c} \subset A$ and $B_{c} \subset B$, for every $c \in C$, with the properties 
\begin{equation}\label{bb5} |A_{c} \times B_{c}| \geq \delta^{\zeta}|A||B| \quad \text{and} \quad |A_{c} + cB_{c}|_{\delta} \leq \delta^{-\zeta}|A|. \end{equation}
We observe that
\begin{displaymath} \int \ldots \int |(A_{c_{1}} \times B_{c_{1}}) \cap \ldots \cap (A_{c_{N}} \times B_{c_{N}})| \, d\nu(c_{1})\cdots d\nu(c_{N}) \geq \delta^{N \zeta}|A||B| \end{displaymath}
by H\"older's inequality. Using $(A \times B) \cap (C \times D) = (A \cap C) \times (B \cap D)$, and Chebyshev's inequality, and $\nu(\R) = 1$, it follows that the set
\begin{equation}\label{def:omega} \Omega := \{(c_{1},\ldots,c_{N}) \in C^{N} : |(A_{c_{1}} \cap \ldots \cap A_{c_{N}}) \times (B_{c_{1}} \cap \ldots \cap B_{c_{N}})| \geq \tfrac{1}{2}\delta^{N\zeta}|A||B|\} \end{equation}
satisfies
\begin{equation}\label{form12a} \nu^{N}(\Omega) \geq \tfrac{1}{2} \cdot \delta^{N\zeta} \end{equation}
For $c_{1},\ldots,c_{n} \in C$ fixed, we define
\begin{displaymath} \Omega_{c_{1}\cdots c_{n}} := \{(c_{n + 1},\ldots,c_{N}) \in C^{N - n} : (c_{1},\ldots,c_{N}) \in \Omega\}. \end{displaymath}
It follows easily from Fubini's theorem that
\begin{equation}\label{bb6} \nu^{N - n}(\Omega_{c_{1}\cdots c_{n}}) = \int \nu^{N - n - 1}(\Omega_{c_{1}\cdots c_{n}c}) \, d\nu(c) \end{equation}
for all $c_{1},\ldots,c_{n} \in C$, and $1 \leq n \leq N - 2$. The same remains true for $n = 0$, if the left hand side is interpreted as $\nu^{N}(\Omega)$, and $c_{1}\cdots c_{n}c = c$. Equation \eqref{bb6} also remains valid for $n = N - 1$ if we define the notation $\nu^{N - n - 1} = \nu^{0}$ as follows:
\begin{equation}\label{bb7} \nu^{0}(\Omega_{c_{1}\cdots c_{N - 1}c}) := \mathbf{1}_{\Omega}(c_{1},\ldots,c_{N - 1},c). \end{equation}
We will use this notation in the sequel. 

For $(c_{1},\ldots,c_{N}) \in C^{N}$ fixed, we define decreasing sequences of sets $\{A_{c_{1}\cdots c_{n}}\}_{n = 1}^{N}$ and $\{B_{c_{1}\cdots c_{n}}\}_{n = 1}^{N}$ as follows:
\begin{displaymath} A_{c_{1}\cdots c_{n}} := A_{c_{1}} \cap \ldots \cap A_{c_{n}} \quad \text{and} \quad B_{c_{1}\cdots c_{n}} := B_{c_{1}} \cap \ldots \cap B_{c_{n}}, \quad 1 \leq n \leq N. \end{displaymath}
The definition formally makes sense for $(c_{1},\ldots,c_{N}) \in C^{N}$, but will only be useful for $(c_{1},\ldots,c_{N}) \in \Omega$. Namely, if $(c_{1},\ldots,c_{N}) \in \Omega$, then it follows from the definition \eqref{def:omega} that
\begin{equation}\label{b4} |A_{c_{1}\cdots c_{n}}| \geq |A_{c_{1}\cdots c_{N}}| \geq \tfrac{1}{2} \cdot \delta^{N\zeta}|A| \quad \text{and} \quad |B_{c_{1}\cdots c_{n}}| \geq \tfrac{1}{2} \cdot \delta^{N\zeta}|B|. \end{equation}
We now construct the sets $\{H_{n}\}_{n = 1}^{N} \subset \delta \cdot \Z$. At the same time, we will construct subsets $C_{1},\ldots,C_{N} \subset C$, and points $c_{n} \in C_{n}$, $1 \leq n \leq N$, with the properties  
\begin{equation}\label{bb3} \nu^{N - n}(\Omega_{c_{1}\cdots c_{n}}) \geq 2^{-n - 1}\delta^{N\zeta} \quad \text{and} \quad \nu(C_{n}) \geq 2^{-n - 1}\delta^{N\zeta}, \quad 1 \leq n \leq N. \end{equation}
In particular, the first part of \eqref{bb3} with $n = N$ shows that $(c_{1},\ldots,c_{N}) \in \Omega$, recall the notation \eqref{bb7}. To begin with, we define
\begin{displaymath} C_{1} := \{c \in C : \nu^{N - 1}(\Omega_{c}) \geq 2^{-2} \delta^{N\zeta}\}, \end{displaymath}
and we choose an arbitrary element $c_{1} \in C_{1}$. Since
\begin{displaymath} \int \nu^{N - 1}(\Omega_{c}) \, d\nu(c) = \nu^{N}(\Omega) \geq 2^{-1}\delta^{N\zeta} \end{displaymath}
by \eqref{form12a}, and the case $n = 0$ of \eqref{bb6}, we observe that $\nu(C_{1}) \geq 2^{-2}\delta^{N\zeta}$ by Chebyshev's inequality. In particular $C_{1} \neq \emptyset$. We then define
\begin{displaymath} H_{1} := (c_{1}B_{c_{1}})_{\delta}. \end{displaymath}

Assume inductively that $H_{1},\ldots,H_{n}$ and $C_{1},\ldots,C_{n} \subset C$, and $c_{j} \in C_{j}$, $1 \leq j \leq n \leq N - 1$, have already been constructed, and satisfy \eqref{bb3}. We then pick an element $c_{n + 1} \in C_{n + 1}$, where
\begin{displaymath} C_{n + 1} := \{c \in C : \nu^{N - n - 1}(\Omega_{c_{1}\cdots c_{n} c}) \geq 2^{-n - 2}\delta^{N\zeta}\}, \quad 1 \leq n \leq N - 1. \end{displaymath}
For $n = N - 1$, the notation $\nu^{N - n - 1}(\Omega_{c_{1}\cdots c_{n}c})$ should be interpreted as in \eqref{bb7}, so 
\begin{displaymath} C_{N} = \{c \in C : \mathbf{1}_{\Omega}(c_{1},\ldots,c_{N - 1},c) \geq 2^{-N - 1}\delta^{N\zeta}\} = \{c \in C : (c_{1},\ldots,c_{N - 1},c) \in \Omega\}. \end{displaymath}
For an arbitrary choice $c_{n + 1} \in C_{n + 1}$, we note that the first part of \eqref{bb3} is satisfied with index "$n + 1$", simply by the definition of $C_{n + 1}$. 

The set $C_{n + 1}$ also satisfies the second part of \eqref{bb3} with index "$n + 1$", by
\begin{displaymath} 2^{-n - 1}\delta^{N\zeta} \stackrel{\eqref{bb3}}{\leq} \nu^{N - n}(\Omega_{c_{1}\cdots c_{n}}) \stackrel{\eqref{bb6}}{=} \int \nu^{N - n - 1}(\Omega_{c_{1}\cdots c_{n}c}) \, d\nu(c), \end{displaymath}
and Chebyshev's inequality.

Whereas $c_{1} \in C_{1}$ was chosen arbitrarily, the element $c_{n + 1} \in C_{n + 1}$ is chosen in such a way that the quantity $|H_{n} + c_{n + 1}B_{c_{1}\cdots c_{n + 1}}|_{\delta}$ is maximised, among all possible choices $c_{n + 1} \in C_{n + 1}$. For this choice of $c_{n + 1} \in C_{n + 1}$, we define
\begin{displaymath} H_{n + 1} := H_{n} + (c_{n + 1}B_{c_{1}\cdots c_{n + 1}})_{\delta}. \end{displaymath}
Proceeding in this manner yields a sequence of sets $H_{1},\ldots,H_{N}$, and a distinguished sequence $(c_{1},\ldots,c_{N}) \in \Omega$, which we fix for the remainder of the argument. We record that if $(c_{1},\cdots,c_{n})$, $1 \leq n \leq N - 1$, is an initial sequence of $(c_{1},\cdots,c_{N})$, then
\begin{equation}\label{form13} |B_{c_{1}\cdots c_{n}c}| \geq \tfrac{1}{2}\delta^{N\zeta} |B_{c_{1}\cdots c_{n}}| \geq \delta^{\bar{\epsilon}}|B_{c_{1}\cdots c_{n}}|, \qquad c \in C_{n + 1}. \end{equation}
The second inequality simply follows from our choice of $\zeta$ at \eqref{bb9}. To see the first inequality, recall from the definition of $c \in C_{n + 1}$ that (in particular) $\Omega_{c_{1}\cdots c_{n}c} \neq \emptyset$ (in the case $n = N - 1$ simply $(c_{1},\ldots,c_{n},c) \in \Omega$). This means that there exists a sequence $(c_{n + 2}',\ldots,c_{N}') \in C^{N - n - 1}$ such that $(c_{1},\ldots c_{n},c,c_{n + 2}',\ldots,c_{N}') \in \Omega$. Consequently,
\begin{displaymath} |B_{c_{1}\cdots c_{n}c}| \geq |B_{c_{1}} \cap \cdots B_{c_{n}} \cap B_{c} \cap B_{c_{n + 2}'} \cap \cdots B_{c_{N}'}| \geq \tfrac{1}{2}\delta^{N\zeta}|B| \geq \tfrac{1}{2}\delta^{N\zeta}|B_{c_{1}\cdots c_{n}}| \end{displaymath}
by the definition of $\Omega$, see \eqref{def:omega}.

Note that $H_{n} \subset (\delta \cdot \Z) \cap [0,n]$ for all $1 \leq n \leq N$ by a straightforward induction, so $|H_{n}| \leq 2N \delta^{-1}$. Therefore, by the pigeonhole principle, there exists an $n \in \{1,\ldots,N - 1\}$ such that
\begin{equation}\label{form22} |H_{n + 1}| \leq (2N\delta^{-1})^{1/(N - 1)}|H_{n}| \leq 4\delta^{-1/(N - 1)}|H_{n}|. \end{equation}
We now consider the objects 
\begin{equation}\label{form9} \bar{A} := H_{n}, \quad \bar{B} := B_{c_{1}\cdots c_{n}}, \quad \text{and} \quad \bar{\nu} := \nu(C_{n + 1})^{-1}\nu|_{C_{n + 1}}. \end{equation}
We will show in a moment these objects satisfy the hypotheses of Theorem \ref{mainSubset1} with constants $\bar{\alpha},\bar{\beta},\bar{\kappa},\bar{\gamma}$, and $\bar{\epsilon}_{0}$. First, however, we conclude the proof of Theorem \ref{main}, taking this for granted. By Theorem \ref{mainSubset1}, there exists $\bar{c} \in C_{n + 1}$ (a set of full $\bar{\nu}$ measure) such that whenever $B' \subset \bar{B}$ is a set of cardinality $|B'| \geq \delta^{\bar{\epsilon}}|B|$, we have 
\begin{equation}\label{form15a} |H_{n} + \bar{c}B'|_{\delta} = |\bar{A} + \bar{c}B'|_{\delta} \geq \delta^{-\bar{\epsilon}}|\bar{A}| = \delta^{-\bar{\epsilon}}|H_{n}|. \end{equation}
(To be accurate, Theorem \ref{mainSubset1} only claims this for some $\bar{c} \in \spt(\bar{\nu})$, but the proof showed, see \eqref{form25}, that actually the set of non-admissible $c \in \spt(\bar{\nu})$ have measure strictly smaller than $1$, so we can pick $c \in C_{n + 1}$.) However, for every $c \in C_{n + 1}$, the set $B' := B_{c_{1}\cdots c_{n}c} \subset B_{c_{1}\cdots c_{n}} = \bar{B}$ satisfies
\begin{equation}\label{form16} |B'| \stackrel{\eqref{form13}}{\geq} \delta^{\bar{\epsilon}}|\bar{B}| \quad \text{and} \quad |H_{n} + cB'|_{\delta} \lesssim |H_{n + 1}| \stackrel{\eqref{form22}}{\leq} 4\delta^{-1/(N - 1)}|H_{n}| \stackrel{\eqref{b23}}{\leq} \delta^{-\bar{\epsilon}/2}|H_{n}|. \end{equation}
The inequality $|H_{n} + cB'|_{\delta} \lesssim |H_{n + 1}|$ follows from the fact that whenever $c \in C_{n + 1}$, the set $H_{n} + (c B')_{\delta} = H_{n} + (c B_{c_{1}\cdots c_{n}c})_{\delta}$ is a competitor in the definition of $H_{n + 1}$. With the choice $c = \bar{c} \in C_{n + 1}$, the inequalities \eqref{form15a}-\eqref{form16} are mutually incompatible for $\delta > 0$ small enough, depending on $\bar{\epsilon} = \bar{\epsilon}(\alpha,\beta,\gamma,\kappa) > 0$. A contradiction has been reached.

It remains to check that that the objects in \eqref{form22} satisfy the hypotheses of Theorem \ref{mainSubset1} with constants $\bar{\alpha},\bar{\beta},\kappa/2,\bar{\gamma}$, and $\bar{\epsilon}_{0}$. More precisely:
\begin{itemize}
\item[(a)] $|\bar{A}| \leq \delta^{-\bar{\alpha}}$,
\item[(b)] $|\bar{B}| \geq \delta^{-\bar{\beta}}$, and $\bar{B}$ satisfies a Frostman condition with exponent $\bar{\kappa}$, for $r \in [\delta,\delta^{\bar{\epsilon}_{0}}]$,
\item[(c)] $\bar{\nu}$ satisfies a Frostman condition with exponent $\bar{\gamma}$.
\end{itemize}
We first use the Pl\"unnecke-Ruzsa inequality to establish (a), assuming that $\delta > 0$ is sufficiently small in terms of $N,\bar{\alpha}$. It is clear by induction that $H_{n}$ can be written as a sum of $n \leq N$ sets of the form $(c_{m}B_{c_{1}\cdots c_{m}})_{\delta}$, for some $1 \leq m \leq n$. Noting that $A_{c_{1}\cdots c_{n}} \subset A_{c_{m}}$, each of these sets individually satisfies
\begin{displaymath} |A_{c_{1}\cdots c_{n}} + (c_{m}B_{c_{1}\cdots c_{m}})_{\delta}| \lesssim |A_{c_{m}} + c_{m}B_{c_{m}}|_{\delta} \stackrel{\eqref{bb5}}{\leq} \delta^{-\zeta}|A| \stackrel{\eqref{b4}}{\leq} 2\delta^{-(N + 1)\zeta}|A_{c_{1}\cdots c_{n}}|. \end{displaymath}
We may therefore infer that
\begin{displaymath} |H_{n}| \lesssim_{N} \delta^{-N(N + 1)\zeta}|A| \leq \delta^{-N(N + 1)\zeta - \alpha}. \end{displaymath}
from the Pl\"unnecke-Ruzsa inequality, Lemma \ref{PRIneq}, applied with $A_{c_{1}\cdots c_{n}}$ in place of $A$ (and finally also using $|A_{c_{1}\cdots c_{n}}| \leq |A| \leq \delta^{-\alpha}$, see above \eqref{form10aa}). This inequality implies $|H_{n}| \leq \delta^{-\bar{\alpha}}$ for small enough $\delta > 0$, recalling our choice of $\zeta$ at \eqref{bb9}.

We move to (b). Recall from \eqref{form10aa} that the set $B$ satisfies the assumptions of Theorem \ref{main} with constants $\epsilon_{0},\kappa > 0$:
\begin{displaymath} |B \cap B(x,r)| \leq r^{\kappa}|B|, \qquad x \in \R, \, \delta \leq r \leq \delta^{\epsilon_{0}} = \delta^{\bar{\epsilon}_{0}}. \end{displaymath}
Since $B_{c_{1}\cdots c_{n}} \subset B$, and $|B_{c_{1}\cdots c_{n}}| \geq \tfrac{1}{2}\delta^{N\zeta}|B|$ by \eqref{b4}, we deduce that $B_{c_{1}\cdots c_{n}}$ satisfies a Frostman condition with exponent $\bar{\kappa}$:
\begin{displaymath} |B_{c_{1}\cdots c_{n}} \cap B(x,r)| \leq 2\delta^{-N\zeta}r^{\kappa}|B_{c_{1}\cdots c_{n}}| \leq r^{\bar{\kappa}}|B_{c_{1}\cdots c_{n}}|, \qquad x \in \R, \, \delta \leq r \leq \delta^{\bar{\epsilon}_{0}}. \end{displaymath}
The final inequality uses our choice of $\zeta$ in \eqref{bb9}, and also assumes that $\delta > 0$ is sufficiently small, depending on $\bar{\epsilon}_{0},\kappa$. Moreover, since $|B| \geq \delta^{-\beta}$ by assumption, we have
\begin{displaymath} |B_{c_{1}\cdots c_{n}}| \geq \tfrac{1}{2}\delta^{N\zeta}|B| \stackrel{\eqref{bb9}}{\geq} \delta^{-\bar{\beta}}. \end{displaymath}
Let us finally check (c), namely that the probability measure $\bar{\nu} = \nu(C_{n + 1})^{-1} \cdot \nu|_{C_{n + 1}}$ satisfies a Frostman condition with exponent $\bar{\gamma}$. Indeed, recalling from \eqref{bb3} that $\nu(C_{n + 1}) \geq 2^{-n - 2}\delta^{N\zeta}$, we have
\begin{displaymath} \bar{\nu}(B(x,r)) \leq 2^{n + 2}\delta^{-N\zeta}\nu(B(x,r)) \leq 2^{N + 2}\delta^{-N\zeta} \cdot r^{\gamma}, \qquad x \in \R, \, \delta \leq r \leq \delta^{\bar{\epsilon}_{0}}. \end{displaymath}
Since $r^{\gamma} \leq \delta^{\epsilon_{0}(\gamma - \bar{\gamma})}r^{\bar{\gamma}}$ for $r \leq \delta^{\epsilon_{0}}$, by our choice of $\zeta$ in \eqref{bb9}, the right hand side is bounded from above by $r^{\bar{\gamma}}$ for all $\delta > 0$ small enough, depending on $N,\gamma,\bar{\gamma}$ (all of which only depend on $\alpha,\beta,\gamma,\kappa$). We have now verified that the objects $\bar{A},\bar{B},\bar{\nu}$ from \eqref{form9} satisfy the hypotheses of Theorem \ref{mainSubset1}. This concludes the proof of Theorem \ref{main}. \end{proof}

\begin{remark}\label{rem6} Once again, in order to deduce Theorem \ref{main} for a fixed exponent "$\kappa$" from Theorem \ref{mainSubset1}, we only needed to apply Theorem \ref{mainSubset1} with a fixed exponent $\bar{\kappa} \in (0,\kappa)$, as close to $\kappa$ as we desire. Combining this with the previous similar Remarks \ref{rem4}-\ref{rem5}, we obtain the conclusion alluded to in Remark \ref{rem7}: to deduce Theorem \ref{main} for a fixed exponent "$\kappa$" from Theorem \ref{mainTechnical}, we only needed to apply Theorem \ref{mainTechnical} for $\bar{\kappa} \in (0,\kappa)$ arbitrarily close to $\kappa$.

\end{remark}

\subsection{Proof of Corollary \ref{hausdorffCor}}\label{appA} I close the paper by recording the (standard pigeonholing) proof of Corollary \ref{hausdorffCor}, whose statement is recalled here:
\begin{cor} Let $0 < \beta \leq \alpha < 1$ and $\kappa > 0$. Then, there exists $\eta = \eta(\alpha,\beta,\kappa) > 0$ such that if $A,B \subset \R$ are Borel sets with $\Hd A = \alpha$, $\Hd B = \beta$, then 
\begin{displaymath} \Hd \{c \in \R : \Hd (A + cB) \leq \alpha + \eta\} \leq \tfrac{\alpha - \beta}{1 - \beta} + \kappa. \end{displaymath}
\end{cor}

\begin{proof} It is easy to reduce to the case where $A,B$ are compact, $A,B \subset [0,1]$, and $\mathcal{H}^{\alpha}(A) > 0$ and $\mathcal{H}^{\beta}(B) > 0$. In this case, one may use Frostman's lemma \cite[Theorem 8.8]{zbMATH01249699} to find Borel probability measures $\mu_{A},\mu_{B}$ with $\spt(\mu_{A}) \subset A$, $\spt (\mu_{B}) \subset B$, and satisfying $\mu_{A}(B(x,r)) \leq C_{A}r^{\alpha}$ and $\mu_{B}(B(x,r)) \leq C_{B}r^{\beta}$ for all balls $B(x,r) \subset \R$. If $\eta > 0$ is small enough, we will show that $\Hd E \leq (\alpha - \beta)/(1 - \beta) + \kappa$, where
\begin{displaymath} E := E_{\eta} := \{c \in [\tfrac{1}{2},1] : \Hd (A + cB) < \alpha + \eta\}. \end{displaymath} 
It is easy to show (by rescaling considerations) that this implies Corollary \ref{hausdorffCor}, where $[\tfrac{1}{2},1]$ is replaced by $\R$. It is well-known that the set $E \subset [\tfrac{1}{2},1]$ is Borel. Consequently, if the inequality fails, one may use Frostman's lemma again to find a Borel probability measure $\nu$, supported on $E$, satisfying $\nu(B(x,r)) \leq C_{\nu}r^{\gamma}$ for all $x \in \R$ and $r > 0$, where $\gamma \geq (\alpha - \beta)/(1 - \beta) + \kappa$.

For future reference, we fix some parameters $\bar{\alpha} > \alpha$, $\bar{\beta} < \beta$, and $\bar{\gamma} < \gamma$ such that the inequality
\begin{equation}\label{form111} \bar{\gamma} > (\bar{\alpha} - \bar{\beta})/(1 - \bar{\beta}) \end{equation}
still holds. We then let $\bar{\epsilon},\bar{\epsilon}_{0},\bar{\delta}_{0} > 0$ be the constants provided by Theorem \ref{main} applied with parameters $\bar{\alpha},\bar{\beta},\kappa = \bar{\beta},\bar{\gamma}$. We pick $\eta > 0$ in the definition of $E$ so small that
\begin{equation}\label{form113} \eta < \min\{\bar{\epsilon},\bar{\alpha} - \alpha\}. \end{equation}

Fix $c \in \spt(\nu) \subset E$, so $\Hd (A + cB) < \alpha + \eta$. This means that for a given fixed threshold $\delta_{0} := 2^{-j_{0}} \in 2^{-\N}$ (the requirements will depend on $\alpha,\beta,\gamma,C_{A},C_{B},C_{\nu}$), one may find a countable cover $\mathcal{I}_{c}$ of $A + cB$, consisting of disjoint dyadic intervals of length $\ell(I) \leq \delta_{0}$, such that
\begin{equation}\label{form100} \sum_{I \in \mathcal{I}_{c}} \ell(I)^{\alpha + \eta} \leq 1. \end{equation} 
Below, we will often write that something holds "for small enough $\delta > 0$": this will always mean "assuming that the upper bound $\delta_{0}$ for $\delta$ has been chosen sufficiently small, depending on the parameters $\alpha,\beta,\gamma,C_{A},C_{B},C_{\nu}$. In particular, we will take $\delta_{0} \leq \bar{\delta}_{0}$.

The "tubes" $\mathcal{T}_{c} := \{\pi_{c}^{-1}(I)\}_{I \in \mathcal{I}_{c}}$ cover $A \times B \supset \spt(\mu_{A} \times \mu_{B})$, so
\begin{displaymath} \int_{E} \sum_{T \in \mathcal{T}_{c}} (\mu_{A} \times \mu_{B})(T) \, d\nu(c) = 1. \end{displaymath}
Recall that $\delta_{0} = 2^{-j_{0}}$, and let $\mathcal{I}_{c}^{j} := \{I \in \mathcal{I}_{c} : \ell(I) = 2^{-j}\}$ for $j \geq j_{0}$. Write also $\mathcal{T}^{j}_{c} := \{\pi_{c}^{-1}(I)\}_{I \in \mathcal{I}_{c}^{j}}$. Since $\mathcal{T}_{c} = \bigcup_{j \geq j_{0}} \mathcal{T}_{c}^{j}$, there exists $j \geq j_{0}$ such that
\begin{displaymath} \int_{E} \sum_{T \in \mathcal{T}_{c}^{j}} (\mu_{A} \times \mu_{B})(T) \, d\nu(c) \gtrsim j^{-2}. \end{displaymath}
Write $\delta := 2^{-j}$ for this index $j$. According to the estimate above, there exists a subset $E_{\delta}' \subset E$ of measure $\nu(E_{\delta}') \gtrsim j^{-2} = \log_{2}(1/\delta)^{-2}$ such that for each $c \in E_{\delta}'$, the tubes $T \in \mathcal{T}_{c}^{j}$ cover a subset $G_{c} \subset \spt (\mu_{A} \times \mu_{B})$ of measure $(\mu_{A} \times \mu_{B})(G_{c}) \gtrsim \log_{2}(1/\delta)^{-2}$. In particular, we record that
\begin{equation}\label{form99} |\pi_{c}(G_{c})|_{\delta} \leq |\mathcal{T}_{c}^{j}| \leq \delta^{-\alpha - \eta}, \qquad c \in E_{\delta}', \end{equation}
by \eqref{form100}. For the remainder of this argument, we use the notation $f \lessapprox g$ to abbreviate an inequality of the form $f \leq C\log_{2}(1/\delta)^{C}g$ for some constant $C > 0$, which may depend on the Frostman constants $\alpha,\beta,\gamma,C_{A},C_{B},C_{\nu}$. In particular, $j^{-2} = \log_{2}(1/\delta)^{-2} \gtrapprox 1$.

For $x \in \R$, let $I_{\delta}(x) \in \mathcal{D}_{\delta}$ be the unique dyadic interval of length $\delta$ with $x \in I_{\delta}(x)$. We now split the set $A$ as follows:
\begin{displaymath} A = \bigcup_{\rho \in 2^{-\N}} A(\rho) := \{x \in A : \rho \leq \mu_{A}(I_{\delta}(x)) < 2\rho\}. \end{displaymath}
We define the sets $B(\rho) \subset B$ similarly. Since $\mu_{A}(I_{\delta}(x)) \leq C_{A}\delta^{\alpha}$ and $\mu_{B}(I_{\delta}(y)) \leq C_{B}\delta^{\beta}$, we see that $A(\rho) \neq \emptyset$ implies $\rho \leq C_{A}\delta^{\alpha}$, and $B(\rho) \neq \emptyset$ implies $\rho \leq C_{\beta}\delta^{\beta}$. We also note that $A(\rho)$ can be expressed as the intersection of $A$ with certain dyadic intervals $\mathcal{A}(\rho) \subset \mathcal{D}_{\delta}$. The same is true for $B(\rho)$, for certain dyadic intervals $\mathcal{B}(\rho) \subset \mathcal{D}_{\delta}$.

Let $\mu_{A}(\rho)$ be the restriction of $\mu_{A}$ to the intervals $\mathcal{A}(\rho)$, and similarly let $\mu_{B}(\rho)$ be the restriction of $\mu_{B}$ to the intervals in $\mathcal{B}(\rho)$. Then 
\begin{equation}\label{form109} \sum_{\rho_{1}} \sum_{\rho_{2}} \int_{E_{\delta}'} (\mu_{A}(\rho_{1}) \times \mu_{B}(\rho_{2}))(G_{c}) \approx 1, \end{equation}
so it follows from the pigeonhole principle that 
\begin{displaymath} \int_{E_{\delta}'} (\mu_{A}(\rho_{A}) \times \mu_{B}(\rho_{A}))(G_{c}) \approx 1 \end{displaymath}
for some fixed choices $\rho_{A} \leq C_{A}\delta^{\alpha}$ and $\rho_{B} \leq C_{B}\delta^{\beta}$ (noting that values $\rho_{1},\rho_{2} \leq \delta^{2}$ cannot contribute substantially to the sum in \eqref{form109}). In particular, there exists a further subset $E_{\delta} \subset E_{\delta}'$ with the property $(\mu_{A}(\rho_{A}) \times \mu_{B}(\rho_{B}))(G_{c}) \approx 1$ for all $c \in E_{\delta}$. We now abbreviate
\begin{displaymath} \bar{\mu}_{A} := \mu_{A}(\rho_{A}) \quad \text{and} \quad \bar{\mu}_{B} := \mu_{B}(\rho_{B}), \end{displaymath}
so $\|\bar{\mu}_{A}\| \approx 1 \approx \|\bar{\mu}_{B}\|$. The measure $\bar{\mu}_{A}$ is supported on the closure of the intervals in $\mathcal{A}(\rho_{A})$, and $\bar{\mu}_{B}$ is supported on the closure of the intervals in $\mathcal{B}(\rho_{B})$. Let 
\begin{displaymath} A_{\delta} := (\delta \cdot \Z) \cap \left(\cup \mathcal{A}(\rho_{A}) \right) \quad \text{and} \quad B_{\delta} := (\delta \cdot \Z) \cap \left(\cup \mathcal{B}(\rho_{B}) \right). \end{displaymath}
We observe that 
\begin{equation}\label{form110} \rho_{A} \cdot |A_{\delta}| \sim \|\mu_{A}\| \approx 1 \quad \Longrightarrow \quad \rho_{A} \approx |A_{\delta}|^{-1}, \end{equation}
and similarly $\rho_{B} \approx |B_{\delta}|^{-1}$. Since $\rho_{A} \leq C_{A}\delta^{\alpha}$, we record that
\begin{equation}\label{form112} |A_{\delta}| \approx \rho_{A}^{-1} \gtrapprox \delta^{-\alpha}. \end{equation}
We next claim that, somewhat conversely, $|A_{\delta}| \leq \delta^{-\bar{\alpha}}$ if $\delta > 0$ is sufficiently small. To see this, fix an arbitrary $c \in E_{\delta}$. Since $(\bar{\mu}_{A} \times \bar{\mu}_{B})(G_{c}) \approx 1$, there exists $b \in \spt (\bar{\mu}_{B})$ such that
\begin{displaymath} \bar{\mu}_{A}(G_{c}(b)) \approx 1, \quad \text{where} \quad  G_{c}(b) = \{x \in \spt(\bar{\mu}_{A}) : (x,b) \in G_{c}\}. \end{displaymath}
Now, if $\mathcal{G}_{c}(b) := \{I \in \mathcal{A}(\rho_{A}) : G_{c}(b) \cap I \neq \emptyset\}$, we see that $\bar{\mu}_{A}(I) \sim \rho_{A}$ for all $I \in \mathcal{G}_{c}(b)$, and $\bar{\mu}_{A}(\cup \mathcal{G}_{c}(b)) \geq \bar{\mu}_{A}(G_{c}(b)) \approx 1$. Moreover, we observe that $|G_{c}(b)|_{\delta} \lesssim |\pi_{c}(G_{c})|_{\delta}$, since $\pi_{c}(G_{c}) \supset G_{c}(b) + bc$. Putting these observations together,
\begin{equation}\label{form101} |A_{\delta}| \stackrel{\eqref{form110}}{\approx} \rho_{A}^{-1} \lessapprox \rho_{A}^{-1} \cdot \bar{\mu}_{A}(\cup \mathcal{G}_{c}(b)) \lesssim |G_{c}(b)|_{\delta} \lesssim |\pi_{c}(G_{c})|_{\delta} \stackrel{\eqref{form99}}{\leq} \delta^{-\alpha - \eta}. \end{equation}
Since $\alpha + \eta < \bar{\alpha}$ by \eqref{form113}, the inequality $|A_{\delta}| \leq \delta^{-\bar{\alpha}}$ holds for $\delta > 0$ sufficiently small.

Next, since $\rho_{B} \leq C_{B}\delta^{\beta}$, we record that 
\begin{equation}\label{form102} |B_{\delta}| \approx \rho_{B}^{-1} \gtrapprox \delta^{-\beta} \quad \Longrightarrow \quad |B_{\delta}| \geq \delta^{\bar{\beta}}, \end{equation}
where the implication holds if $\delta > 0$ is sufficiently small. Moreover, for $x \in \R$ and $r \geq \delta$, we note that every point $y \in B_{\delta} \cap B(x,r)$ is contained in an interval $I_{y}(\delta) \in \mathcal{B}(\rho_{B})$ with $\mu_{B}(I_{y}(\delta)) \geq \rho_{B}$. Since $I_{y}(\delta) \subset B(x,2r)$, we deduce that 
\begin{equation}\label{form103} |B_{\delta} \cap B(x,r)| \leq \rho_{B}^{-1} \cdot \mu_{B}(B(x,2r)) \leq \rho_{B}^{-1} \cdot C_{B}(2r)^{\beta} \lessapprox r^{\beta}|B_{\delta}|. \end{equation}
In particular, for the parameter $\bar{\epsilon}_{0} > 0$ fixed below \eqref{form111}, we have $|B_{\delta} \cap B(x,r)| \leq r^{\bar{\beta}}|B_{\delta}|$ for $\delta \leq r \leq \delta^{\bar{\epsilon}_{0}}$, provided that $\delta > 0$ is small enough.

Finally, the measure $\nu_{\delta} := \nu(E_{\delta})^{-1} \cdot \nu|_{E_{\delta}}$ satisfies
\begin{equation}\label{form104} \nu_{\delta}(B(x,r)) \lessapprox \nu(B(x,r)) \leq C_{\nu}r^{\gamma}, \qquad r > 0, \end{equation}
so the inequality $\nu_{\delta}(B(x,r)) \leq r^{-\bar{\gamma}}$ holds for all $r \leq \delta^{\bar{\epsilon}_{0}}$, provided that $\delta > 0$ is small enough. The estimates \eqref{form101}-\eqref{form104}, and \eqref{form111}, imply that the triple $A_{\delta},B_{\delta},\nu_{\delta}$ satisfies all the hypotheses of Theorem \ref{main} with constants $\bar{\alpha},\bar{\beta},\kappa = \bar{\beta},\bar{\gamma}$, and $\bar{\epsilon}_{0}$. Consequently, there exists $c \in E_{\delta} \subset E_{\delta}'$ (a set of full $\nu_{\delta}$ measure) such that
\begin{equation}\label{form114} |\pi_{c}(G)|_{\delta} \geq \delta^{-\bar{\epsilon}}|A_{\delta}| \stackrel{\eqref{form112}}{\gtrapprox} \delta^{-\alpha - \bar{\epsilon}} \end{equation}
for all subsets $G \subset A_{\delta} \times B_{\delta}$ of cardinality $|G| \geq \delta^{\bar{\epsilon}}|A||B|$. We argue that this contradicts \eqref{form99}. The only issue is that set $G_{c} \subset \spt(\mu_{A} \times \mu_{B})$ is not exactly a subset of $A_{\delta} \times B_{\delta}$. To fix this, recall that nevertheless $(\bar{\mu}_{A} \times \bar{\mu}_{B})(G_{c}) \approx 1$. Let 
\begin{displaymath} \mathcal{G}_{c} := \{I \times J \in \mathcal{A}(\rho_{A}) \times \mathcal{B}(\rho_{B}) : (I \times J) \cap G_{c} \neq \emptyset\}. \end{displaymath}
Then $\mathcal{G}_{c}$ is a cover of $G_{c}$, and $(\bar{\mu}_{A} \times \bar{\mu}_{B})(Q) \sim \rho_{A} \rho_{B} \approx |A_{\delta}|^{-1}|B_{\delta}|^{-1}$ for all $Q = I \times J \in \mathcal{G}_{c}$. Consequently,
\begin{displaymath} |\mathcal{G}_{c}| \gtrsim (\rho_{A}\rho_{B})^{-1} \cdot (\bar{\mu}_{A} \times \bar{\mu}_{B})(G_{c}) \approx |A_{\delta}||B_{\delta}|. \end{displaymath}
Now, let $G_{c,\delta} \subset (A_{\delta} \times B_{\delta}) \cap G_{c}(2\delta)$ be subset of cardinality $|G_{c,\delta}| \gtrapprox |A_{\delta}||B_{\delta}|$. In particular $|G_{c,\delta}| \geq \delta^{\bar{\epsilon}}|A_{\delta}||B_{\delta}|$ for $\delta > 0$ small enough. Therefore the estimate \eqref{form114} holds for $G = G_{c,\delta}$. On the other hand, since $G_{c,\delta} \subset G_{c}(2\delta)$, we have
\begin{displaymath} |\pi_{c}(G_{c,\delta})|_{\delta} \lesssim |\pi_{c}(G_{c})|_{\delta} \leq \delta^{-\alpha - \eta} \end{displaymath}
by \eqref{form99}. Since we chose $\eta < \bar{\epsilon}$ in \eqref{form113}, this estimate is not compatible with \eqref{form114}. A contradiction has been reached, and the proof of Corollary \ref{hausdorffCor} is complete. \end{proof}

\bibliographystyle{plain}
\bibliography{references}

\def\cprime{$'$}
\begin{thebibliography}{10}

\bibitem{MR3529116}
Yves Benoist and Nicolas de~Saxc\'{e}.
\newblock A spectral gap theorem in simple {L}ie groups.
\newblock {\em Invent. Math.}, 205(2):337--361, 2016.

\bibitem{Bo1}
J.~Bourgain.
\newblock On the {E}rd\"os-{V}olkmann and {K}atz-{T}ao ring conjectures.
\newblock {\em Geom. Funct. Anal.}, 13(2):334--365, 2003.

\bibitem{MR2966656}
J.~Bourgain and A.~Gamburd.
\newblock A spectral gap theorem in {${\rm SU}(d)$}.
\newblock {\em J. Eur. Math. Soc. (JEMS)}, 14(5):1455--1511, 2012.

\bibitem{MR2481734}
Jean Bourgain.
\newblock Multilinear exponential sums in prime fields under optimal entropy
  condition on the sources.
\newblock {\em Geom. Funct. Anal.}, 18(5):1477--1502, 2009.

\bibitem{Bourgain10}
Jean Bourgain.
\newblock The discretized sum-product and projection theorems.
\newblock {\em J. Anal. Math.}, 112:193--236, 2010.

\bibitem{MR2358056}
Jean Bourgain and Alex Gamburd.
\newblock On the spectral gap for finitely-generated subgroups of {$\rm
  SU(2)$}.
\newblock {\em Invent. Math.}, 171(1):83--121, 2008.

\bibitem{MR4452675}
Damian D{\k a}browski, Tuomas Orponen, and Michele Villa.
\newblock Integrability of orthogonal projections, and applications to
  {F}urstenberg sets.
\newblock {\em Adv. Math.}, 407:Paper No. 108567, 34, 2022.

\bibitem{MR820223}
P.~Erd\H{o}s and E.~Szemer\'{e}di.
\newblock On sums and products of integers.
\newblock In {\em Studies in pure mathematics}, pages 213--218. Birkh\"{a}user,
  Basel, 1983.

\bibitem{MR673510}
K.~J. Falconer.
\newblock Hausdorff dimension and the exceptional set of projections.
\newblock {\em Mathematika}, 29(1):109--115, 1982.

\bibitem{MR4447307}
Yuqiu Fu, Shengwen Gan, and Kevin Ren.
\newblock An incidence estimate and a {F}urstenberg type estimate for tubes in
  {$\Bbb R^2$}.
\newblock {\em J. Fourier Anal. Appl.}, 28(4):Paper No. 59, 28, 2022.

\bibitem{MR2344270}
M.~Z. Garaev.
\newblock An explicit sum-product estimate in {$\Bbb F_p$}.
\newblock {\em Int. Math. Res. Not. IMRN}, (11):Art. ID rnm035, 11, 2007.

\bibitem{MR2359478}
A.~A. Glibichuk and S.~V. Konyagin.
\newblock Additive properties of product sets in fields of prime order.
\newblock In {\em Additive combinatorics}, volume~43 of {\em CRM Proc. Lecture
  Notes}, pages 279--286. Amer. Math. Soc., Providence, RI, 2007.

\bibitem{MR4283564}
Larry Guth, Nets~Hawk Katz, and Joshua Zahl.
\newblock On the discretized sum-product problem.
\newblock {\em Int. Math. Res. Not. IMRN}, (13):9769--9785, 2021.

\bibitem{GSW19}
Larry Guth, Noam Solomon, and Hong Wang.
\newblock Incidence estimates for well spaced tubes.
\newblock {\em Geom. Funct. Anal.}, 29(6):1844--1863, 2019.

\bibitem{MR2484645}
Katalin Gyarmati, M\'{a}t\'{e} Matolcsi, and Imre~Z. Ruzsa.
\newblock Pl\"{u}nnecke's inequality for different summands.
\newblock In {\em Building bridges}, volume~19 of {\em Bolyai Soc. Math.
  Stud.}, pages 309--320. Springer, Berlin, 2008.

\bibitem{MR4041116}
Weikun He.
\newblock Discretized sum-product estimates in matrix algebras.
\newblock {\em J. Anal. Math.}, 139(2):637--676, 2019.

\bibitem{MR4148151}
Weikun He.
\newblock Orthogonal projections of discretized sets.
\newblock {\em J. Fractal Geom.}, 7(3):271--317, 2020.

\bibitem{MR4244524}
Weikun He and Nicolas de~Saxc\'{e}.
\newblock Sum-product for real {L}ie groups.
\newblock {\em J. Eur. Math. Soc. (JEMS)}, 23(6):2127--2151, 2021.

\bibitem{Ho}
Michael Hochman.
\newblock On self-similar sets with overlaps and inverse theorems for entropy.
\newblock {\em Ann. of Math. (2)}, 180(2):773--822, 2014.

\bibitem{Ka}
Robert Kaufman.
\newblock On {H}ausdorff dimension of projections.
\newblock {\em Mathematika}, 15:153--155, 1968.

\bibitem{MR4299175}
Jialun Li.
\newblock Discretized {S}um-product and {F}ourier decay in {$\Bbb R^n$}.
\newblock {\em J. Anal. Math.}, 143(2):763--800, 2021.

\bibitem{Mar}
J.~M. Marstrand.
\newblock Some fundamental geometrical properties of plane sets of fractional
  dimensions.
\newblock {\em Proc. London Math. Soc. (3)}, 4:257--302, 1954.

\bibitem{zbMATH01249699}
P.~{Mattila}.
\newblock {\em {Geometry of sets and measures in Euclidean spaces. Fractals and
  rectifiability. 1st paperback ed.}}
\newblock Cambridge: Cambridge University Press, 1st paperback ed. edition,
  1999.

\bibitem{MR4565644}
Ali Mohammadi and Sophie Stevens.
\newblock Attaining the exponent 5/4 for the sum-product problem in finite
  fields.
\newblock {\em Int. Math. Res. Not. IMRN}, (4):3516--3532, 2023.

\bibitem{MR3162243}
Daniel~M. Oberlin.
\newblock Some toy {F}urstenberg sets and projections of the four-corner
  {C}antor set.
\newblock {\em Proc. Amer. Math. Soc.}, 142(4):1209--1215, 2014.

\bibitem{MR3590535}
Tuomas Orponen.
\newblock On the distance sets of {A}hlfors-{D}avid regular sets.
\newblock {\em Adv. Math.}, 307:1029--1045, 2017.

\bibitem{MR4388762}
Tuomas Orponen.
\newblock On arithmetic sums of {A}hlfors-regular sets.
\newblock {\em Geom. Funct. Anal.}, 32(1):81--134, 2022.

\bibitem{2021arXiv210603338O}
Tuomas {Orponen} and Pablo {Shmerkin}.
\newblock {On the Hausdorff dimension of Furstenberg sets and orthogonal
  projections in the plane}.
\newblock {\em Duke Math. J. (to appear)}.

\bibitem{OV18}
Tuomas {Orponen} and Laura {Venieri}.
\newblock {A note on expansion in prime fields}.
\newblock {\em arXiv e-prints}, page arXiv:1801.09591, January 2018.

\bibitem{MR1749437}
Yuval Peres and Wilhelm Schlag.
\newblock Smoothness of projections, {B}ernoulli convolutions, and the
  dimension of exceptions.
\newblock {\em Duke Math. J.}, 102(2):193--251, 2000.

\bibitem{2021arXiv210807311R}
Orit~E. {Raz} and Joshua {Zahl}.
\newblock {On the dimension of exceptional parameters for nonlinear
  projections, and the discretized Elekes-R{\'o}nyai theorem}.
\newblock {\em Geom. Funct. Anal. (to appear)}.

\bibitem{MR4469270}
Misha Rudnev and Sophie Stevens.
\newblock An update on the sum-product problem.
\newblock {\em Math. Proc. Cambridge Philos. Soc.}, 173(2):411--430, 2022.

\bibitem{MR2314377}
Imre~Z. Ruzsa.
\newblock An application of graph theory to additive number theory.
\newblock {\em Sci. Ser. A Math. Sci. (N.S.)}, 3:97--109, 1989.

\bibitem{Sh}
Pablo Shmerkin.
\newblock On {F}urstenberg's intersection conjecture, self-similar measures,
  and the {$L^q$} norms of convolutions.
\newblock {\em Ann. of Math. (2)}, 189(2):319--391, 2019.

\bibitem{MR3940442}
Pablo Shmerkin.
\newblock On the {H}ausdorff dimension of pinned distance sets.
\newblock {\em Israel J. Math.}, 230(2):949--972, 2019.

\bibitem{Shmerkin20}
Pablo Shmerkin.
\newblock A nonlinear version of bourgain's projection theorem.
\newblock (J. Eur. Math. Soc. to appear), 2020.

\bibitem{2021arXiv211209044S}
Pablo {Shmerkin} and Hong {Wang}.
\newblock {On the distance sets spanned by sets of dimension $d/2$ in
  $\mathbb{R}^d$}.
\newblock {\em arXiv e-prints}, page arXiv:2112.09044, December 2021.

\bibitem{MR3742451}
Sophie Stevens and Frank de~Zeeuw.
\newblock An improved point-line incidence bound over arbitrary fields.
\newblock {\em Bull. Lond. Math. Soc.}, 49(5):842--858, 2017.

\bibitem{MR729791}
Endre Szemer\'{e}di and William~T. Trotter, Jr.
\newblock Extremal problems in discrete geometry.
\newblock {\em Combinatorica}, 3(3-4):381--392, 1983.

\bibitem{MR2289012}
Terence Tao and Van Vu.
\newblock {\em Additive combinatorics}, volume 105 of {\em Cambridge Studies in
  Advanced Mathematics}.
\newblock Cambridge University Press, Cambridge, 2006.

\end{thebibliography}

\end{document}